\documentclass[10pt,reqno]{amsart} 

\usepackage{float} 

\usepackage{soul}

\usepackage{hhline}

\usepackage{epsfig,amssymb,amsmath,version}
\usepackage{amssymb,version,graphicx,fancybox,mathrsfs}

\usepackage{enumitem}
\usepackage{bbm}
\usepackage{url,hyperref}
\usepackage{subfigure}
\usepackage{color}
\usepackage{stmaryrd}
\usepackage{multirow}
\usepackage{booktabs,siunitx}
\usepackage{booktabs}
\usepackage{multicol,epstopdf}
\usepackage{xypic}
\usepackage[]{graphicx}
\usepackage[all]{xy}
\usepackage{tikz,ifthen}
\usetikzlibrary{positioning, math, decorations.markings, arrows.meta}
\usetikzlibrary{calc,shapes,cd}
\usetikzlibrary{knots} 
\usepgfmodule{decorations}
\allowdisplaybreaks
\tikzset{knotarrow/.pic={ \draw[edge, <-] (0,0) -- +(-.001,0);}}

\tikzset{edge/.style={line width=0.8}}
\tikzset{wall/.style={very thick}}

\tikzset{->-/.style n args={2}{decoration={markings, mark=at position #1 with {\arrow{#2}}}, postaction={decorate}}} 

\ExplSyntaxOn
\cs_new_eq:NN \ifstreqF \str_if_eq:nnF
\cs_new_eq:NN \ifstreqTF \str_if_eq:nnTF
\ExplSyntaxOff

\tikzset{-o-/.code 2 args={\ifstreqF{#2}{} 
{\ifstreqTF{#2}{>}
   {\pgfkeysalso{decoration={markings,mark=at position #1 with {\arrow[scale=0.8]{#2}}}
                    ,postaction={decorate}}
    }
   {\ifstreqTF{#2}{<}
       {\pgfkeysalso{decoration={markings,mark=at position #1 with {\arrow[scale=0.8]{#2}}}
                    ,postaction={decorate}}
        }
       {\pgfkeysalso{decoration={markings,
                    mark=at position #1 with
                    {\draw[black, fill={#2}] circle[radius=2pt];}}
                    ,postaction={decorate}}
        }
     }
  }}}

\allowdisplaybreaks[4]

\newtheorem{theorem}{Theorem}[section]
\newtheorem{lemma}[theorem]{Lemma}
\newtheorem{definition}[theorem]{Definition}
\newtheorem{corollary}[theorem]{Corollary}
\newtheorem{proposition}[theorem]{Proposition}
\newtheorem{remark}[theorem]{Remark}

\newcommand{\Rlabel}[1]{\hyperref[R#1]{(R#1)}}

\newcommand{\bp}{\begin{proposition}}
\newcommand{\ep}{\end{proposition}}
\newcommand{\bpr}{\begin{proof}}
\newcommand{\epr}{\end{proof}}

\newcommand{\bt}{\begin{theorem}}
\newcommand{\et}{\end{theorem}}

\newcommand{\bl}{\begin{lemma}}
\newcommand{\el}{\end{lemma}}

\newcommand{\bcr}{\begin{corollary}}
\newcommand{\ecr}{\end{corollary}}

\newcommand{\be}{\begin{equation}}
\newcommand{\ee}{\end{equation}}

\newcommand{\bes}{\begin{equation*}}
\newcommand{\ees}{\end{equation*}}

\newcommand{\ba}{\begin{align}}
\newcommand{\ea}{\end{align}}

\newcommand{\bas}{\begin{align*}}
\newcommand{\eas}{\end{align*}}

\DeclareMathOperator{\im}{\mathrm{Im}}

\DeclareMathOperator{\sgn}{\mathrm{sgn}}

\graphicspath{{./Figures/} }

\newcommand{\red}[1]{{\color{red}#1}}

\begin{document}
\bibliographystyle{alpha}

\title{Centers and representations of ${\rm SL}_n$ quantum Teichm{\"u}ller spaces}

\author[Zhihao Wang]{Zhihao Wang}
\address{Zhihao Wang, School of Mathematics, Korea Institute for Advanced Study (KIAS), 85 Hoegi-ro, Dongdaemun-gu, Seoul 02455, Republic of Korea}
\email{zhihaowang@kias.re.kr}


 \maketitle

\begin{abstract}

In this paper, we compute the center of the balanced Fock–Goncharov algebra and determine its rank over the center when the quantum parameter is a root of unity. These results have potential applications to the study of the center and rank of the ${\rm SL}_n$-skein algebra. Building on this computation, we classify the irreducible representations of the balanced Fock–Goncharov algebra. 

Due to the Frobenius homomorphism, every irreducible representation of the (projected) ${\rm SL}_n$-skein algebra of a punctured surface $\mathfrak{S}$ determines a point in the ${\rm SL}_n$ character variety of $\mathfrak{S}$, known as the classical shadow of the representation. By pulling back the irreducible representations of the balanced Fock–Goncharov algebra via the quantum trace map, we show that there exists a “large” subset of the ${\rm SL}_n$ character variety such that, for any point in this subset, there exists an irreducible representation of the (projected) ${\rm SL}_n$-skein algebra whose classical shadow is this point. Finally, we prove that, under mild conditions, the representations of the ${\rm SL}_n$-skein algebra obtained in this way are independent of the choice of ideal triangulation.
 
\end{abstract}

\tableofcontents

\newcommand{\ca}{{\cev{a}  }}

\def\CT{\mathcal T}
\def\BZ{\mathbb Z}
\def\Id{\mathrm{Id}}
\def\Mat{\mathrm{Mat}}
\def\BN{\mathbb N}
\def\BC{\mathbb C}
\def\BCS{\mathbb{C}^{*}}

\def \cb {\color{blue}}
\def\red {\color{red}}
\def \cbf {\color{blue}\bf}
\definecolor{ligreen}{rgb}{0.0, 0.3, 0.0}
\def \cg {\color{ligreen}}
\def \cgf {\color{ligreen}\bf}
\definecolor{darkblue}{rgb}{0.0, 0.0, 0.55}
\def \dbf {\color{darkblue}\bf}
\definecolor{anti-flashwhite}{rgb}{0.55, 0.57, 0.68}
\def \afw {\color{anti-flashwhite}}
\def\cF{\mathcal F}
\def\cP{\mathcal P}
\def\embed{\hookrightarrow}
\def\pr{\mathrm{pr}}
\def\cV{\VV}
\def\ot{\otimes}
\def\buu{{\mathbf u}}
\def\HA{\text{Hom}_{\text{Alg}}}


\def \ri {{\rm i}}
\newcommand{\bs}[1]{\boldsymbol{#1}}
\newcommand{\cev}[1]{\reflectbox{\ensuremath{\vec{\reflectbox{\ensuremath{#1}}}}}}
\def\bS{\bar \fS}
\def\cE{\mathcal E}
\def\fB{\mathfrak B}
\def\cY{\mathcal Y}
\def\cS{\mathscr S}
\def\rS{\overline{\cS}_{\hat\omega}}

\def\fS{\mathfrak{S}}
\def\OfS{\overline{\mathfrak{S}}}
\def\fST{\mathfrak{S}\times (-1,1)}

\def\Xsle{\mathfrak X_{{\rm SL}_n(\mathbb C)}^{(d)}(\fS)}

\def\MN {(M)}
\def\cN {\mathcal{N}}
\def\SL{{\rm SL}_n}

\def\bP{\mathbb P}
\def\bR{\mathbb R}

\def\SS{\cS_{\hat\omega}(\fS)}
\def\Sso{\cS_{\bar\omega}(\fS)}
\def\Sse{\cS_{\bar\eta}(\fS)}
\def\Ssc{\cS_{1}(\fS)}
\def\rdS{\overline \cS_{\hat\omega}^{\rm st}(\fS)}

\newcommand{\beq}{\begin{equation}}
	\newcommand{\eeq}{\end{equation}}

\section{Introduction}
Throughout this paper, all algebras and vector spaces are assumed to be defined over the complex field $\mathbb C$, unless explicitly stated otherwise.  
We denote by $\mathbb C^{*}$ the multiplicative group $\mathbb C\setminus\{0\}$.

\subsection{Review of the ${\rm SL}_n$-skein algebra and the quantum trace map}

A \emph{punctured surface} is obtained from an oriented closed surface with finitely many connected components by removing a finite set of points, called \emph{punctures}. 

Let $\hat\omega^{\frac{1}{2}}$ be a nonzero complex number. Define
\[
\bar\omega = \hat\omega^{\,n}, 
\qquad 
\omega = \hat\omega^{\,n^2},
\]
so that $\bar\omega^{\frac{1}{2n}} = \hat\omega^{\frac{1}{2}}$ and $\omega^{\frac{1}{2n^2}} = \hat\omega^{\frac{1}{2}}$.

The ${\rm SL}_n$-skein algebra of a punctured surface, denoted $\cS_{\bar\omega}(\fS)$, was introduced in \cite{Sik05} as a quantization of the coordinate ring, denoted as $\mathcal{O}_{{\rm SL}_n(\mathbb{C})}(\fS)$, of the ${\rm SL}_n$ character variety of $\fS$:
\begin{align}\label{intro-Sl3-character}
\mathfrak{X}_{{\rm SL}_n(\mathbb{C})}(\fS) 
= \text{Hom}(\pi_1(\fS),{\rm SL}_n(\mathbb C)) \big/ \simeq,
\end{align}
where $\chi\simeq\chi'$ if and only if 
$\operatorname{tr}(\chi(x))=\operatorname{tr}(\chi'(x))$ for all $x\in\pi_1(\fS)$.

In our setting, the $\SL$-skein algebra is a $\BC$-algebra. It is defined as the quotient of the $\BC$-vector space freely generated by the isotopy classes of $n$-webs in $\fS\times (-1,1)$, modulo the relations \eqref{w.cross}–\eqref{wzh.four}. Here an $n$-web in $\fS\times (-1,1)$ is a disjoint union of framed oriented knots and embedded framed oriented graphs such that each graph contains only $n$-valent sinks and sources (see \S\ref{sec-sub-def-stated-skein}). For two $n$-webs $\alpha,\beta$, their product $\alpha\beta$ is defined by stacking $\alpha$ over $\beta$. It is well known that the ${\rm SL}_2$-skein algebra is isomorphic to the Kauffman bracket skein algebra \cite{LS21}.

\medskip

When $\fS$ is triangulable with a chosen triangulation $\lambda$ (see \S\ref{subsec-quantum-trace-map}; note that we exclude self-folded triangles), there exists an algebra homomorphism \cite{LY23}
\begin{align}\label{intro-quaantum-trace-map}
      {\rm tr}_\lambda\colon \cS_{\bar\omega}(\fS)\;\longrightarrow\; \mathcal Z_{\hat\omega}(\fS,\lambda).
\end{align}
Here $\mathcal Z_{\hat\omega}(\fS,\lambda)$ is a quantum torus defined as follows (see \S\ref{subsec-quantum-trace-map} for details). Cutting $\fS$ along the ideal arcs of $\lambda$ produces a collection of ideal triangles $\mathbb F_\lambda$. There is a natural projection
\[
{\bf pr}_{\lambda} \colon \bigsqcup_{\tau\in\mathbb F_\lambda} \tau \;\longrightarrow\; \fS.
\]
For each $\tau\in\mathbb F_\lambda$, as illustrated in the right panel of Figure~\ref{Fig;coord_ijk}, we associate a quiver $\Gamma_\tau$ whose vertices are called \emph{small vertices}. Gluing these quivers via ${\bf pr}_\lambda$ identifies small vertices lying on edges $b$ and $b'$ whenever $b$ and $b'$ are identified, and cancels any pair of oppositely oriented arrows between the same two vertices. The resulting quiver in $\fS$ is denoted $\Gamma_\lambda$, with vertex set $V_\lambda$. Each arrow in $\Gamma_\lambda$ is assigned weight $1$. Let
\[
Q_\lambda\colon V_\lambda\times V_\lambda \to \frac{1}{2}\BZ \qquad \text{(see \eqref{eq-sign-matrix-ad})}
\]
be the signed adjacency matrix of the weighted quiver $\Gamma_\lambda$.
The \emph{Fock–Goncharov algebra} is defined by \cite{FG09b,FG09a}
\begin{align}\label{intro-X}
\mathcal{X}_{\omega}(\fS,\lambda) 
= \BC \langle
X_v^{\pm 1}, v \in V_\lambda \rangle \big/ (
X_v X_{v'}= \omega^{\, 2 Q_\lambda(v,v')} X_{v'} X_v \;\; \text{for } v,v'\in V_\lambda ).
\end{align}
The \emph{$n$-th root Fock–Goncharov algebra} is defined by
\[
\mathcal{Z}_{\hat\omega}(\fS,\lambda)
=  \BC \langle
Z_v^{\pm 1}, v \in V_\lambda \rangle \big/ (
Z_v Z_{v'}= \hat\omega^{\, 2 Q_\lambda(v,v')} Z_{v'} Z_v \;\; \text{for } v,v'\in V_\lambda ).
\]
Since $\omega= \hat\omega^{n^2}$, there is an algebra embedding
\begin{align}\label{intro-embedding}
    \mathcal{X}_{\omega}(\fS,\lambda) \hookrightarrow \mathcal{Z}_{\hat\omega}(\fS,\lambda),
\qquad
X_v \mapsto Z_v^n.
\end{align}

\medskip

It is shown in \cite{LY23} that the image of ${\rm tr}_\lambda$ is contained in a subalgebra
$\mathcal{Z}_{\hat\omega}^{\rm bl}(\fS,\lambda)$ of $\mathcal{Z}_{\hat\omega}(\fS,\lambda)$. This subalgebra will be the main object of study in this paper.  

For each $\tau\in\mathbb F_\lambda$, as illustrated in the left panel of Figure~\ref{Fig;coord_ijk}, every small vertex in $\tau$ has a barycentric coordinate $ijk$ with $0\leq i,j,k\leq n-1$ and $i+j+k=n$. Define functions
\[
{\bf k}_{1}(ijk)=i,\quad {\bf k}_{2}(ijk)=j,\quad {\bf k}_3(ijk)=k,
\]
on the vertex set $V_\tau$ of $\Gamma_\tau$. Let $\mathcal B_{\tau}\subset\BZ^{V_{\tau}}$ be the subgroup generated by ${\bf k}_1,{\bf k}_2,{\bf k}_3$ together with $(n\BZ)^{V_{\tau}}$. Elements of $\mathcal{B}_{\tau}$ are called \emph{balanced}.  
A vector ${\bf k}\in\BZ^{V_\lambda}$ is \emph{balanced} if its restriction
\[
{\bf k}_\tau=(k_v)_{v\in V_\tau},\qquad k_v = {\bf k}({\bf pr}_\lambda(v)),
\]
lies in $\mathcal B_\tau$ for each $\tau\in\mathbb F_\lambda$. Let $\mathcal B_\lambda$ denote the subgroup of $\BZ^{V_\lambda}$ generated by all such balanced vectors.  
The \emph{balanced Fock–Goncharov algebra} is then the monomial subalgebra \cite{LY23}
\[
\mathcal{Z}_{\hat\omega}^{\rm bl}(\fS,\lambda)
=\operatorname{span}_{\BC}\{Z^{\bf k}\mid {\bf k}\in\mathcal B_\lambda\}
\;\subset\;
\mathcal{Z}_{\hat\omega}(\fS,\lambda),
\]
where $Z^{\bf k}$ is Weyl-ordered monomial defined in \eqref{def-mmmmmmm}.


\def\PT{\Phi^{\mathbb T}}

\subsection{The center and the rank of the balanced Fock--Goncharov algebra at roots of unity}

Let $\fS$ be a triangulable punctured surface with a triangulation $\lambda$.  
For $n=2,3$, it is known that the center and the rank of $\mathcal{Z}_{\hat\omega}^{\rm bl}(\fS,\lambda)$ serves as a powerful tool for understanding the center and the rank of $\cS_{\bar\omega}(\fS)$ \cite{unicity,frohman2021dimension,kim2024unicity}.
For $n>3$, however, a ``good basis'' for $\cS_{\bar\omega} (\fS)$ is not yet available. Once such a basis is constructed, one can compute the center and the rank of $\cS_{\bar\omega} (\fS)$ via the quantum trace map \eqref{intro-quaantum-trace-map} and the center and the rank of $\mathcal{Z}_{\hat\omega}^{\rm bl}(\fS,\lambda)$.

\begin{figure}[h]
    \centering
    \includegraphics[width=150pt]{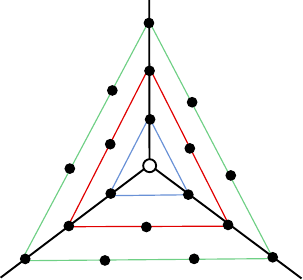}
    \caption{The central puncture is $p$. The blue line is at distance~$1$ from $p$, the red line at distance~$2$, and the green line at distance~$3$.}
    \label{Fig;distance}
\end{figure}

We begin with central elements in $\mathcal{Z}_{\hat\omega}^{\rm bl}(\fS,\lambda)$ arising from punctures.  
For each puncture $p$ and $1 \leq i \leq n-1$, define ${\bf a}(\lambda,p,i) = (a_v)_{v \in V_\lambda}$ by
\[
   a_v = \text{the number of times the vertex $v$ appears on the line at distance $i$ from $p$},
\]
where the line at distance $i$ from $p$ is illustrated in Figure~\ref{Fig;distance}.
A formal definition of ${\bf a}(\lambda,p,i)$ is given in \S\ref{sec-sub-center-bal-generic}.  
Note that ${\bf a}(\lambda,p,i)$ is not balanced. We therefore set
\begin{align*}
     {\bf b}(\lambda,p,i)
     := \left(\sum_{1\leq t\leq n-i-1} n(t+i-n)\,{\bf a}(\lambda,p,t) \right) 
     + (n-i)\sum_{1\leq t\leq n-1} (n-t)\,{\bf a}(\lambda,p,t).
\end{align*}
It is shown in \S\ref{sec-sub-center-bal-generic} that 
$${\bf b}(\lambda,p,i)\in\mathcal B_\lambda.$$
Moreover, when $\omega$ is not a root of unity, Corollary~\ref{cor-center-generic} gives
$$\mathsf{Z}(\mathcal Z_{\hat\omega}^{\rm bl}(\fS,\lambda))
= \mathbb C\big[Z^{{\bf b}(\lambda,p,i)} \;\mid\; 1\leq i\leq n-1,\; p\in\mathcal P\big],$$
where $\mathsf{Z}(\mathcal Z_{\hat\omega}^{\rm bl}(\fS,\lambda))$ denotes the center and $\mathcal P$ is the set of punctures of $\fS$.

\medskip

When $\hat\omega$ is a root of unity, the center 
$\mathsf{Z}(\mathcal Z_{\hat\omega}^{\rm bl}(\fS,\lambda))$ becomes significantly larger.  
We now introduce two versions of root of unity assumptions, which will be used frequently in the sequel.

\begin{itemize}
    \item[\Rlabel{1}]:  
    $\hat\omega^{\frac{1}{2}}\in\mathbb C$ is a root of unity such that the order of $\hat\omega^2$ is $N''$.  
    Define $\bar\omega = \hat\omega^n$ and $\omega = \hat\omega^{n^2}$, so that 
    $\bar\omega^{\frac{1}{2n}} = \hat\omega^{\frac{1}{2}}$ and 
    $\omega^{\frac{1}{2n^2}} = \hat\omega^{\frac{1}{2}}$.  
    Then $\bar\omega^2$ has order $N' = \frac{N''}{\gcd(N'',n)}$, and $\omega^2$ has order $N=\frac{N'}{d}$, where $d=\gcd(N',n)$.  
    Set $\hat \eta^{\frac{1}{2}} = (\hat\omega^{\frac{1}{2}})^{N^2}$.  
    Then $\bar \eta = \hat\eta^n= \bar\omega^{N^2}$ and
    $\eta=\hat \eta^{n^2}=\omega^{N^2}$, with $\bar\eta^d = \pm 1$. 
    \label{R1}
    
    \item[\Rlabel{2}]:  
    In addition to \Rlabel{1}, assume $d=1$.  
    \label{R2}
\end{itemize}

\medskip

Under assumption \Rlabel{1}, there is an algebra embedding
\begin{align}\label{intro-eq-def-F}
    \Phi^{\mathbb T}\colon \mathcal{Z}_{\hat\eta}^{\rm bl}(\fS,\lambda)\hookrightarrow
    \mathcal{Z}^{\rm bl}_{\hat\omega}(\fS,\lambda), 
    \qquad
    Z^{\bf k}\mapsto Z^{N{\bf k}} \quad \text{for ${\bf k}\in\mathcal B_\lambda$.}
\end{align}
To describe the central elements arising from $\im \PT$, we introduce the notation
\begin{align*}
    \mathcal Z_{\hat\eta}^{\rm bl}(\fS,\lambda)_d   :=\operatorname{span}_{\BC}\big\{Z^{\bf k} \;\big|\; {\bf k}\in \mathcal B_\lambda\text{ s.t. }
    {\bf k} Q_\lambda {\bf t}^T=0 \;\;\text{in}\;\; \mathbb Z_{dn}\;\;\text{for all }{\bf t}\in\mathcal B_\lambda\big\},
\end{align*}
which is obviously a subalgebra of $\mathcal Z_{\hat\eta}^{\rm bl}(\fS,\lambda)$.

We now describe the center $\mathsf{Z}(\mathcal Z_{\hat\omega}^{\rm bl}(\fS,\lambda))$ at roots of unity, together with the rank of $\mathcal Z_{\hat\omega}^{\rm bl}(\fS,\lambda)$ over its center.  
The following result is the main theorem of this paper.  

\begin{theorem}[Theorems~\ref{thm-center-balanced-root-of-unity} and  \ref{thm-rank-Z}]\label{intro-1}
Assume \Rlabel{1}.  
Let $\fS$ be a triangulable connected punctured surface with triangulation $\lambda$, and let $\PT$ be the algebra embedding defined in \eqref{intro-eq-def-F}.  
Then:
\begin{enumerate}[label={\rm (\alph*)}]
    \item The center of $\mathcal{Z}^{\rm bl}_{\hat\omega}(\fS,\lambda)$ is generated by 
    \[
        \Phi^{\mathbb T}\big(\mathcal Z_{\hat\eta}^{\rm bl}(\fS,\lambda)_d\big)
    \]
    together with the elements 
    \[
        Z^{{\bf b}(\lambda,p,i)}, \qquad 1\leq i\leq n-1,\; p\in\mathcal P,
    \]
    where $\mathcal P$ is the set of punctures of $\fS$.

    \item If $\fS$ has genus $g$ and $m$ punctures, then $\mathcal Z_{\hat\omega}^{\rm bl}(\fS,\lambda)$ is a free module over its center of rank
    \[
        d^{2g}\, N^{2(n^2-1)(g-1) + n(n-1)m}.
    \]
\end{enumerate}
\end{theorem}

\vspace{0.2cm}

Here we briefly state the techniques used to prove Theorem~\ref{intro-1}.

\vspace{0.2cm}

For each $e\in\lambda$, let $\tau$ and $\tau'$ be the two ideal triangles adjacent to $e$.
Denote by $e_1,e_2$ (resp. $e_3,e_4$) the two other edges of $\tau$ (resp. $\tau'$), noting that repetitions may occur among $e_1,e_2,e_3,e_4$.
By cutting along $e_1,e_2,e_3,e_4$, we isolate a quadrilateral $\mathbb P_{4,e}$ having $e$ as its diagonal.
Although $\mathbb P_{4,e}$ may not be globally embedded in $\fS$, its interior is embedded.
A crucial technical step in proving Theorem~\ref{intro-1}(a) is the analysis of central elements in the interior of each quadrilateral $\mathbb P_{4,e}$.
This local analysis, combined with further arguments around punctures, leads to the full description of the center.

For Theorem~\ref{intro-1}(b), we rely on the following bilinear form.  
From Lemma~\ref{lem-balanced-property}, there exists an anti-symmetric $\mathbb Z$-bilinear form
\begin{align*}
    \langle\cdot,\cdot\rangle\colon \mathcal B_\lambda\times \mathcal B_\lambda\rightarrow\mathbb Z,
    \qquad 
    \langle {\bf k}, {\bf t}\rangle = \tfrac{1}{n}{\bf k}Q_\lambda{\bf t}^T.
\end{align*}
We construct a surjective group homomorphism (Lemma~\ref{lem-surj-zeta-def})
\[
    \zeta\colon \mathcal B_\lambda\longrightarrow H_1(\overline{\fS},\mathbb{Z}_n),
\]
where $\overline{\fS}$ is obtained from $\fS$ by filling in all punctures.  
We then show (Proposition~\ref{prop-bilinear-key}) that
\begin{align}\label{intro-inter-number}
    \langle{\bf k},{\bf h}\rangle
    =\langle\zeta({\bf k}),\zeta({\bf h})\rangle_n \in\mathbb Z_n,
\end{align}
where $\langle\zeta({\bf k}),\zeta({\bf h})\rangle_n$ is the algebraic intersection number modulo $n$.  
Combining this relation with direct computations, we establish Theorem~\ref{intro-1}(b).  
Equation~\eqref{intro-inter-number} is the most technical and crucial step of the proof, and it also has further applications (see Remark~\ref{rem-d-trace}).

An application of Theorem~\ref{intro-1}(b) yields the following corollary.  

\begin{corollary}[Corollary~\ref{lam-decom}]\label{intro-cor-1}
Suppose that $\fS$ is a triangulable connected punctured surface of genus $g$ with $m$ punctures, and let $\lambda$ be a triangulation of $\fS$.  
Then there exists a basis $\{{\bf h}_i\mid 1\leq i\leq |V_\lambda|\}$ of $\mathcal B_\lambda$ such that the matrix of the bilinear form is the following:
\[
    \big(\langle {\bf h}_i,{\bf h}_j\rangle\big)_{1\leq i,j\leq |V_\lambda|}
    = \text{diag}\!\left\{
    \begin{pmatrix}
        0 & s_1\\
        - s_1 & 0
    \end{pmatrix},
    \;\dots,\;
    \begin{pmatrix}
        0 & s_r\\
        - s_r & 0
    \end{pmatrix},\;0,\dots,0
    \right\},
\]
where $r=(n^2-1)(g-1) + \tfrac{1}{2}n(n-1)m$ and 
\[
    s_i=\begin{cases}
        1 & 1\leq i\leq g,\\[0.3em]
        n & g< i\leq r.
    \end{cases}
\]
\end{corollary}

For $n=2$, Corollary~\ref{intro-cor-1} was established in \cite{representation2,BL07} by different methods.  
For general $n$, the structure of the center and the rank of $\mathcal Z_{\hat\omega}^{\rm bl}(\fS,\lambda)$ were studied in \cite{KaruoWangToAppear} in the case $\partial\fS\neq \emptyset$ (the ``stated case") and when $\fS$ has no interior punctures.  

\subsection{Representations of the balanced Fock--Goncharov algebra}

It was shown in~\cite{KaruoWangToAppear} that, when $\hat\omega$ is a root of unity, every irreducible representation of $\mathcal{Z}^{\rm bl}_{\hat\omega}(\fS,\lambda)$ is uniquely determined by its action of the center $\mathsf{Z}(\mathcal{Z}^{\rm bl}_{\hat\omega}(\fS,\lambda))$, and that the dimension of each irreducible representation equals the square root of the rank of $\mathcal{Z}^{\rm bl}_{\hat\omega}(\fS,\lambda)$ over its center.  
Combining this result with Theorem~\ref{intro-1}, we obtain the following classification of irreducible representations of $\mathcal{Z}^{\rm bl}_{\hat\omega}(\fS,\lambda)$ under assumption \Rlabel{1}.

\begin{theorem}[Theorem~\ref{thm-representation-Fock}]\label{intro-2}
Let $\fS$ be a triangulable connected punctured surface of genus $g$ and $m$ punctures, and let $\mathcal P$ denote the set of punctures of $\fS$.  
Suppose that assumption \Rlabel{1} holds.  
Let 
$$f\colon \mathcal{Z}^{\rm bl}_{\hat\eta}(\fS,\lambda)_d\to \mathbb C$$ be an algebra homomorphism, and let $b(p,i)\in\mathbb C^{*}$ satisfy
\[
b(p,i)^N=f(Z^{{\bf b}(\lambda,p,i)}) \qquad \text{for } 1\leq i\leq n-1,\; p\in\mathcal P.
\]
Then there exists a unique irreducible representation
\[
\rho\colon \mathcal{Z}^{\rm bl}_{\hat\omega}(\fS,\lambda)\;\longrightarrow\; \mathrm{End}_{\mathbb C}(V)
\]
satisfying the following conditions:
\begin{enumerate}
    \item $\rho(\PT(x)) = f(x)\,\Id_V$ for all $x\in \mathcal{Z}^{\rm bl}_{\hat\eta}(\fS,\lambda)_d$, where $\PT$ is the embedding defined in \eqref{intro-eq-def-F};
    \item $\rho(Z^{{\bf b}(\lambda,p,i)}) = b(p,i)\,\Id_V$ for $1\leq i\leq n-1$ and $p\in\mathcal P$.
\end{enumerate}
In particular, 
$$\dim_{\mathbb C}V = d^{g} N^{(n^2-1)(g-1) + \tfrac{1}{2}n(n-1)m}.$$
\end{theorem}

For $n=2$, Theorem~\ref{intro-2} was proved in \cite{representation2,BL07}.

\subsection{Representations of the (projected) ${\rm SL}_n$-skein algebra}

Under assumption \Rlabel{1} and
\[
[n]_\omega! = \prod_{k=1}^n \frac{\omega^k - \omega^{-k}}{\omega-\omega^{-1}} \neq 0,
\]
Kim, L{\^e}, and the author constructed the Frobenius homomorphism for the projected $\SL$-skein algebra \cite{KLW} (see Theorem~\ref{thm-Fro}):
\[
\Phi\colon \cS_{\bar\eta}^{*}(\fS) \longrightarrow \cS_{\bar\omega}^{*}(\fS),
\]
where $\cS^{*}(\fS)$ (see \eqref{def-projected-skein}), called the \emph{projected $\SL$-skein algebra}, is obtained from $\cS(\fS)$ by quotienting out the kernel of the splitting homomorphism (see \S\ref{sub-splitting}).  

It was conjectured in \cite{LS21} that $\cS_{\bar\omega}^{*}(\fS) = \cS_{\bar\omega}(\fS)$.  
This conjecture was confirmed for $n=2,3$ in \cite{le2018triangular,higgins2020triangular}, and for general $n$ with $\bar\omega^n=1$ in \cite{wang2024TQFT}.

Note that $\bar\eta=\pm 1$ under assumption \Rlabel{2}.  
By \cite[Theorems~4.2 and 5.8]{wang2024TQFT}, we have $\cS_{\bar\eta}^{*}(\fS)= \cS_{\bar\eta}(\fS)$ when $\bar\eta=\pm 1$.  
Furthermore, it was shown in \cite{S2001SLn,wang2024TQFT} that in this case there exists an algebra isomorphism
\begin{align}\label{intro-bijection-skein-ring}
\mathcal T\colon \cS_{\bar\eta}(\fS)\;\xrightarrow{\;\cong\;}\; \mathcal{O}_{{\rm SL}_n(\mathbb{C})}(\fS).
\end{align}

Now let $\fS$ be a triangulable punctured surface with a triangulation $\lambda$, and let $\mathcal P$ denote the set of punctures of $\fS$.  
Consider a finite-dimensional irreducible representation
\[
\rho\colon \cS_{\bar\omega}^{*}(\fS)\longrightarrow \mathrm{End}_{\mathbb C}(V).
\]
As shown in \S\ref{subsec-rep-skein}, $\rho$ determines a pair
\[
(\chi,\{s(p,i)\mid 1\leq i\leq n-1,\;p\in\mathcal P\}), \qquad \text{(the \emph{classical shadow} of $\rho$)},
\]
where $\chi\in \mathfrak{X}_{{\rm SL}_n(\mathbb{C})}(\fS)$ and $\{s(p,i)\}\subset \mathbb C$, such that
\begin{align*}
\rho(\Phi(\beta)) &= \mathcal T(\beta)(\chi)\,\Id_V, && \text{for $\beta\in\cS_{\bar\eta}(\fS)$}, \\
\rho([\alpha(p)_i]_{\bar\omega}) &= s(p,i)\,\Id_V, && \text{for $1\leq i\leq n-1$, $p\in\mathcal P$,}
\end{align*}
where each $[\alpha(p)_i]_{\bar\omega}$ is a central element in $\cS_{\bar\omega}^{*}(\fS)$ associated to the puncture $p$ (see \S\ref{subsec-image-loop}).  
These data satisfy the relation
\begin{equation}\label{intro-shadow-relation}
\CT([\alpha(p)_i]_{\bar\eta})(\chi)
=\bar P_{N,i}(s(p,1),\ldots, s(p,n-1)), \qquad
\text{for $1\leq i\leq n-1,\;p\in\mathcal P$},
\end{equation}
where $\bar P_{N,i}$
is defined in \eqref{eq-def-barP}.

We denote by $\overline{\mathcal{T}}$ the composition
\[
\HA(\mathcal{Z}_{\hat\eta}^{\rm bl}(\fS,\lambda),\mathbb C)\xrightarrow{{\rm tr}^*}\HA(\cS_{\bar\eta}(\fS),\mathbb C)\xrightarrow{\mathcal{T}_{*}}\mathfrak{X}_{{\rm SL}_n(\mathbb{C})}(\fS),
\]
where ${\rm tr}^{*}$ is induced by ${\rm tr}_\lambda$ and $\mathcal T_*$ is induced by $\mathcal T^{-1}$.

\begin{theorem}[Theorem~\ref{thm-main-3}]\label{intro-3}
Suppose assumption \Rlabel{2} holds and $[n]_\omega!\neq 0$.  
Let $\fS$ be a triangulable punctured surface with triangulation $\lambda$, and let $\mathcal P$ be its set of punctures.  
Let
\[
\chi\in \mathfrak{X}_{{\rm SL}_n(\mathbb{C})}(\fS), 
\qquad 
\{s(p,i)\mid 1\leq i\leq n-1,\;p\in\mathcal P\}\subset \mathbb C,
\]
satisfy \eqref{intro-shadow-relation} with $\chi\in\im\overline{\CT}$.  
Then there exists an irreducible representation
\[
\bar\rho\colon \mathcal{Z}^{\rm bl}_{\hat\omega}(\fS,\lambda)\;\longrightarrow\; \mathrm{End}_{\mathbb C}(V)
\]
such that the classical shadow of any irreducible sub-representation of
\[
\cS_{\bar\omega}^{*}(\fS)\xrightarrow{\rm tr}\mathcal{Z}^{\rm bl}_{\hat\omega}(\fS,\lambda)\xrightarrow{\bar\rho}\mathrm{End}_{\mathbb C}(V)
\]
is precisely $(\chi,\{s(p,i)\}_{p\in\mathcal P,\,1\leq i\leq n-1})$.
\end{theorem}

We expect that, for any pair 
\[
(\chi,\{s(p,i)\}_{p\in\mathcal P,\,1\leq i\leq n-1})
\]
satisfying \eqref{intro-shadow-relation}, there exists an irreducible representation of 
$\cS_{\bar\omega}^{*}(\fS)$ whose classical shadow is exactly this pair. Moreover, there exists a sufficiently large family $U$ of such pairs with the property that, for every 
\[
(\chi,\{s(p,i)\}_{p\in\mathcal P,\,1\leq i\leq n-1}) \in U,
\]  
the pair uniquely determines an irreducible representation of $\cS_{\bar\omega}^{*}(\fS)$ of dimension  
\[
d^{g} N^{(n^2-1)(g-1) + \tfrac{1}{2}n(n-1)m},
\]  
where $g$ is the genus of $\fS$ and $m$ is the number of punctures of $\fS$.

For ${\rm SL}_2$, this property was conjectured by Bonahon and Wong~\cite{BW16}, and later proved by Frohman, Kania-Bartoszynska, and L{\^e}~\cite{unicity}. The case of the ${\rm SL}_3$-skein algebra was established by Kim and the author in~\cite{kim2024unicity}.

The following theorem shows that the representation constructed in Theorem~\ref{intro-3} is independent of the choice of triangulation $\lambda$.  
Its proof relies on the balanced quantum coordinate change isomorphism constructed in \cite{KimWang}.

\begin{theorem}[Theorem~\ref{thm-naturality}]\label{intro-4}
Let $\fS$ be a triangulable punctured surface, and let $\lambda$ and $\lambda'$ be two triangulations of $\fS$.  
Let $f\colon \mathcal{Z}^{\rm bl}_{\hat\eta}(\fS,\lambda)\to \mathbb C$ be an algebra homomorphism, and let $b(p,i)\in\mathbb C^{*}$ satisfy
\[
b(p,i)^N=f(Z^{{\bf b}(\lambda,p,i)}) \qquad \text{for $1\leq i\leq n-1$, $p\in\mathcal P$}.
\]

Suppose assumption \Rlabel{2} holds.  
Then, under Condition~(\hyperref[cond:star]{*}) for $(f,\{b(p,i)\})$, there exists an algebra homomorphism
\[
f'\colon \mathcal{Z}^{\rm bl}_{\hat\eta}(\fS,\lambda')\longrightarrow \mathbb C
\]
such that:
\begin{enumerate}[label={\rm (\alph*)}]
\item The following diagram commutes:
\[
\begin{tikzcd}
\cS_{\bar \eta}(\fS) \arrow[r, "{\rm tr}_\lambda"]
\arrow[d, "{\rm tr}_{\lambda'}"]  
& \mathcal{Z}^{\rm bl}_{\hat\eta}(\fS,\lambda) \arrow[d, "f"] \\
\mathcal{Z}^{\rm bl}_{\hat\eta}(\fS,\lambda')
\arrow[r, "f'"] 
&  \BC
\end{tikzcd}
\]
and
\[
f(Z^{{\bf b}(\lambda,p,i)}) \;=\; f'(Z^{{\bf b}(\lambda',p,i)}) \qquad \text{for $p\in \mathcal P,\; 1\leq i\leq n-1$};
\]

\item By Theorem~\ref{intro-2}, the data $f$ (resp. $f'$) together with $(b(p,i))_{1\leq i\leq n-1,\; p\in\mathcal P}$ uniquely determine irreducible representations
\[
\rho\colon  \mathcal{Z}^{\rm bl}_{\hat\omega}(\fS,\lambda) \to \mathrm{End}_{\BC}(V), 
\qquad 
\rho'\colon \mathcal{Z}^{\rm bl}_{\hat\omega}(\fS,\lambda') \to \mathrm{End}_\BC(V').
\]
 As representations of $\cS_{\hat\omega}(\fS)$, the following two are isomorphic:
\[
\cS_{\bar\omega}(\fS)\xrightarrow{{\rm tr}_\lambda}\mathcal{Z}^{\rm bl}_{\hat\omega}(\fS,\lambda)\xrightarrow{\rho}\mathrm{End}_\BC(V), 
\quad
\cS_{\bar\omega}(\fS)\xrightarrow{{\rm tr}_{\lambda'}}\mathcal{Z}^{\rm bl}_{\hat\omega}(\fS,\lambda')\xrightarrow{\rho'}\mathrm{End}_\BC(V').
\]
\end{enumerate}
\end{theorem}

Now we state the Condition~(\hyperref[cond:star]{*}).
When $\lambda'$ is obtained from $\lambda$ by a flip (see Figure~\ref{flip}), 
it is well known that there exists a sequence of $\mathcal X$-mutations 
$\mu_{v_r},\ldots,\mu_{v_2},\mu_{v_1}$ \cite{FG06,GS19}
(see \S\ref{subsec-coordinate-change} or \cite[Figure~3]{KimWang}) such that  
\begin{align}\label{intro-mutation}
\Gamma_{\lambda'} \;=\; \mu_{v_r}\cdots \mu_{v_2}\mu_{v_1}(\Gamma_\lambda),    
\end{align}
where each $\mu_{v_i}$ is the mutation of the corresponding quiver at the vertex $v_i$  (see \S\ref{subsec-mutation}).
Since any two triangulations $\lambda$ and $\lambda'$ are related by a sequence of flips, 
it follows that for arbitrary triangulations $\lambda$ and $\lambda'$ there exists a sequence of $\mathcal X$-mutations 
$\mu_{v_m},\ldots,\mu_{v_2},\mu_{v_1}$ satisfying  
\[
\Gamma_{\lambda'} \;=\; \mu_{v_m}\cdots \mu_{v_2}\mu_{v_1}(\Gamma_\lambda).    
\]
The precise mutation sequence depends on the chosen sequence of flips connecting $\lambda$ to $\lambda'$, together with the local mutation sequence from \eqref{intro-mutation}.  

Define  
\begin{align*}
    \mathcal Z_{\hat\omega}^{\rm mbl}(\fS,\lambda)
    := \mathrm{span}_{\mathbb C}\Bigl\{
        Z^{\bf k}\;\Big|\; 
        {\bf k}\in \mathbb Z^{V_\lambda},\ 
        \sum_{u\in V_\lambda} {\bf k}(u)\, Q_\lambda(u,v)\in n\mathbb Z
        \ \text{for all } v\in V_\lambda
    \Bigr\},
\end{align*}
where $Z^{\bf k}$ is defined in \eqref{def-mmmmmmm}.
It is shown in \S\ref{sec-naturality} that  
\[
\mathcal Z_{\hat\omega}^{\rm bl}(\fS,\lambda)\subset \mathcal Z_{\hat\omega}^{\rm mbl}(\fS,\lambda)\text{ (\cite{KimWang})},
\qquad
Z^{{\bf b}(\lambda,p,i)}\in 
\mathsf{Z}\!\left(\mathcal Z_{\hat\omega}^{\rm mbl}(\fS,\lambda)\right)
\quad \text{for $1\leq i\leq n-1$, $p\in\mathcal P$},
\]
where $\mathsf{Z}(\mathcal Z_{\hat\omega}^{\rm mbl}(\fS,\lambda))$ denotes the center of $\mathcal Z_{\hat\omega}^{\rm mbl}(\fS,\lambda)$.  

We define  
\[
L\colon 
\mathcal Z_{\hat\eta}^{\rm bl}(\fS,\lambda)
\xrightarrow{\ \Phi^{\mathbb T}\ }
\mathcal Z_{\hat\omega}^{\rm bl}(\fS,\lambda)
\hookrightarrow 
\mathcal Z_{\hat\omega}^{\rm mbl}(\fS,\lambda).
\]

\medskip

\noindent\textbf{Condition~(\hyperref[cond:star]{*}):}\label{cond:star}  
There exists an algebra homomorphism  
\[
h\colon \mathsf{Z}(\mathcal Z_{\hat\omega}^{\rm mbl}(\fS,\lambda))\longrightarrow \mathbb C
\]
such that:  
\begin{enumerate}
    \item $f = h\circ L$, and  
    $b(p,i) = h\!\left(Z^{{\bf b}(\lambda,p,i)}\right)$ for $1\leq i\leq n-1$, $p\in\mathcal P$;

    \item $f|_{\mathcal X_{\eta}(\fS,\lambda)}$ is compatible with the mutation sequence $(\mu_{v_1}^\eta,\ldots,\mu_{v_m}^\eta)$ (see Definition~\ref{def-compatible}), where $\mathcal X_{\eta}(\fS,\lambda)$, defined in \eqref{intro-X}, is regarded as a subalgebra of $\mathcal Z_{\hat\eta}^{\rm bl}(\fS,\lambda)$ via the embedding \eqref{intro-embedding}.
\end{enumerate}

\medskip

For $n=2$, Theorem~\ref{intro-3} was proved in \cite{representation2}, and 
Theorem~\ref{intro-4} was proved in in \cite{representation3}.

\medskip

{\bf
Acknowledgments}: 
The author would like to thank Hyun Kyu Kim for the helpful discussion. 
Part of this research was carried out while the author was a Dual PhD student at the University of Groningen (The Netherlands) and Nanyang Technological University (Singapore), supported by a PhD scholarship from the University of Groningen and a research scholarship from Nanyang Technological University.  
Z.W. was supported by a KIAS Individual Grant (MG104701) at the Korea Institute for Advanced Study.

\section{Preliminaries}
In this section, we review the reduced stated ${\rm SL}_n$-skein algebra as presented in \cite{LS21,LY23}.  
We also introduce the splitting homomorphisms and the quantum trace maps for  
reduced stated ${\rm SL}_n$-skein algebras.

\subsection{Reduced stated $\SL$-skein algebras}
\label{sec-sub-def-stated-skein}

Let $\hat\omega^{\frac{1}{2}}\in \mathbb C^{*}$. Set $\bar \omega = \omega^{n}$ and $\omega = \omega^{n^2}$ with $\bar\omega^{\frac{1}{2n}} = \omega^{\frac{1}{2}}$ and 
$\omega^{\frac{1}{2n^2}} = \omega^{\frac{1}{2}}$.  
Define the following constants:
\begin{align*}
\mathbbm{c}_{i}= (-\omega)^{n-i} \omega^{\frac{n-1}{2n}},\quad
\mathbbm{t}= (-1)^{n-1} \omega^{\frac{n^2-1}{n}},\quad 
\mathbbm{a} =   \omega^{\frac{n+1-2n^2}{4}}.
\end{align*}

\begin{definition}
    \label{def:surface}
    By a {\bf pb surface} we mean a surface obtained from an oriented compact surface $\overline{\fS}$ by removing a finite set $\mathcal{P}$ of points such that $\mathcal{P}$ has a non-empty intersection with each boundary component of $\overline{\fS}$. The members of $\mathcal{P}$ are called {\bf punctures} of $\fS$. 
    If $\fS$ has empty boundary, 
    we say that $\fS$ is a {\bf punctured surface}, which is of primary interest in the present paper. 
    If $\mathcal{P}=\emptyset$, we say that $\fS$ is a {\bf closed surface}.

    An embedding $c\colon(0,1)\rightarrow \fS$ is called an {\bf ideal arc} if both
$\bar c(0)$ and $\bar c(1)$ are punctures, where $\bar c\colon [0,1] \to \overline{\fS}$ is the `closure' of $c$.

\end{definition}

\def\Si{\fS}

For any positive integer $k$, we use $\mathbb P_k$ to denote the pb surface obtained from the closed disk by removing $k$ punctures from the boundary. 
We call $\mathbb P_1$ the {\bf monogon}.

Consider the 3-manifold $\Si \times (-1,1)$, the thickened surface. For a point $(x,t) \in \Si \times (-1,1)$, the value $t$ is called the {\bf height} of this point. An {\bf $n$-web} $\alpha$ in $\Si\times(-1,1)$ is a disjoint union of oriented closed curves and a directed finite graph properly embedded into $\Si\times(-1,1)$, satisfying the following requirements:
\begin{enumerate}
    \item $\alpha$ only contains $1$-valent or $n$-valent vertices. Each $n$-valent vertex is a source or a  sink. The set of $1$-valent vertices is denoted as $\partial \alpha$, which are called \textbf{endpoints} of $\alpha$. For any boundary component $c$ of $\Si$, we require that the points of  $\partial\alpha\cap (c\times(-1,1))$ have mutually distinct heights.
    \item Every edge of the graph is an embedded oriented  closed interval  in $\Si\times(-1,1)$.
    \item $\alpha$ is equipped with a transversal \textbf{framing}. 
    \item The set of half-edges at each $n$-valent vertex is equipped with a  cyclic order. 
    \item $\partial \alpha$ is contained in $\partial\Si\times (-1,1)$ and the framing at these endpoints is given by the positive direction of $(-1,1)$.
\end{enumerate}
We will consider $n$-webs up to (ambient) \textbf{isotopy} which are continuous deformations of $n$-webs in their class. 
The empty $n$-web, denoted by $\emptyset$, is also considered as an $n$-web, with the convention that $\emptyset$ is only isotopic to itself. 

A {\bf state} for $\alpha$ is a map $s\colon\partial\alpha\rightarrow \{1,2,\cdots,n\}$. A {\bf stated $n$-web} in $\Si\times(-1,1)$ is an $n$-web equipped with a state.

By regarding $\fS$ as $\fS\times\{0\}$, there is a projection $\text{pr}\colon\fS\times(-1,1)\rightarrow \fS.$
We say the (stated) $n$-web $\alpha$ is in {\bf vertical position} if 
\begin{enumerate}
    \item the framing at everywhere is given by the positive direction of $(-1,1)$,
    \item $\alpha$ is in general position with respect to the projection  $\text{pr}\colon \Si\times(-1,1)\rightarrow \Si\times\{0\}$,
    \item at every $n$-valent vertex, the cyclic order of half-edges as the image of $\text{pr}$ is given by the positive orientation of $\Si$ (drawn counter-clockwise in pictures).
\end{enumerate}

For every (stated) $n$-web $\alpha$, we can isotope $\alpha$ to be in vertical position. For each boundary component $c$ of $\fS$, the heights of $\partial\alpha\cap (c\times(-1,1))$ determine a linear order on  $c\cap \text{pr}(\alpha)$.
Then a {\bf (stated) $n$-web diagram} of $\alpha$ is $\text{pr}(\alpha)$ equipped with the usual over/underpassing information at each double point (called a crossing) and a linear order on $c\cap \text{pr}(\alpha)$ for each boundary component $c$ of $\fS$.

Let $S_n$ denote the permutation group on the set $\{1,2,\cdots,n\}$. 
For an integer $i\in\{1,2,\cdots,n\}$, we use $\bar{i}$ to denote $n+1-i$.

\def\M {M,\cN}
\def\BC{\mathbb C}

The \textbf{stated $\SL$-skein algebra} $\cS_{\hat\omega}^{\rm st}(\fS)$ of $\fS$ is
the quotient module of the $\BC$-vector space freely generated by the set 
 of all isotopy classes of stated 
$n$-webs in $\fS\times (-1,1)$ subject to  relations \eqref{w.cross}-\eqref{wzh.eight}.

\beq\label{w.cross}
\omega^{\frac{1}{n}} 
\raisebox{-.20in}{

\begin{tikzpicture}
\tikzset{->-/.style=

{decoration={markings,mark=at position #1 with

{\arrow{latex}}},postaction={decorate}}}
\filldraw[draw=white,fill=gray!20] (-0,-0.2) rectangle (1, 1.2);
\draw [line width =1pt,decoration={markings, mark=at position 0.5 with {\arrow{>}}},postaction={decorate}](0.6,0.6)--(1,1);
\draw [line width =1pt,decoration={markings, mark=at position 0.5 with {\arrow{>}}},postaction={decorate}](0.6,0.4)--(1,0);
\draw[line width =1pt] (0,0)--(0.4,0.4);
\draw[line width =1pt] (0,1)--(0.4,0.6);
\draw[line width =1pt] (0.4,0.6)--(0.6,0.4);
\end{tikzpicture}
}
- \omega^{-\frac {1}{n}}
\raisebox{-.20in}{
\begin{tikzpicture}
\tikzset{->-/.style=

{decoration={markings,mark=at position #1 with

{\arrow{latex}}},postaction={decorate}}}
\filldraw[draw=white,fill=gray!20] (-0,-0.2) rectangle (1, 1.2);
\draw [line width =1pt,decoration={markings, mark=at position 0.5 with {\arrow{>}}},postaction={decorate}](0.6,0.6)--(1,1);
\draw [line width =1pt,decoration={markings, mark=at position 0.5 with {\arrow{>}}},postaction={decorate}](0.6,0.4)--(1,0);
\draw[line width =1pt] (0,0)--(0.4,0.4);
\draw[line width =1pt] (0,1)--(0.4,0.6);
\draw[line width =1pt] (0.6,0.6)--(0.4,0.4);
\end{tikzpicture}
}
= (\omega-\omega^{-1})
\raisebox{-.20in}{

\begin{tikzpicture}
\tikzset{->-/.style=

{decoration={markings,mark=at position #1 with

{\arrow{latex}}},postaction={decorate}}}
\filldraw[draw=white,fill=gray!20] (-0,-0.2) rectangle (1, 1.2);
\draw [line width =1pt,decoration={markings, mark=at position 0.5 with {\arrow{>}}},postaction={decorate}](0,0.8)--(1,0.8);
\draw [line width =1pt,decoration={markings, mark=at position 0.5 with {\arrow{>}}},postaction={decorate}](0,0.2)--(1,0.2);
\end{tikzpicture}
},
\eeq 
\beq\label{w.twist}
\raisebox{-.15in}{
\begin{tikzpicture}
\tikzset{->-/.style=
{decoration={markings,mark=at position #1 with
{\arrow{latex}}},postaction={decorate}}}
\filldraw[draw=white,fill=gray!20] (-1,-0.35) rectangle (0.6, 0.65);
\draw [line width =1pt,decoration={markings, mark=at position 0.5 with {\arrow{>}}},postaction={decorate}](-1,0)--(-0.25,0);
\draw [color = black, line width =1pt](0,0)--(0.6,0);
\draw [color = black, line width =1pt] (0.166 ,0.08) arc (-37:270:0.2);
\end{tikzpicture}}
= \mathbbm{t}
\raisebox{-.15in}{
\begin{tikzpicture}
\tikzset{->-/.style=
{decoration={markings,mark=at position #1 with
{\arrow{latex}}},postaction={decorate}}}
\filldraw[draw=white,fill=gray!20] (-1,-0.5) rectangle (0.6, 0.5);
\draw [line width =1pt,decoration={markings, mark=at position 0.5 with {\arrow{>}}},postaction={decorate}](-1,0)--(-0.25,0);
\draw [color = black, line width =1pt](-0.25,0)--(0.6,0);
\end{tikzpicture}}
,  
\eeq
\beq\label{w.unknot}
\raisebox{-.20in}{
\begin{tikzpicture}
\tikzset{->-/.style=
{decoration={markings,mark=at position #1 with
{\arrow{latex}}},postaction={decorate}}}
\filldraw[draw=white,fill=gray!20] (0,0) rectangle (1,1);
\draw [line width =1pt,decoration={markings, mark=at position 0.5 with {\arrow{>}}},postaction={decorate}](0.45,0.8)--(0.55,0.8);
\draw[line width =1pt] (0.5 ,0.5) circle (0.3);
\end{tikzpicture}}
= (-1)^{n-1} [n]_\omega\ 
\raisebox{-.20in}{
\begin{tikzpicture}
\tikzset{->-/.style=
{decoration={markings,mark=at position #1 with
{\arrow{latex}}},postaction={decorate}}}
\filldraw[draw=white,fill=gray!20] (0,0) rectangle (1,1);
\end{tikzpicture}}
,\ \text{where}\ [n]_{\omega}=\frac{\omega^n-\omega^{-n}}{\omega-\omega^{-1}},
\eeq
\beq\label{wzh.four}
\raisebox{-.30in}{
\begin{tikzpicture}
\tikzset{->-/.style=
{decoration={markings,mark=at position #1 with
{\arrow{latex}}},postaction={decorate}}}
\filldraw[draw=white,fill=gray!20] (-1,-0.7) rectangle (1.2,1.3);
\draw [line width =1pt,decoration={markings, mark=at position 0.5 with {\arrow{>}}},postaction={decorate}](-1,1)--(0,0);
\draw [line width =1pt,decoration={markings, mark=at position 0.5 with {\arrow{>}}},postaction={decorate}](-1,0)--(0,0);
\draw [line width =1pt,decoration={markings, mark=at position 0.5 with {\arrow{>}}},postaction={decorate}](-1,-0.4)--(0,0);
\draw [line width =1pt,decoration={markings, mark=at position 0.5 with {\arrow{<}}},postaction={decorate}](1.2,1)  --(0.2,0);
\draw [line width =1pt,decoration={markings, mark=at position 0.5 with {\arrow{<}}},postaction={decorate}](1.2,0)  --(0.2,0);
\draw [line width =1pt,decoration={markings, mark=at position 0.5 with {\arrow{<}}},postaction={decorate}](1.2,-0.4)--(0.2,0);
\node  at(-0.8,0.5) {$\vdots$};
\node  at(1,0.5) {$\vdots$};
\end{tikzpicture}}=(-\omega)^{\frac{n(n-1)}{2}}\cdot \sum_{\sigma\in S_n}
(-\omega^{\frac{1-n}n})^{\ell(\sigma)} \raisebox{-.30in}{
\begin{tikzpicture}
\tikzset{->-/.style=
{decoration={markings,mark=at position #1 with
{\arrow{latex}}},postaction={decorate}}}
\filldraw[draw=white,fill=gray!20] (-1,-0.7) rectangle (1.2,1.3);
\draw [line width =1pt,decoration={markings, mark=at position 0.5 with {\arrow{>}}},postaction={decorate}](-1,1)--(0,0);
\draw [line width =1pt,decoration={markings, mark=at position 0.5 with {\arrow{>}}},postaction={decorate}](-1,0)--(0,0);
\draw [line width =1pt,decoration={markings, mark=at position 0.5 with {\arrow{>}}},postaction={decorate}](-1,-0.4)--(0,0);
\draw [line width =1pt,decoration={markings, mark=at position 0.5 with {\arrow{<}}},postaction={decorate}](1.2,1)  --(0.2,0);
\draw [line width =1pt,decoration={markings, mark=at position 0.5 with {\arrow{<}}},postaction={decorate}](1.2,0)  --(0.2,0);
\draw [line width =1pt,decoration={markings, mark=at position 0.5 with {\arrow{<}}},postaction={decorate}](1.2,-0.4)--(0.2,0);
\node  at(-0.8,0.5) {$\vdots$};
\node  at(1,0.5) {$\vdots$};
\filldraw[draw=black,fill=gray!20,line width =1pt]  (0.1,0.3) ellipse (0.4 and 0.7);
\node  at(0.1,0.3){$\sigma_{+}$};
\end{tikzpicture}},
\eeq
where the ellipse enclosing $\sigma_+$  is the minimum crossing positive braid representing a permutation $\sigma\in S_n$ and $\ell(\sigma)=\#\{(i,j)\mid 1\leq i<j\leq n,\ \sigma(i)>\sigma(j)\}$ is the length of $\sigma\in S_n$.

\beq\label{wzh.five}
   \raisebox{-.30in}{
\begin{tikzpicture}
\tikzset{->-/.style=
{decoration={markings,mark=at position #1 with
{\arrow{latex}}},postaction={decorate}}}
\filldraw[draw=white,fill=gray!20] (-1,-0.7) rectangle (0.2,1.3);
\draw [line width =1pt](-1,1)--(0,0);
\draw [line width =1pt](-1,0)--(0,0);
\draw [line width =1pt](-1,-0.4)--(0,0);
\draw [line width =1.5pt](0.2,1.3)--(0.2,-0.7);
\node  at(-0.8,0.5) {$\vdots$};
\filldraw[fill=white,line width =0.8pt] (-0.5 ,0.5) circle (0.07);
\filldraw[fill=white,line width =0.8pt] (-0.5 ,0) circle (0.07);
\filldraw[fill=white,line width =0.8pt] (-0.5 ,-0.2) circle (0.07);
\end{tikzpicture}}
   = 
   \mathbbm{a} \sum_{\sigma \in \fS_n} (-\omega)^{\ell(\sigma)}\,  \raisebox{-.30in}{
\begin{tikzpicture}
\tikzset{->-/.style=
{decoration={markings,mark=at position #1 with
{\arrow{latex}}},postaction={decorate}}}
\filldraw[draw=white,fill=gray!20] (-1,-0.7) rectangle (0.2,1.3);
\draw [line width =1pt](-1,1)--(0.2,1);
\draw [line width =1pt](-1,0)--(0.2,0);
\draw [line width =1pt](-1,-0.4)--(0.2,-0.4);
\draw [line width =1.5pt,decoration={markings, mark=at position 1 with {\arrow{>}}},postaction={decorate}](0.2,1.3)--(0.2,-0.7);
\node  at(-0.8,0.5) {$\vdots$};
\filldraw[fill=white,line width =0.8pt] (-0.5 ,1) circle (0.07);
\filldraw[fill=white,line width =0.8pt] (-0.5 ,0) circle (0.07);
\filldraw[fill=white,line width =0.8pt] (-0.5 ,-0.4) circle (0.07);
\node [right] at(0.2,1) {$\sigma(n)$};
\node [right] at(0.2,0) {$\sigma(2)$};
\node [right] at(0.2,-0.4){$\sigma(1)$};
\end{tikzpicture}},
\eeq
\beq \label{wzh.six}
\raisebox{-.20in}{
\begin{tikzpicture}
\tikzset{->-/.style=
{decoration={markings,mark=at position #1 with
{\arrow{latex}}},postaction={decorate}}}
\filldraw[draw=white,fill=gray!20] (-0.7,-0.7) rectangle (0,0.7);
\draw [line width =1.5pt,decoration={markings, mark=at position 1 with {\arrow{>}}},postaction={decorate}](0,0.7)--(0,-0.7);
\draw [color = black, line width =1pt] (0 ,0.3) arc (90:270:0.5 and 0.3);
\node [right]  at(0,0.3) {$i$};
\node [right] at(0,-0.3){$j$};
\filldraw[fill=white,line width =0.8pt] (-0.5 ,0) circle (0.07);
\end{tikzpicture}}   = \delta_{\bar j,i }\,  \mathbbm{c}_{i} \raisebox{-.20in}{
\begin{tikzpicture}
\tikzset{->-/.style=
{decoration={markings,mark=at position #1 with
{\arrow{latex}}},postaction={decorate}}}
\filldraw[draw=white,fill=gray!20] (-0.7,-0.7) rectangle (0,0.7);
\draw [line width =1.5pt](0,0.7)--(0,-0.7);
\end{tikzpicture}},
\eeq
\beq \label{wzh.seven}
\raisebox{-.20in}{
\begin{tikzpicture}
\tikzset{->-/.style=
{decoration={markings,mark=at position #1 with
{\arrow{latex}}},postaction={decorate}}}
\filldraw[draw=white,fill=gray!20] (-0.7,-0.7) rectangle (0,0.7);
\draw [line width =1.5pt](0,0.7)--(0,-0.7);
\draw [color = black, line width =1pt] (-0.7 ,-0.3) arc (-90:90:0.5 and 0.3);
\filldraw[fill=white,line width =0.8pt] (-0.55 ,0.26) circle (0.07);
\end{tikzpicture}}
= \sum_{i=1}^n  (\mathbbm{c}_{\bar i})^{-1}\, \raisebox{-.20in}{
\begin{tikzpicture}
\tikzset{->-/.style=
{decoration={markings,mark=at position #1 with
{\arrow{latex}}},postaction={decorate}}}
\filldraw[draw=white,fill=gray!20] (-0.7,-0.7) rectangle (0,0.7);
\draw [line width =1.5pt,decoration={markings, mark=at position 1 with {\arrow{>}}},postaction={decorate}](0,0.7)--(0,-0.7);
\draw [line width =1pt](-0.7,0.3)--(0,0.3);
\draw [line width =1pt](-0.7,-0.3)--(0,-0.3);
\filldraw[fill=white,line width =0.8pt] (-0.3 ,0.3) circle (0.07);
\filldraw[fill=black,line width =0.8pt] (-0.3 ,-0.3) circle (0.07);
\node [right]  at(0,0.3) {$i$};
\node [right]  at(0,-0.3) {$\bar{i}$};
\end{tikzpicture}},
\eeq
\beq\label{wzh.eight}
\raisebox{-.20in}{

\begin{tikzpicture}
\tikzset{->-/.style=

{decoration={markings,mark=at position #1 with

{\arrow{latex}}},postaction={decorate}}}
\filldraw[draw=white,fill=gray!20] (-0,-0.2) rectangle (1, 1.2);
\draw [line width =1.5pt,decoration={markings, mark=at position 1 with {\arrow{>}}},postaction={decorate}](1,1.2)--(1,-0.2);
\draw [line width =1pt](0.6,0.6)--(1,1);
\draw [line width =1pt](0.6,0.4)--(1,0);
\draw[line width =1pt] (0,0)--(0.4,0.4);
\draw[line width =1pt] (0,1)--(0.4,0.6);
\draw[line width =1pt] (0.4,0.6)--(0.6,0.4);
\filldraw[fill=white,line width =0.8pt] (0.2 ,0.2) circle (0.07);
\filldraw[fill=white,line width =0.8pt] (0.2 ,0.8) circle (0.07);
\node [right]  at(1,1) {$i$};
\node [right]  at(1,0) {$j$};
\end{tikzpicture}
} =\omega^{-\frac{1}{n}}\left(\delta_{{j<i} }(\omega-\omega^{-1})\raisebox{-.20in}{

\begin{tikzpicture}
\tikzset{->-/.style=

{decoration={markings,mark=at position #1 with

{\arrow{latex}}},postaction={decorate}}}
\filldraw[draw=white,fill=gray!20] (-0,-0.2) rectangle (1, 1.2);
\draw [line width =1.5pt,decoration={markings, mark=at position 1 with {\arrow{>}}},postaction={decorate}](1,1.2)--(1,-0.2);
\draw [line width =1pt](0,0.8)--(1,0.8);
\draw [line width =1pt](0,0.2)--(1,0.2);
\filldraw[fill=white,line width =0.8pt] (0.2 ,0.8) circle (0.07);
\filldraw[fill=white,line width =0.8pt] (0.2 ,0.2) circle (0.07);
\node [right]  at(1,0.8) {$i$};
\node [right]  at(1,0.2) {$j$};
\end{tikzpicture}
}+\omega^{\delta_{i,j}}\raisebox{-.20in}{

\begin{tikzpicture}
\tikzset{->-/.style=

{decoration={markings,mark=at position #1 with

{\arrow{latex}}},postaction={decorate}}}
\filldraw[draw=white,fill=gray!20] (-0,-0.2) rectangle (1, 1.2);
\draw [line width =1.5pt,decoration={markings, mark=at position 1 with {\arrow{>}}},postaction={decorate}](1,1.2)--(1,-0.2);
\draw [line width =1pt](0,0.8)--(1,0.8);
\draw [line width =1pt](0,0.2)--(1,0.2);
\filldraw[fill=white,line width =0.8pt] (0.2 ,0.8) circle (0.07);
\filldraw[fill=white,line width =0.8pt] (0.2 ,0.2) circle (0.07);
\node [right]  at(1,0.8) {$j$};
\node [right]  at(1,0.2) {$i$};
\end{tikzpicture}
}\right),
\eeq
where   
$\delta_{j<i}= 
\begin{cases}
1  & j<i\\
0 & \text{otherwise}
\end{cases},\ 
\delta_{i,j}= 
\begin{cases} 
1  & i=j\\
0  & \text{otherwise}
\end{cases}$, and small white dots represent an arbitrary orientation of the edges (left-to-right or right-to-left), consistent for the entire equation. The black dot represents the opposite orientation. When a boundary edge of a shaded area is directed, the direction indicates the height order of the endpoints of the diagrams on that directed line, where going along the direction increases the height, and the involved endpoints are consecutive in the height order. The height order outside the drawn part can be arbitrary.

 \def \Sv{\cS_n(\Si,\mathbbm{v})}

 The algebra structure for $\cS_{\hat\omega}^{\rm st}(\fS)$ is given by stacking the stated $n$-webs, i.e. for any two stated $n$-webs $\alpha,\alpha'
 \subset\fS\times(-1,1)$, the product $\alpha\alpha'$ is defined by stacking $\alpha$ above $\alpha'$. That is, if $\alpha \subset \fS\times(0,1)$ and $\alpha' \subset \fS\times (-1,0)$, we have $\alpha \alpha' = \alpha \cup \alpha'$.

For a boundary puncture $p$ (i.e., a puncture contained in $\partial\fS$) of a pb surface $\fS$,  
the corner arcs $C(p)_{ij}$ and $\cev{C}(p)_{ij}$ are stated arcs, as illustrated in Figure \ref{Fig;badarc}.
For a boundary puncture $p$ which is not on a monogon component of $\fS$, set 
$$C_p=\{C(p)_{ij}\mid i<j\},\quad\cev{C}_p=\{\cev{C}(p)_{ij}\mid i<j\}.$$  
Each element of $C_p\cup \cev{C}_p$ is called a \emph{bad arc} at $p$. 
\begin{figure}[h]
    \centering
    \includegraphics[width=150pt]{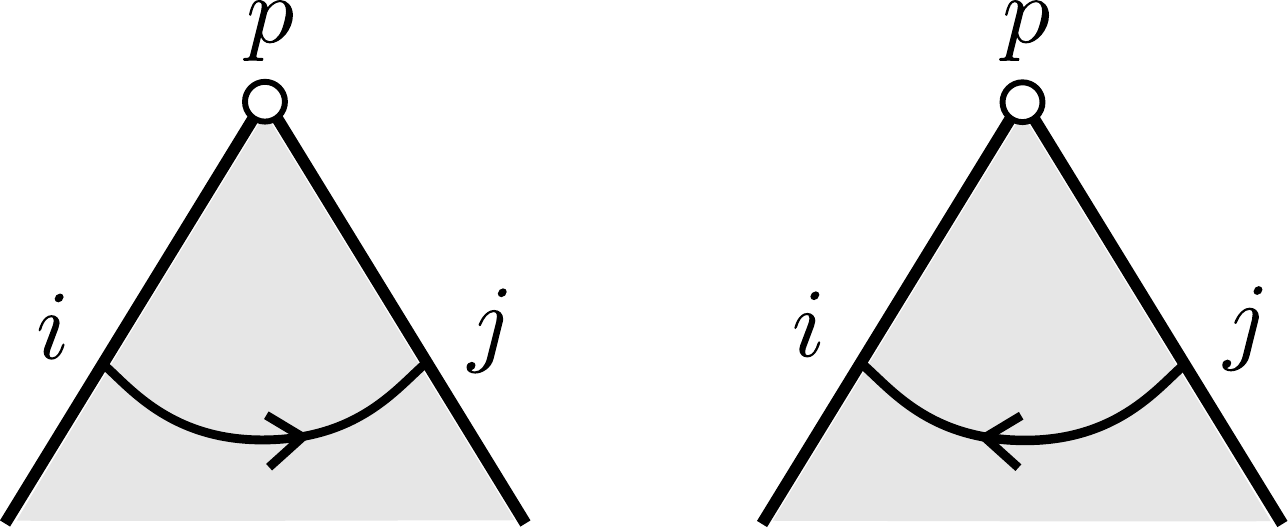}
    \caption{The left is $C(p)_{ij}$ and the right is $\cev{C}(p)_{ij}$.}\label{Fig;badarc}
\end{figure}

For a pb surface $\fS$,  $$\overline \cS_{\hat\omega}^{\rm st}(\fS) = \cS_{\hat\omega}^{\rm st}(\fS)/I^{\text{bad}}$$ 
is called the \textbf{reduced stated $\SL$-skein algebra}, defined in \cite{LY23}, where $I^{\text{bad}}$ is the two-sided ideal of $\cS_{\hat\omega}^{\rm st}(\fS)$ generated by all bad arcs. 

When $\partial\fS = \emptyset$, we have 
$\overline{\cS}_{\hat\omega}^{\rm st}(\fS)=\cS_{\hat\omega}^{\rm st}(\fS)$,  
and the definition of $\cS_{\hat\omega}^{\rm st}(\fS)$ depends only on $\bar\omega$.  
For this reason, we often denote $\cS_{\hat\omega}^{\rm st}(\fS)$ simply by $\cS_{\bar\omega}(\fS)$ when $\partial\fS = \emptyset$,  
although in certain contexts we may still retain the notation $\cS_{\hat\omega}^{\rm st}(\fS)$.  
In this case, we refer to the reduced stated ${\rm SL}_n$-skein algebra simply as the {\bf ${\rm SL}_n$-skein algebra}.



\subsection{The splitting homomorphism}\label{sub-splitting}
\def\cut{\mathsf{Cut}}
\def\pr{{\bf pr}}

Let $e$ be an ideal arc of a pb surface $\fS$ contained in the interior of $\fS$.  
Cutting $\fS$ along $e$ produces a new pb surface $\cut_e(\fS)$ with two copies $e_1,e_2$ of $e$, so that  
${\fS}= \cut_e(\fS)/(e_1=e_2)$. Denote by $\pr_e$ the projection from $\cut_e(\fS)$ to $\fS$.  

Suppose $\alpha$ is a stated $n$-web diagram in $\fS$ transverse to $e$.  
Let $s \colon e\cap\alpha \to \{1,2,\dots,n\}$ be a map, and let $h$ be a linear order on $e\cap\alpha$.  
Then there is a lifted diagram $\alpha(h,s)$ of a stated $n$-web in $\cut_e(\fS)$.  
For $i=1,2$, the heights of the endpoints of $\alpha(h,s)$ on $e_i$ are induced from $h$ (via $\pr_e$),  
and their states are induced from $s$ (via $\pr_e$).  
The splitting homomorphism $\mathbb S_e \colon \cS^{\rm st}_{\hat\omega}(\cut_e(\fS)) \longrightarrow \cS^{\rm st}_{\hat\omega}(\cut_e(\fS))$ is then defined by  
\begin{align}\label{eq-def-splitting}
    \mathbb S_e(\alpha) = \sum_{s \colon \alpha \cap e \to \{1,\dots, n\}} \alpha(h,s). 
\end{align}
Moreover, $\mathbb S_e$ is an algebra homomorphism \cite{LS21}.  

Note that $\mathbb S_e$ sends bad arcs to bad arcs.  
Hence it induces an algebra homomorphism  
\[
   \mathbb S_e \colon \rdS \longrightarrow \rS^{\rm st}(\cut_e(\fS)),
\]  
which we continue to denote by $\mathbb S_e$.


\def\bT{\mathbb T}

\def\VV{V}

\subsection{Quantum torus and its $n$-th root of version}\label{sub-quantum-torus}
In the next subsection, we will review the $\SL$ quantum trace map from \cite{LY23},  
which is an algebra homomorphism from the reduced stated $\SL$-skein algebra to a suitable quantum torus.  
In the present subsection, we introduce the relevant definitions of the quantum torus.

Let $V$ be a finite set, and $Q$ be an anti-symmetric matrix $Q\colon V\times V\rightarrow \frac{1}{2}\mathbb Z$.
Define the quantum torus
\begin{equation*}
\bT_\omega(Q) = \BC \langle
X_v^{\pm 1}, v \in V \rangle / (
X_v 
X_{v'}= \omega^{\, 2 Q(v,v')} 
X_{v'} 
X_v \text{ for } v,v'\in V),
\end{equation*}
and its $n$-th root of version as follows:
\begin{equation*}
\bT_{\hat \omega}(Q) = \BC \langle
Z_v^{\pm 1}, v \in V \rangle / (
Z_v 
Z_{v'}= \hat\omega^{\, 2 Q(v,v')} 
Z_{v'} 
Z_v \text{ for } v,v'\in V ).
\end{equation*}
There is an algebra embedding from $\bT_\omega(Q)$
to $\bT_{\hat{\omega}}(Q)$ defined by $X_v\mapsto Z_v^n$
for $v\in V$.

A useful notion is the Weyl-ordered product, which is a certain normalization of a Laurent monomial defined as follows: for any $v_1,\ldots,v_r \in V$ and $a_1,\ldots,a_r \in \mathbb{Z}$,
\begin{align}
\label{Weyl_ordering-X-1}
\left[ X_{v_1}^{a_1} X_{v_2}^{a_2} \cdots X_{v_r}^{a_r} \right] := \omega^{-\sum_{i<j} a_i a_jQ(v_i,v_j)} X_{v_1}^{a_1} X_{v_2}^{a_2} \cdots X_{v_r}^{a_r}\\
\label{Weyl_ordering-Z-1}
\left[ Z_{v_1}^{a_1} Z_{v_2}^{a_2} \cdots Z_{v_r}^{a_r} \right] := \hat \omega^{-\sum_{i<j}a_i a_jQ(v_i,v_j)} Z_{v_1}^{a_1} Z_{v_2}^{a_2} \cdots Z_{v_r}^{a_r}.
\end{align}
In particular, for ${\bf t} = (t_v)_{v\in \VV} \in \mathbb{Z}^{V}$, the following notation for the corresponding Weyl-ordered Laurent monomial will become handy:
\begin{align}\label{def-mmmmmmm}
    X^{\bf t} := \left[ \prod_{v\in \VV} X_v^{t_v}\right],\quad
Z^{\bf t} := \left[ \prod_{v\in \VV} Z_v^{t_v}\right].
\end{align}
Due to the property of Weyl-ordered products, note that these elements are independent of the choice of the order of the product inside the bracket.
The element $X^{\bf t}$ (or $Z^{\bf t}$) is called the {\bf monomial}. It is well-known that $\{X^{\bf t}\mid {\bf t}\in \mathbb Z^{V}\}$ (resp.
$\{Z^{\bf t}\mid {\bf t}\in \mathbb Z^{V}\}$) is a basis of $\bT_{\omega}(Q)$ (resp.
$\bT_{\hat\omega}(Q)$), called the {\bf monomial basis}.

For any ${\bf k},{\bf t}\in\mathbb Z^V$, it is straightforward to verify that
\begin{align}\label{eq-prod-vector-q}
    X^{\bf k} X^{\bf t} = \omega^{2{\bf k}Q{\bf t}^T} X^{\bf t} X^{\bf k} \text{ and }
Z^{\bf k} Z^{\bf t} = \hat\omega^{2{\bf k}Q{\bf t}^T} Z^{\bf t} Z^{\bf k},
\end{align}
where ${\bf k}$ and ${\bf t}$ are regarded as row vectors. 

Let $\Lambda$ be a subgroup of $\mathbb Z^V$. Define 
\begin{align*}
    \mathbb T_{\hat\omega}(Q;\Lambda):=\text{span}_{\BC}\{Z^{\bf k}\mid {\bf k}\in\Lambda\}.
\end{align*}
Then $\mathbb T_{\hat\omega}(Q;\Lambda)$ is a subalgebra of $\bT_{\hat{\omega}}(Q)$.
We have the following lemma regarding the center of $\mathbb T_{\hat\omega}(Q;\Lambda)$.

\begin{lemma}\cite[Lemmas~3.1 and 3.2]{KaruoWangToAppear}\label{lem-center-torus}
(a) When $\hat \omega$ is a not a root of unity, then the center of $\mathbb T_{\hat\omega}(Q;\Lambda)$ is 
    $$\text{span}_{\BC}\{Z^{\bf k}\mid {\bf k} Q{\bf t}^T = 0\in\mathbb Z\text{ for all }{\bf t}\in\Lambda\}.$$

(b) When ${\hat{\omega}}^2$ is a root of unity of order $N''$, then the center of $\mathbb T_{\hat\omega}(Q;\Lambda)$ is 
$\mathbb T_{\hat\omega}(Q;\Lambda_{N''})$, where
    $$\Lambda_{N''}=\{{\bf k}\in\Lambda\mid {\bf k} Q{\bf t}^T = 0\in\mathbb Z_{N''}\text{ for all }{\bf t}\in\Lambda\}.$$

(c) When ${\hat{\omega}}^2$ is a root of unity of order $N''$, then 
$\mathbb T_{\hat\omega}(Q;\Lambda)$ is a free module over its center with rank
$\left|\dfrac{\Lambda}{\Lambda_{N''}}\right|.$
    
\end{lemma}

Suppose that $\hat\omega$ is a root of unity.
We denote by ${\bf irrep}$ the set of irreducible representations of $\mathbb T_{\hat\omega}(Q;\Lambda)$ (considered up to isomorphism).
Let $(\rho\colon \mathbb T_{\hat\omega}(Q;\Lambda)\rightarrow \text{End}_\BC(V))\in {\bf irrep}$. 
By Schur's lemma, for any $x\in\mathsf{Z}(\mathbb T_{\hat\omega}(Q;\Lambda))$ we have 
$\rho(x) = r_{\rho}(x)\Id_{V}$ for some $r_{\rho}(x)\in\mathbb C$, where $\mathsf{Z}(\mathbb T_{\hat\omega}(Q;\Lambda))$ is the center of $\mathbb T_{\hat\omega}(Q;\Lambda)$.
It follows that $r_{\rho}\in \text{Hom}_{\text{Alg}}(\mathsf{Z}(\mathbb T_{\hat\omega}(Q;\Lambda)),\mathbb C)$, the set of algebra homomorphisms from $\mathsf{Z}(\mathbb T_{\hat\omega}(Q;\Lambda))$ to $\BC$.

\begin{lemma}\cite[Lemma 8.4]{KaruoWangToAppear}\label{lem-irre-Azumaya}
Suppose that $\hat\omega$ is a root of unity.
    The following map is a bijection
\begin{align*}
    {\bf irrep}\rightarrow \text{Hom}_{\text{Alg}}(\mathsf{Z}(\mathbb T_{\hat\omega}(Q;\Lambda)),\mathbb C),\quad
    \rho\mapsto r_{\rho}.
\end{align*}
Moreover, the square of the dimension of every irreducible representation of $\mathbb T_{\hat\omega}(Q;\Lambda)$ equals the rank of $\mathbb T_{\hat\omega}(Q;\Lambda)$ over $\mathsf{Z}(\mathbb T_{\hat\omega}(Q;\Lambda))$.
\end{lemma}

\def\bZ{\mathbb Z}

\subsection{Quantum trace maps}\label{subsec-quantum-trace-map}

\begin{figure}[h]
    \centering
    \includegraphics[width=260pt]{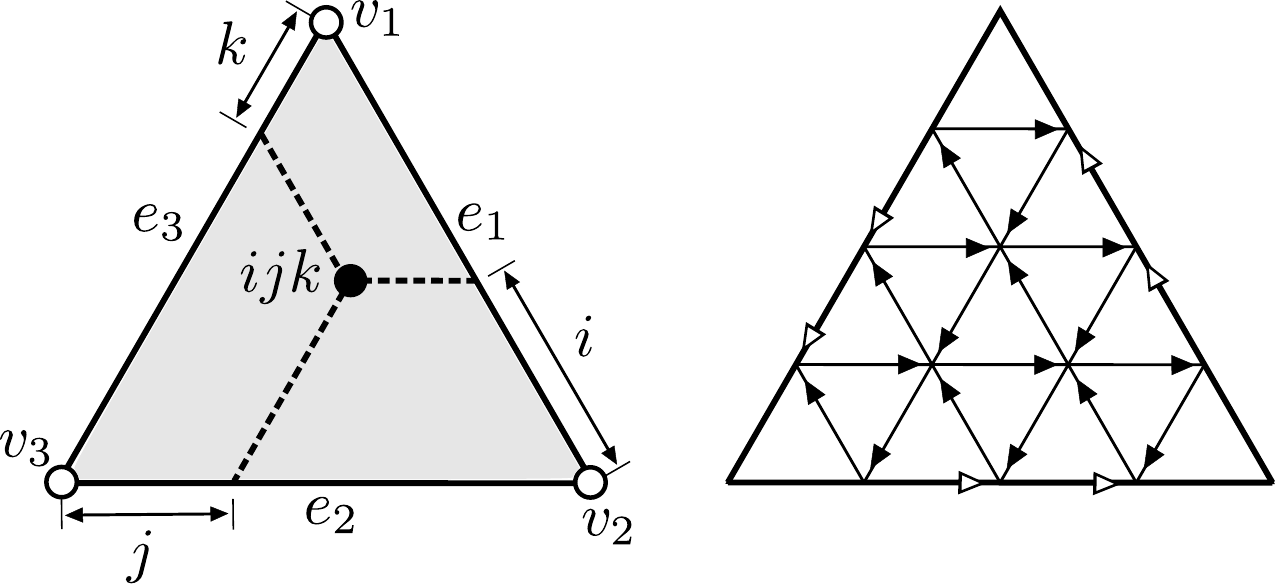}
    \caption{Barycentric coordinates $ijk$ and a $4$-triangulation with its quiver}\label{Fig;coord_ijk}
\end{figure}

In this subsection we review the ($X$-version) ${\rm SL}_n$ quantum trace map constructed in \cite{LY23}, which is an algebra homomorphism from the reduced stated $\SL$-skein algebra of a `triangulable' pb surface $\fS$ to a quantum torus algebra associated to a special quiver built from any choice of an `(ideal) triangulation' of $\fS$.

To begin with, consider barycentric coordinates for an ideal triangle $\bP_3$ so that
\begin{equation}\label{eq-coordinate-P3}
\bP_3=\{(i,j,k)\in\bR^3\mid i,j,k\geq 0,\ i+j+k=n\}\setminus\{(0,0,n),(0,n,0),(n,0,0)\}, 
\end{equation}
where $(i,j,k)$ (or $ijk$ for simplicity) are the barycentric coordinates. 
Let $v_1=(n,0,0)$, $v_2=(0,n,0)$, $v_3=(0,0,n)$. 
Let $e_i$ denote the edge on $\partial \bP_3$ whose endpoints are $v_i$ and $v_{i+1}$. 
See Figure \ref{Fig;coord_ijk}.

The \textbf{$n$-triangulation} of $\bP_3$ is the subdivision of $\bP_3$ into $n^2$ small triangles using lines $i,j,k=\text{constant integers}$. 
For the $n$-triangulation of $\bP_3$, the vertices and edges of all small triangles, except for $v_1,v_2,v_3$ and the small edges adjacent to them, form a quiver $\Gamma_{\bP_3}$.
An \textbf{arrow} is the direction of a small edge defined as follows. If a small edge $e$ is in the boundary $\partial\bP_3$ then $e$ has the counterclockwise direction of $\partial \bP_3$. If $e$ is interior then its direction is the same with that of a boundary edge of $\bP_3$ parallel to $e$. Assign weight $\frac{1}{2}$ to any boundary arrow and weight $1$ to any interior arrow.

Let $V_{\bP_3}$ be the set of all
points with integer barycentric coordinates of $\bP_3$:
\begin{align}
V_{\bP_3} = \{ijk \in \bP_3 \mid i, j, k \in \bZ\}.
\end{align}
Elements of $V_{\bP_3}$ are called \textbf{small vertices}, and small vertices on the boundary of $\bP_3$ are called the \textbf{edge vertices}. 

A pb surface $\fS$ is {\bf triangulable} if 
no component of 
$\fS$ is 
one the following cases:  
\begin{enumerate}[label={\rm (\arabic*)}]
     \item the closed surface,
     
    \item the monogon $\mathbb{P}_1$, 

    \item the bigon $\mathbb{P}_2$,

    \item once or twice punctured sphere
\end{enumerate}

A {\bf triangulation} $\lambda$ of $\fS$ is a collection of disjoint ideal arcs in $\fS$ with the following properties: (1) any two arcs in $\lambda$ are not isotopic; (2) $\lambda$ is maximal under condition (1); (3) every puncture is adjacent to at least two ideal arcs.
Our definition of triangulation excludes the so-called self-folded triangles.
We will call each ideal arc in $\lambda$ an {\bf edge} of $\lambda$.
If an edge is isotopic to a boundary component of $\fS$, we call such an edge a {\bf boundary edge}.
We use $\mathbb F_{\lambda}$ to denote the set of faces after we cut $\fS$ along all ideal arcs in $\lambda$ that are not boundary edges.
It is well-known that any triangulable surface admits a triangulation.

Suppose that $\fS$ is a triangulable pb surface with a triangulation $\lambda$.
By cutting $\fS$ along all edges not isotopic to a boundary edge, we have a disjoint union of ideal triangles. Each triangle is called a \textbf{face} of $\lambda$. Then
\begin{equation}\label{eq.glue}
\fS = \Big( \bigsqcup_{\tau\in\mathbb F_\lambda} \tau \Big) /\sim,
\end{equation}
where each face $\tau$ is regarded as a copy of $\bP_3$, and $\sim$ is the identification of edges of the faces to recover $\lambda$. 
Each face $\tau$ is characterized by a \textbf{characteristic map} $f_\tau\colon \mathbb P_3 \to \fS$, which is a homeomorphism when we restrict $f_\tau$ to the interior of $\mathbb P_3$  or the interior of each edge of $\mathbb P_3$.

An \textbf{$n$-triangulation} of $\lambda$ is a collection of $n$-triangulations of the faces $\tau$ which are compatible with the gluing $\sim$,  where compatibility means, for any edges $b$ and $b'$ glued via $\sim$, the vertices on $b$ and $b'$ are identified. Define
$$V_\lambda=\bigcup_{\tau\in\mathbb F_\lambda} V_\tau, \quad V_\tau=f_\tau(V_{\mathbb P_3}).$$
The images of the weighted quivers $\Gamma_{\mathbb P_3}$ by $f_\tau$ form a quiver $\Gamma_\lambda$ on $\fS$.
Note that when edges $b$ and $b'$ are glued, a small edge on $b$ and the corresponding small edge of $b'$ have opposite directions, i.e. the resulting arrows are of weight $0$. 

For any two $v,v'\in V_\lambda$, define
$$
a_\lambda(v,v') = \begin{cases} w \quad & \text{if there is an arrow from $v$ to $v'$ of weight $w$},\\
0 &\text{if there is no arrow from $v$ to $v'$.} 
\end{cases}$$
Let $Q_\lambda\colon V_\lambda\times V_\lambda \to \frac{1}{2}\bZ$ be the signed adjacency matrix of the weighted quiver $\Gamma_\lambda$ defined by 
\begin{equation}\label{eq-sign-matrix-ad}
Q_\lambda(v,v') = a_\lambda(v,v') - a_\lambda(v',v).
\end{equation}
Especially we use $Q_{\bP_3}$ to denote $Q_\lambda$ when $\fS=\bP_3$.  

\def\bT{\mathbb T}

\def\tr{{\rm tr}_{\lambda}}
\def\X{\mathcal X_{\omega}(\fS,\lambda)}
\def\Y{\mathcal Y_{q}(\fS,\lambda)}
\def\bX{\mathcal Z_{\hat\omega}^{\rm bl}(\fS,\lambda)}
\def\bXe{\mathcal Z_{\hat\eta}^{\rm bl}(\fS,\lambda)}
\def\bk{{\bf k}}
\def\V{V_\lambda}

The \textbf{Fock-Goncharov algebra} is the quantum torus algebra associated to $Q_\lambda$, i.e.:
\begin{equation*}
\mathcal{X}_{\omega}(\fS,\lambda)= \bT_\omega(Q_\lambda) = \BC \langle
X_v^{\pm 1}, v \in V_\lambda \rangle / (
X_v 
X_{v'}= \omega^{\, 2 Q_\lambda(v,v')} 
X_{v'} 
X_v \text{ for } v,v'\in V_\lambda ).
\end{equation*}
The \textbf{$n$-th root Fock-Goncharov algebra} is the quantum torus algebra defined as follows:
\begin{equation*}
\mathcal{Z}_{\hat\omega}(\fS,\lambda)= \bT_{\hat\omega}(Q_\lambda) = \BC \langle
Z_v^{\pm 1}, v \in V_\lambda \rangle / (
Z_v 
Z_{v'}= \hat\omega^{\, 2 Q_\lambda(v,v')} 
Z_{v'} 
Z_v \text{ for } v,v'\in V_\lambda ).
\end{equation*}

Recall that $\omega= \hat\omega^{n^2}$. There is an algebra embedding
from 
$\mathcal{X}_{\omega}(\fS,\lambda)$ to 
$\mathcal{Z}_{\hat\omega}(\fS,\lambda)$ given by 
$$X_v \mapsto Z_v^n$$
for $v\in \V$.
We will regard 
$\mathcal{X}_{\omega}(\fS,\lambda)$ as a subalgebra of 
$\mathcal{Z}_{\hat\omega}(\fS,\lambda)$.

Let ${\bf k}_i\colon V_{\bP_3} \to\bZ\ (i=1,2,3)$ be the functions defined by
\begin{equation}
{\bf k}_1(ijk)=i,\quad {\bf k}_2(ijk)=j,\quad {\bf k}_3(ijk)=k.\label{def:proj}
\end{equation} 
Let $\mathcal B_{\bP_3}\subset\bZ^{V_{\bP_3}}$ be the subgroup generated by ${\bf k}_1,{\bf k}_2,{\bf k}_3$ and $(n\bZ)^{V_{\bP_3}}$. Elements in $\mathcal{B}_{\bP_3}$ are called \textbf{balanced}. 

A vector $\bk\in\bZ^{\V}$ is \textbf{balanced} if its pullback 
\begin{align}\label{eq-pull-back}
    {\bf k}_\tau:=f_\tau^\ast\bk
\end{align}
to ${\bP_3}$ is balanced for every face of $\lambda$, where for every face $\tau$ and its characteristic map  $f_\tau\colon\mathbb P_3\to\fS$, the pullback $f_\tau^\ast\bk$ is a vector $\V\to\bZ$ given by $f_\tau^\ast\bk(v)=\bk(f_\tau(v))$. 
Let $\mathcal B_\lambda$ denote the subgroup of $\bZ^{\V}$ generated by all the balanced vectors.

The proof of \cite[Lemma 3.11]{KimWang} implies
the following:
\begin{lemma}\label{lem-balanced-property}
    We have $$\mathcal B_\lambda\subset\{{\bf k}\in \mathbb Z^{V_\lambda}\mid {\bf k}Q_\lambda\in (n\mathbb Z)^{V_\lambda}\}.$$
\end{lemma}

\def\lt{\text{lt}}

The \textbf{balanced Fock-Goncharov algebra} \cite{LY23} is the monomial subalgebra
\begin{align}\label{def-balanced}
    \bX=\text{span}_{\BC}\{Z^\bk\mid \bk\in\mathcal B_\lambda\}\subset
\mathcal{Z}_{\hat\omega}(\fS,\lambda).
\end{align}
It is easy to verify that $
\mathcal{X}_{\omega}(\fS,\lambda) \subset \bX$, and that $\{Z^\bk\mid \bk\in\mathcal B_\lambda\}$ is a basis of $\bX$.

Suppose $v,v'\in V_\lambda$, If $v$ and $v'$ are not contained in the same boundary edge, define 
\begin{align}\label{def-Hn-non-boundary}
    H_\lambda(v,v')=-Q_\lambda(v,v').
\end{align}
If $v$ and $v'$ are contained in the same boundary edge, define 
\begin{align}\label{def-Hn-boundary}
    H_\lambda(v,v')=\begin{cases}
    1 & \text{when $v=v'$},\\
    -1 & \text{when there is an arrow from $v$ to $v'$},\\
    0 & \text{otherwise.}
\end{cases}.
\end{align}

The following lemma provides a necessary and sufficient condition for a vector to be balanced in the case where $\fS$ has no interior punctures.

\begin{lemma}\cite[Lemma 11.9 and Proposition 11.10]{LY23}\label{lem-matrix-HK}
Suppose that $\fS$ contains no interior punctures.
We have the following:
\begin{enumerate}[label={\rm (\alph*)}]  
    \item There exists a unique matrix  $K_\lambda\colon V_\lambda\times V_\lambda\to\mathbb Z$
    such that $H_\lambda K_\lambda = nI$.

    \item ${\bf k}\in\mathbb Z^{V_\lambda}$ is balanced if and only if ${\bf k}={\bf c}K_\lambda$ for some ${\bf c}\in\mathbb Z^{V_\lambda}$.

     \item ${\bf k}\in\mathbb Z^{V_\lambda}$ is balanced if and only if ${\bf k}H_\lambda\in (n\mathbb Z)^{V_\lambda}$.
\end{enumerate}
   
\end{lemma}

\begin{theorem}[\cite{BW11,LY22} for $n=2$; \cite{Kim20} for $n=3$; \cite{LY23} for general $n$]\label{thm.quantum_trace}
    Suppose that $\fS$ is a triangulable pb surface with a triangulation $\lambda$. Then the following hold:

    \begin{enumerate}[label={\rm (\alph*)}]
    \item There is an algebra homomorphism
    $$\tr \colon \rdS\rightarrow
    \mathcal{Z}_{\hat\omega}(\fS,\lambda),$$
    called the {\bf (reduced) quantum trace map},
    such that $\im\tr\subset \bX.$

    \item  
    $\tr$ is injective when $n=2$ or when $n=3$.

    \item 
    $\tr$ is injective when $\fS$ is a polygon.
    \end{enumerate}
\end{theorem}

\subsection{Compatibility between the splitting homomorphism and the quantum trace map}

Let $\fS$ be a triangulable pb surface with a triangulation $\lambda$, and let $e$ be a non-boundary edge of $\lambda$.  
As in \S\ref{sub-splitting}, denote by $\cut_e(\fS)$ the pb surface obtained from $\fS$ by cutting along $e$, and let
\begin{align}\label{def-pr-cut-fs}
\pr_e \colon \cut_e(\fS) \to \fS
\end{align}
be the projection map, so that $\pr_e^{-1}(e)$ consists of two ideal arcs $e',e''$ of $\cut_e(\fS)$.  
In \S\ref{sub-splitting} we recalled the splitting homomorphism for reduced stated ${\rm SL}_n$-skein algebras,
\[
\mathbb{S}_e : \rS^{\rm st}(\fS) \to \rS^{\rm st}(\cut_e(\fS)).
\]

We now introduce an analogous “splitting homomorphism” for the ($n$-th root) Fock–Goncharov algebras.  
Note that 
\begin{align}\label{induced-ideal-tri}
    \lambda_e = (\lambda\setminus\{e\}) \cup \{e',e''\}
\end{align}
is a triangulation of $\cut_e(\fS)$, which we call the triangulation induced from $\lambda$.  
For a small vertex $v \in V_\lambda$ lying on $e$, its preimage under $\pr_e$ is $\pr^{-1}_{e}(v) = \{v',v''\}$, where $v',v''$ are small vertices of $\lambda_e$.  

There is an algebra embedding
\begin{align}\label{eq-splitting-torus}
    \mathcal S_e : \mathcal{Z}_{\hat\omega}(\fS,\lambda)
    \longrightarrow \mathcal{Z}_{\hat\omega}(\cut_e(\fS),\lambda_e)
\end{align}
defined on the generators $Z_v$ for $v \in V_\lambda$ by
\begin{align}\label{cutting_homomorphism_for_Z_omega}
    \mathcal S_e(Z_v) =
    \begin{cases}
        Z_v, & v \notin e, \\[6pt]
        [Z_{v'} Z_{v''}], & v \in e \text{ with } \pr^{-1}_{e}(v) = \{v',v''\},
    \end{cases}
\end{align}
where $[\cdot]$ denotes the Weyl-ordered product \eqref{Weyl_ordering-Z-1}.  
Since self-folded triangles are excluded in this paper, $Z_{v'}$ commutes with $Z_{v''}$ in the second case, so in fact
$\mathcal{S}_e(Z_v) = Z_{v'} Z_{v''}.$

The following statement is about the compatibility of the quantum trace map ${\rm tr}_\lambda 
   $ (Theorem \ref{thm.quantum_trace}) with the splitting homomorphisms.

\begin{theorem}[\cite{LY23}]\label{thm-trace-cut}
    The following diagram commutes
    \begin{equation*}
\begin{tikzcd}
\rdS \arrow[r, "\mathbb S_e"]
\arrow[d, "\tr"]  
&  \overline{\cS}_{\hat\omega}^{\rm st}(\cut_e(\fS)) \arrow[d, "{\rm tr}_{\lambda_e}"] \\
 \mathcal{Z}_{\hat\omega}(\fS,\lambda)
 \arrow[r, "\mathcal S_e"] 
&  
\mathcal{Z}_{\hat\omega}(\cut_e(\fS),\lambda_e),
\end{tikzcd}
\end{equation*}
    where $\mathbb S_e$ is the splitting homomorphism defined in \S\ref{sub-splitting}.
\end{theorem}

\def\Sx{\cS_{\bar\omega}}
\def\An{\mathsf{A}}
\def\BN{\mathbb N}
\def\BC{\mathbb C}
\def\W{x_1,\cdots,x_{n-1}}
\def\BX{\mathbb C[X]}
\def\BCX{\BC[x_1, \dots, x_n]}

\section{Frobenius maps of the projected ${\rm SL}_n$-skein algebras}
In this section we assume that $\fS$ is a punctured surface, i.e. $\partial\fS=\emptyset$.  

We recall the Frobenius homomorphism of \cite{KLW} (Theorem~\ref{thm-Fro}), defined for the projected $\SL$-skein algebra (see \eqref{def-projected-skein}).  
In \S\ref{sec-rep-sln} we will apply this Frobenius homomorphism to study the representation theory of the projected $\SL$-skein algebra at roots of unity.


\subsection{Elementary polynomials}

We use the notation
\[
\BX := \BC[x_1^{\pm1}, \dots, x_n^{\pm1}] \big/ (x_1x_2\cdots x_n=1).
\]
The permutation group $S_n$ acts naturally on $\BC[x_1^{\pm1}, \dots, x_n^{\pm1}]$ by
\[
\sigma(x_i) = x_{\sigma(i)} \qquad (1 \leq i \leq n,\ \sigma \in S_n),
\]
and this action descends to an action on $\BX$.  
We denote by $\BX^{S_n}$ the subalgebra of $\BX$ consisting of $S_n$-invariant elements.  

For each $1 \leq i \leq n-1$, define the $i$-th elementary symmetric polynomial
\begin{align}\label{elementary_symmetric_polynomial}
e_i = \sum_{1 \leq j_1 < \cdots < j_i \leq n} x_{j_1}\cdots x_{j_i} \in \BX.
\end{align} 
It is well known that 
$\BX^{S_n} = \BC[e_1, \dots, e_{n-1}]$ \cite{Macdonald}.

For $k \in \BN$, the $\BC$-algebra endomorphism of $\BC[x_1, \dots, x_n]$ defined by 
\[
x_i \mapsto x_i^k \qquad (i=1,\dots,n)
\]
descends to a $\BC$-algebra endomorphism
\begin{align}\label{Adams}
\Psi_k : \BX^{S_n} \to \BX^{S_n},
\end{align}
called the {\bf $k$-th Adams operation} \cite{BtD}.  
Since $\BX^{S_n} = \BC[e_1, \dots, e_{n-1}]$, for each $k \in \BN$ and $1 \leq i \leq n-1$, there exists a unique polynomial  
\[
\bar P_{k,i} \in \BZ[y_1,\dots,y_{n-1}]
\]
such that
\begin{equation}\label{eq-def-barP}
\Psi_k(e_i) = \bar P_{k,i}(e_1,\dots,e_{n-1}).
\end{equation}
This polynomial $\bar P_{k,i}$ is called the {\em reduced power elementary polynomial} in \cite{BH23}.  

Now let $c_1,\dots,c_n \in \BC$ with $c_1\cdots c_n = 1$.  
For each $1 \leq k \leq n$, define
\[
e_k(c_1,\dots,c_n) := \sum_{1 \leq i_1 < \cdots < i_k \leq n} c_{i_1}\cdots c_{i_k}.
\]

The following lemma will be used in \S\ref{sec-rep-sln}.
\begin{lemma}\label{lem-solu}
Let $s_i\in\mathbb C$ for $1\leq i\leq n$, with $s_1\cdots s_{n}=1$. Let $m$ be a positive integer.  
Consider the system of equations
\begin{equation}\label{eq-system}
        \bar P_{m,k}(x_1,\dots,x_{n-1}) \;=\; e_k(s_1,\dots,s_{n}), 
        \quad 1\leq k\leq n-1.
\end{equation}
Then the set of solutions $(x_1,\dots,x_{n-1})$ is given by
\begin{align*}
    \Bigl\{(e_1(c_1,\dots,c_{n}),\dots,e_{n-1}(c_1,\dots,c_{n})) 
    \;\Bigm|\;
    c_i\in\mathbb C,\; c_i^m = s_i \text{ for } 1\leq i\leq n, 
    \text{and } c_1\cdots c_{n}=1 \Bigr\}.
\end{align*}
\end{lemma}

\begin{proof}
Let $(d_1,\dots,d_{n-1})\in\mathbb C^{n-1}$ and $d_n=1$. Suppose that 
$c_1,\dots,c_n\in\mathbb C$ are the $n$ roots of the polynomial
\(
x^n+\sum_{k=1}^{n}(-1)^k d_k x^{n-k}=0.
\)
Then
\[
e_k(c_1,\dots,c_{n})=d_k \quad (1\leq k\leq n).
\]
Setting $C=\mathrm{diag}(c_1,\dots,c_{n})$, we have $C\in\SL(\mathbb C)$.  

Now consider
\begin{align*}
\det(xI-C^m)
    &= \prod_{i=1}^n (x-c_i^m) \\
    &= x^n + \sum_{k=1}^{n-1} (-1)^k e_k(c_1^m,\dots,c_n^m)\,x^{n-k} + (-1)^n,
\end{align*}
where $I$ is the identity matrix.
By the definition of $\bar P_{m,k}$ (see \eqref{eq-def-barP}), this can be rewritten as
\begin{align*}
\det(xI-C^m)
    &= x^n + \sum_{k=1}^{n-1} (-1)^k \bar P_{m,k}\bigl(e_1(c_1,\dots,c_n),\dots,e_{n-1}(c_1,\dots,c_n)\bigr)\, x^{n-k} + (-1)^n \\
    &= x^n + \sum_{k=1}^{n-1} (-1)^k \bar P_{m,k}(d_1,\dots,d_{n-1})\, x^{n-k} + (-1)^n.
\end{align*}

Hence $(x_1,\dots,x_{n-1})=(d_1,\dots,d_{n-1})$ is a solution of \eqref{eq-system} if and only if
\[
\det(xI-C^m) 
= x^n + \sum_{k=1}^{n-1} (-1)^k e_k(s_1,\dots,s_n)\, x^{n-k} + (-1)^n 
= (x-s_1)\cdots(x-s_n).
\]
Since $\det(xI-C^m)=\prod_{i=1}^n(x-c_i^m)$, this is equivalent to the condition
\[
\{s_1,\dots,s_n\}=\{c_1^m,\dots,c_n^m\}.
\]
Finally, note that $(d_1,\dots,d_{n-1})$ depends only on the elementary symmetric functions of the $c_i$, and hence is independent of how we index $c_i$.  
This completes the proof.

\end{proof}

\subsection{Threading along a knot}

We use $\mathsf A$ to denote the twice-punctured sphere, which is diffeomorphic to a once-punctured open disc; see Figure \ref{fig-twice-punctured-sphere}(a).
There is an $n$-web $\alpha$ in $\mathsf{A}\times (-1,1)$ as shown in Figure \ref{fig-twice-punctured-sphere}(b).

\begin{figure}
\centering
\begin{tikzpicture}[baseline=0cm,every node/.style={inner sep=2pt}]
\draw[wall,fill=gray!20,dotted] (0,0) circle[radius=1.3];
    \draw[fill=white] (0,0) circle[radius=0.1];
\path (0,-1.8)node{(a)};

\begin{scope}[xshift=3.5cm]
\draw[wall,fill=gray!20,dotted] (0,0) circle[radius=1.3];
    \draw[fill=white] (0,0) circle[radius=0.1];
    \draw[decoration={markings, mark=at position 0.25 with {\arrow{>}}}, postaction={decorate}] 
        (0,0) circle[radius=0.5];
    \path (0,-1.8)node{(b)};
\end{scope}

\begin{scope}[xshift=7cm]
[baseline=0cm,every node/.style={inner sep=2pt}]
\draw[wall,fill=gray!20,dotted] (0,0) circle[radius=1.3];
    \draw[fill=white] (0,0) circle[radius=0.1];
    \draw[postaction={decorate, decoration={markings, mark=at position 0.5 with {\arrow{>}}}}] (1,0) parabola bend (0,1) (-1,0);
    \draw[postaction={decorate, decoration={markings, mark=at position 0.5 with {\arrow{>}}}}] 
        (1,0) parabola bend (0,0.25) (-1,0);
    \node at (0.15,0.7) {$\vdots$};
    \node at (-0.15,0.65) {$k$};
    \draw[postaction={decorate, decoration={markings, mark=at position 0.5 with {\arrow{>}}}}] (1,0) parabola bend (0,-1) (-1,0);
    \draw[postaction={decorate, decoration={markings, mark=at position 0.5 with {\arrow{>}}}}] 
        (1,0) parabola bend (0,-0.25) (-1,0);
    \node at (0,-0.55) {$n-k$};
\path (0,-1.8)node{(c)};
\end{scope}
\end{tikzpicture}
\caption{(a) The twice punctured sphere $\mathsf A$. (b) A knot in $\mathsf A$. (c) The $n$-web diagram $\alpha_k'$ in $\mathsf{A}$.}\label{fig-twice-punctured-sphere}
\end{figure}
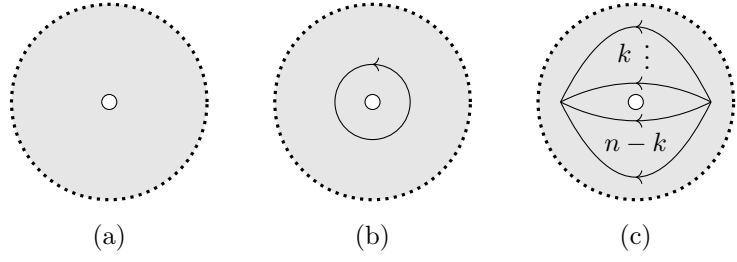

\def\SqA{\cS_{\bar\omega}(\mathsf{A})}
\def\An{\mathsf{A}}
\def\al{\alpha}


Assume that $[n]_{\omega}!:=\prod_{1\leq k\leq n}[k]_{\omega}\neq 0$, where $[k]_{\omega}
=(\omega^n-\omega^{-n})/(\omega-\omega^{-1})$.
For each $1\leq k\leq n-1$, let $\al'_k\in \SqA$ represented by the $n$-web in Figure \ref{fig-twice-punctured-sphere}(c), and let
\begin{align}\label{eq-alphak-loop}
	\left[\alpha_k\right]_{\bar\omega} :=
		(-1)^{\binom{k}{2}+\binom{n-k}{2}}\frac{1}{[n-k]_\omega![k]_\omega!}\alpha_k'  \in \cS_{\bar\omega}(\An)
\end{align}

\begin{remark}
    It is shown in \cite{KLW,cremaschi2024monomial,queffelec2018sutured} that
    $$\cS_{\bar\omega}(\An)=\mathbb C\left[\left[\alpha_1\right]_{\bar\omega},\cdots,\left[\alpha_{n-1}\right]_{\bar\omega}\right].$$
\end{remark}

An $n$-web $\beta$ in $\fS\times (-1,1)$ is called a {\bf framed oriented knot} if $\beta$ consists of a single oriented circle. 
Then there exists an embedding $f_\beta\colon \mathsf{A}\times (-1,1)\rightarrow \fS\times (-1,1)$ into the interior of $\fS\times (-1,1)$ such that 
$f_\beta(\alpha) = \beta$.
The embedding $f_{\beta}$ induces a $\mathbb C$-linear map
\begin{align}\label{def-embedding-D}
    (f_\beta)_{*}\colon \cS_{\bar\omega}(\mathsf{A})\rightarrow\Sx(\fS),
    \quad W\mapsto f_\beta(W)
    \text{ for any $n$-web $W$ in $\mathsf{A}\times (-1,1)$.}
\end{align}

\begin{definition}\label{def.threading_of_element}
Let $\fS$ be a pb surface, and let $\beta$ be a framed oriented knot in $\fS\times (-1,1)$.
For $z\in \cS_{\bar\omega}(\An)$ define the {\bf result of threading $z$ along $\beta$} as
$$
\gamma *_{\bar\omega} z := (f_\beta)_*(z) \in \cS_{\bar\omega}(\fS).
$$

\def\BC{\mathbb C}

For each $1\leq k\leq n-1$, define
$$\left[\beta_k\right]_{\bar\omega}:=\beta\ast_{\bar\omega} [\alpha_k]_{\bar\omega}\in \cS_{\bar\omega}(\fS).$$

Let $
 P\in \BC[y_1, \dots, y_{n-1}]$ and $
 \beta$ be a framed oriented 
 knot in $\fS\times (-1,1)$. 
 Define the {\bf threading of $
 P$ along $
 \beta$} by
$$ 
\left[\beta^{[P]} \right]_{\bar\omega}:= \beta *_{\bar\omega} P(\left[\alpha_1\right]_{\bar\omega},\cdots,\left[\alpha_{n-1}\right]_{\bar\omega}) \in \cS_{\bar\omega}(\fS).
$$
\end{definition}

Note that $$\left[\beta^{[P]}\right]_{\bar\omega} = P(\left[\beta_1\right]_{\bar\omega},\cdots,\left[\beta_{n-1}\right]_{\bar\omega})$$
when $\beta$ is represented by a crossingless oriented circle diagram in $\fS$.

\def\cK{\mathcal K}
\def\BC{\mathbb C}
\def\Sx{\cS_{\bar\omega}}

\subsection{Projected  ${\rm SL}_n$-skein algebras}
Let $\fS$ be a connected punctured surface with at least one puncture, and let
$p$ be a puncture of $\fS$. 
Let $c_p$ be a trivial ideal arc at $p$. This means the two endpoints of $c_p$ are both $p$ and $c_p$ bounds an embedded monogon $D_p$; see Figure \ref{fig_cp}. By removing the interior of $D_p$ from $\fS$, we get 
a new pb surface $\fS'$, where $c_p$ now becomes a boundary edge.  
With the identification $\fS=\fS\setminus D_p$ 
there is a unique morphism $f: \fS \embed \fS'$, such that $f$ is the identity outside a small neighborhood of $D_p$. One can imagine that $f$ `enlarges' the puncture $p$.

For a non-zero $\hat\omega^{\frac{1}{2}}\in \BC$, it is straightforward to verify that the $\mathbb C$-linear map 
$$f_*: \cS_{\bar \omega}(\fS) \to \cS_{\hat\omega}^{\rm st}(\fS'),
\quad W\mapsto f(W)
    \text{ for any $n$-web diagram $W$ in $\fS$}$$
is an algebra homomorphism.
Let $\cK_{\hat\omega, p}$ be the kernel of $f_{*}$.

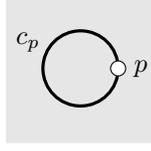
\begin{figure}
	\centering
	\begin{tikzpicture}
\draw[draw=white,fill=gray!20]  (-1,-1) rectangle (1,1);
    \draw[wall,fill=gray!20] (0,0) circle[radius=0.5];
    \draw[fill=white] (0.5,0) circle[radius=0.1];
    \node at (0.8,0) {$p$};
    \node [above] at (-0.7,0.1) {$c_p$};
\end{tikzpicture}
	\caption{The arc $c_p$ and the enclosed monogon $D_p$.}\label{fig_cp}
\end{figure}

\def\hxi{\omega}
\def\Cut{\mathsf{Cut}}
\def\SxS{\Sx(\fS)}
\def\SxsS{\Sx^{\ast}(\fS)}

In \cite[Section 9]{LS21}, it is proved that  $\cK_{\hat{\omega},p}$ does not depend on the choice of a puncture $p$, and we denote it by $\cK_{\hat{\omega},\fS}$. Define the {\bf projected ${\rm SL}_n$-skein algebra} by 
\begin{align}\label{def-projected-skein}
    \SxsS:= \SxS/\cK_{\hat{\omega},\fS}.
\end{align}
For $\beta \in \SxS$, we use the same symbol $\beta$ to denote its image under the projection $\SxS \to \SxsS$.

Currently, the projected ${\rm SL}_n$-skein algebra is defined only for connected punctured surfaces with at least one puncture. 
Let $\fS$ be a punctured surface with connected components $\fS_1,\dots,\fS_k$, each containing at least one puncture. 
We then define
\[
\Sx^{\ast}(\fS) := \bigotimes_{i=1}^k \Sx^{\ast}(\fS_i).
\]

\begin{remark}
   In \cite{KLW}, the projected ${\rm SL}_n$-skein algebra is defined for all punctured surfaces. 
Since our focus excludes closed surfaces, we do not introduce the projected  ${\rm SL}_n$-skein algebra in that case.

It was conjectured in \cite{LS21} that the projected ${\rm SL}_n$-skein algebra coincides with the ${\rm SL}_n$-skein algebra.  
This conjecture was confirmed in \cite{le2018triangular,higgins2020triangular} for $n=2,3$, and in
\cite{wang2024TQFT} for general $n$ in the case $\omega^{n}=1$.

\end{remark}

\def\bSS{\overline{\cS}_{\bar\omega}(\fS)}

Note that $\overline{\cS}_{\hat\omega}^{\rm st}(\fS) = \cS_{\bar\omega}(\fS)$ when $\fS$ has empty boundary.  
The following result shows that the quantum trace map for $\cS_{\bar\omega}(\fS)$ naturally induces a quantum trace map for $\cS_{\bar\omega}^{*}(\fS)$.

\begin{lemma}\cite[Lemma 8.9]{KLW}
    Assume that $\fS$ is a triangulable punctured surface with a triangulation $\lambda$.
    The quantum trace map $\tr\colon
    \cS_{\bar \omega}(\fS)\rightarrow
    \mathcal Z_{\hat \omega}(\fS,\lambda)$ of Theorem \ref{thm.quantum_trace}(a)
    induces a quantum trace map for the projected $\SL$-skein algebra
    \begin{align}
        \label{eq.X-quantum_trace_projected}
        \tr\colon
    \cS_{\bar\omega}^{*}(\fS)\rightarrow
    \mathcal Z_{\hat\omega}(\fS,\lambda).
    \end{align}
\end{lemma}

\def\PT{\Phi^{\mathbb T}}


\def\homega{\hat{\omega}}
\def\heta{\hat{\eta}}

Let $\fS$ be a triangulable punctured surface with a triangulation $\lambda$.
When we have the assumption \Rlabel{1},
it is straightforward to verify that there exists an algebra embedding 
\begin{align*}
    \Phi^{\mathbb T}\colon \mathcal{Z}_{\hat\eta}(\fS,\lambda)\rightarrow
    \mathcal{Z}_{\homega}(\fS,\lambda),\quad
    Z^{\bf k}\mapsto Z^{N{\bf k}}\text{ for ${\bf k}\in\mathbb Z^{V_\lambda}.$}
\end{align*}
Obviously, this algebra embedding induces an algebra embedding
\begin{align}\label{eq-def-F}
    \Phi^{\mathbb T}\colon \mathcal{Z}_{\heta}^{\rm bl}(\fS,\lambda)\rightarrow
    \mathcal{Z}^{\rm bl}_{\homega}(\fS,\lambda).
\end{align}

\begin{theorem}
[\cite{KLW}]\label{thm-Fro}
     Let $\fS$ be a punctured surface such that each connected component contains at least one puncture.
    Suppose that we have the assumption \Rlabel{1} and 
    $[n]_\omega!\neq 0$.
     There exists a unique algebra homomorphism 
     $$\Phi\colon \cS_{\bar \eta}^{\ast}(\fS)\rightarrow \cS_{\bar\omega}^{\ast}(\fS)\quad\text{(called the {\bf Frobenius homomorphism})}$$
     with the following properties:

     \begin{enumerate}[label={\rm (\alph*)}]
 %
         \item If $\beta$ is a framed oriented knot in $\fS\times (-1,1)$, then
         $$\Phi(\left[\beta_k\right]_{\bar\eta}) = \left[\beta^{[\bar P_{N,k}]}\right]_{\bar\omega}\in\cS_{\bar\omega}^{*}(\fS),$$
         where $1\leq k\leq n-1$.

\item If we have the assumption \Rlabel{2}, then 
$\im\Phi\subset\mathsf{Z}(\cS_{\bar\omega}^{*}(\fS))$.

\item Suppose that $\fS$ is a triangulable punctured surface with a triangulation $\lambda$. Then
$\Phi$
is compatible with $\PT$ \eqref{eq-def-F} for quantum tori via 
the quantum traces for projected skein algebras in \eqref{eq.X-quantum_trace_projected}, meaning 
\begin{align*}
\tr \circ 
\Phi = \PT \circ \tr. 
\end{align*}
     \end{enumerate}
\end{theorem}

\begin{remark}
    For $n=2$, the Frobenius homomorphism was constructed in \cite{BW16}, and for $n=3$ with $d=1$, it was constructed in \cite{higgins2025miraculous}.
\end{remark}

\section{The center and the rank of the balanced Fock-Goncharov algebra}\label{sec-rank-center}

It is well known that irreducible representations of a quantum torus at roots of unity are uniquely determined by the action of its center~\cite{KaruoWangToAppear} (see Lemma~\ref{lem-irre-Azumaya}). 
Motivated by this fact, we determine the center of the balanced Fock–Goncharov algebra: in the generic case, the description is given in Corollary~\ref{cor-center-generic}, while in the root of unity case it is provided in Theorem~\ref{thm-center-balanced-root-of-unity}.  
In Theorem~\ref{thm-rank-Z}, we compute the rank of the balanced Fock–Goncharov algebra over its center, which equals the dimension of each of its irreducible representations.  
Theorems~\ref{thm-center-balanced-root-of-unity} and \ref{thm-rank-Z} are the two central and technically most substantial results of this paper.  
In \S\ref{sec-rep-torus-b}, we apply these two theorems to classify irreducible representations of the balanced Fock–Goncharov algebra.

\subsection{The center of the balanced Fock-Goncharov algebra when $\omega$ is not a root of unity}\label{sec-sub-center-bal-generic}
We begin this subsection by introducing the central elements of the balanced Fock–Goncharov algebra in the case where $\omega$ is not a root of unity. For each puncture of $\fS$, there are $n-1$ associated central elements. Moreover, we will see that the image of an oriented loop around a puncture $p$ under the quantum trace map lies in the subalgebra generated by these $n-1$ central elements (Proposition \ref{lem-image-loop-trace}).

Suppose that $p$ is a puncture of an ideal triangle $\tau$ (i.e., $\tau=\mathbb P_3$). Recall that we defined a quiver $\Gamma_{\tau}$ with vertex set $V_{\tau}$, see Figure \ref{Fig;coord_ijk}.
As in \eqref{eq-coordinate-P3}, we introduced  barycentric coordinates for $\tau=\mathbb P_3$. Then the coordinate of $p$ is $v_1=(n,0,0)$, $v_2=(0,n,0)$, or $v_3=(0,0,n)$.
For $1\leq i\leq n-1$, we define 
${\bf a}(\tau,p,i)=(a_v)_{v\in V_\tau}\in \mathbb Z^{V_\tau}$ as the following:
Suppose that $p=v_j$ for some $j\in\{1,2,3\}$.
Then
\begin{align*}
    a_v=\begin{cases}
        1 &
        {\bf k}_j(v) = n-i,\\
        0 & \text{otherwise},
    \end{cases}
\end{align*}
where ${\bf k}_j$ is defined in \eqref{def:proj}.
It is easy to check that
\begin{align}\label{eq-relation-k-a-p}
    \sum_{1\leq t\leq n-1} (n-t){\bf a}(\tau,p,t) = {\bf k}_j.
\end{align}

For each $1\leq i\leq n-1$, define 
 \begin{align}\label{eq-ba-relation-k}
     {\bf b}(\tau,p,i):=
    \left(\sum_{1\leq t\leq n-i-1} n(t+i-n){\bf a}(\tau,p,t) \right) + (n-i){\bf k}_j\in\mathcal B_\tau,
 \end{align}
 where $\mathcal B_\tau\subset \mathbb Z^{V_\tau}$ is the balanced part.
For each $1\leq i\leq n$, define
\begin{align}\label{def-vector-c-tau}
    {\bf c}(\tau,p,i):=
    \sum_{1\leq t\leq i-1}{\bf a}(\tau,p,t).
\end{align}
Then, for each $1\leq i\leq n$, we have
\begin{align}\label{def-vector-c-tau1}
\sum_{1\leq t\leq i}
{\bf c}(\tau,p,i)= 
    \sum_{1\leq t\leq i-1}(i-t){\bf a}(\tau,p,t).
\end{align}
 Note that ${\bf c}(\tau,p,1)={\bf 0}$. Although its value is zero, this notation remains useful for later reference.

We will show that ${\bf b}(\tau,p,i)$ (resp. ${\bf a}(\tau,p,i)$), for $1\leq i\leq n-1$ and $p\in{v_1,v_2,v_3}$, are $\mathbb Z$-linearly independent.
Before giving the proof, we first establish the following result.

\begin{lemma}\label{lem-ijk-sum-n}
   Suppose $t_1,\dots,t_{n-1}$ are integers satisfying 
$t_i+t_j+t_k=0  \text{ whenever } i+j+k=n,$
and $t_s+t_{n-s}=0  \text{ for all } 1\leq s\leq n-1.$
Then $t_s=0 \text{ for all } 1\leq s\leq n-1$.
\end{lemma}
\begin{proof}

    {\bf Case 1} when $n$ is even:
    Since $t_{\frac{n}{2}}+t_{\frac{n}{2}}=0$, we have $t_{\frac{n}{2}}=0$.
    We have $2t_1+t_{n-2} = t_2+t_{n-2}=0$, which shows that $t_2=2t_1$.
    We have  $t_1+t_{2}+t_{n-3} = t_3+t_{n-3}=0$, which shows that $t_3=t_2+t_1=3t_1$.
    Inductively, we have 
    $t_s=st_1$ for $1\leq s\leq \frac{n}{2}-1$.
    For any $1\leq s\leq \frac{n}{2}-1$, we have 
    $t_{n-s} + t_s=0$, which implies that 
    $t_{n-s}=-st_1$.
    We have $t_1+t_{\frac{n}{2}-1} + t_{\frac{n}{2}} =t_1+t_{\frac{n}{2}-1}= 0$ and 
    $t_{\frac{n}{2}-1}+t_{\frac{n}{2}+1}=0$, which shows that $t_1=t_{\frac{n}{2}+1}$. Meanwhile, we have $t_{\frac{n}{2}+1}=-(\frac{n}{2}-1)t_1$. This shows that $t_1=0$ and 
    $t_s=0$ for $1\leq s\leq n-1$.

    {\bf Case 2} when $n$ is odd: As in the case 1, we can show that $t_s=st_1$ and $t_{n-s}=-st_1$ for $1\leq s\leq \frac{n-1}{2}$.
    Since $2t_{\frac{n-1}{2}} + t_1= nt_1=0$, then $t_1=0$. This shows that $t_1=0$ and 
    $t_s=0$ for $1\leq s\leq n-1$.
\end{proof}

\begin{lemma}\label{lem-independence-P3}
Let $\tau$ be an ideal triangle, then
we have
    \begin{enumerate}[label={\rm (\alph*)}]
        \item\label{lem-independence-P3-a}  ${\bf a}(\tau,p,i)$, for $1\leq i\leq n-1$ and $p\in\{v_1,v_2,v_3\}$, are $\mathbb Z$-linearly independent;

        \item  ${\bf b}(\tau,p,i)$, for $1\leq i\leq n-1$ and $p\in\{v_1,v_2,v_3\}$, are $\mathbb Z$-linearly independent.
    \end{enumerate}
\end{lemma}
\begin{proof}
    (a) Suppose $\sum_{j=1,2,3}\sum_{1\leq i\leq n-1} k_{j,i}{\bf a}(\tau,v_j,i)=0$, where $k_{j,i}\in\mathbb Z$ for $j=1,2,3$ and $1\leq i\leq n-1$.
    Recall that every vertex $v_0\in V_\tau$ has coordinate $(i_0,j_0,k_0)$, where $i_0,j_0,k_0\in\{0,1,\cdots,n-1\}$ and  $i_0+j_0+k_0=n$.
    The entry of $\sum_{j=1,2,3}\sum_{1\leq i\leq n-1} k_{j,i}{\bf a}(\tau,v_j,i)$, labeled by the vertex $v_0=(i_0,j_0,k_0)$, is $k_{1,n-i_0}+k_{2,n-j_0}+k_{3,n-k_0}$,
    where $k_{j,n}=0$ for $j=1,2,3$.
    Then we have 
    \begin{align}\label{eq-ijk=equal-zero}
        k_{1,n-i_0}+k_{2,n-j_0}+k_{3,n-k_0}=0.
    \end{align}
    Set $i_0=0$, $j_0=0$, or $k_0=0$  in \eqref{eq-ijk=equal-zero}, we have 
    \begin{equation}\label{eq-conjugate-n-ijk}
        \begin{split}
             k_{2,n-j_0} +  k_{3,j_0} = 0,\quad
             k_{1,n-i_0} +  k_{3,i_0} = 0,\quad
            k_{1,n-i_0} +  k_{2,i_0} = 0.
        \end{split}
    \end{equation}
    Comparing equations in \eqref{eq-conjugate-n-ijk}, we have 
    $k_{1,s}=k_{2,s}=k_{3,s}$ for $1\leq s\leq n-1$. Then equation \eqref{eq-ijk=equal-zero} implies that $ k_{j,n-i_0}+k_{j,n-j_0}+k_{j,n-k_0}=0$ for $j=1,2,3$.
    Then Lemma \ref{lem-ijk-sum-n} shows that 
    $k_{j,i} = 0$ for $j=1,2,3$ and $1\leq i\leq n-1$. This completes the proof.

    (b) It follows from 
    (a) that ${\bf c}(\tau,v_j,i)$ (see \eqref{def-vector-c-tau}), for $j=1,2,3$ and $2\leq i\leq n$, are $\mathbb Z$-linearly independent.
    Equations \eqref{eq-relation-k-a-p}, \eqref{eq-ba-relation-k}, and \eqref{def-vector-c-tau1} show that 
    \begin{equation}\label{eq-vectors-b-c}
    \begin{split}
       {\bf b}(\tau,v_j,i) &=  (n-i)\sum_{1\leq t\leq n} {\bf c}(\tau,v_j,t) - n \sum_{1\leq t\leq n-i} {\bf c}(\tau,v_j,t)\\
       &= (n-i)\sum_{2\leq t\leq n} {\bf c}(\tau,v_j,t) - n \sum_{2\leq t\leq n-i} {\bf c}(\tau,v_j,t).\\
    \end{split}
    \end{equation}
For each $j=1,2,3$,
Equation \eqref{eq-vectors-b-c} implies 
\begin{align}\label{eq-vectors-b-c-new}
    \begin{cases}
        {\bf b}(\tau,v_j,1)=n {\bf c}(\tau,v_j,n) - \sum_{2\leq t\leq n} {\bf c}(\tau,v_j,t),\\
        {\bf b}(\tau,v_j,i)-{\bf b}(\tau,v_j,i-1)=n {\bf c}(\tau,v_j,n+1-i) - \sum_{2\leq t\leq n} {\bf c}(\tau,v_j,t)
        \text{ for $2\leq i\leq n-1$,}\\
         {\bf b}(\tau,v_j,n-1)
        =\sum_{2\leq t\leq n} {\bf c}(\tau,v_j,t).
    \end{cases}
\end{align}
    This shows that each ${\bf c}(\tau,v_{j_0},i_0)$ is a $\mathbb Q$-linear combination of ${\bf b}(\tau,v_j,i)$ for $j=1,2,3$ and $1\leq i\leq n-1$. Then the $\mathbb Z$-linear independence of 
    $\{{\bf c}(\tau,v_j,i)\mid j=1,2,3,\;2\leq i\leq n\}$ implies the 
    $\mathbb Z$-linear independence of 
    $\{{\bf b}(\tau,v_j,i)\mid j=1,2,3,\;1\leq i\leq n-1\}$.
    
\end{proof}

Let $\fS$ be a triangulable pb surface with a triangulation $\lambda$.  
By cutting $\fS$ along all non-boundary edges of $\lambda$, we obtain a collection of ideal triangles, denoted by $\mathbb F_\lambda$.  
There is a natural projection
\[
{\bf pr}_\lambda \colon \bigsqcup_{\tau\in\mathbb F_\lambda} \tau \;\longrightarrow\; \fS 
= \Big( \bigsqcup_{\tau\in\mathbb F_\lambda} \tau \Big) /\!\!\sim,
\]
as in \eqref{eq.glue}.  

There is a group embedding
\[
\iota \colon \mathbb Z^{V_\lambda} \;\hookrightarrow\; \bigoplus_{\tau\in\mathbb F_\lambda} \mathbb Z^{V_\tau},
\qquad 
{\bf k}\mapsto \sum_{\tau\in\mathbb F_\lambda}{\bf k}_\tau,
\]
where ${\bf k}_\tau$ is defined in \eqref{eq-pull-back}.
Moreover, $\iota$ restricts to a group embedding
\begin{align}\label{def-iota-emb}
    \iota \colon \mathcal B_\lambda \;\hookrightarrow\; \bigoplus_{\tau\in\mathbb F_\lambda} \mathcal B_\tau,
\qquad 
{\bf k}\mapsto \sum_{\tau\in\mathbb F_\lambda}{\bf k}_\tau.
\end{align}

Now fix a puncture $p$ of $\fS$ and $1\leq i\leq n-1$.  
We will define elements ${\bf a}(\lambda,p,i)\in\mathbb Z^{V_\lambda}$ and ${\bf b}(\lambda,p,i)\in\mathcal B_\lambda$.  
Suppose that 
\[
{\bf pr}_\lambda^{-1}(\{p\}) = \{p_1,\dots,p_m\},
\]
and let $\tau_j\in\mathbb F_\lambda$ denote the ideal triangle containing $p_j$ (note that some $\tau_j$ may coincide).  
For each $1\leq j\leq m$, the element ${\bf a}(\tau_j,p_j,i)\in\mathbb Z^{V_{\tau_j}}$ can be regarded naturally as an element of $\bigoplus_{\tau\in\mathbb F_\lambda} \mathbb Z^{V_\tau}$.  
It follows from the definition of ${\bf a}(\tau_j,p_j,i)$ that
\[
\sum_{j=1}^m {\bf a}(\tau_j,p_j,i) \;\in\; \im \iota.
\]
Since $\iota$ is injective, there exists a unique ${\bf a}(\lambda,p,i)\in \mathbb Z^{V_\lambda}$ such that
\begin{equation}\label{def-a-l-p-i}
    \iota\big({\bf a}(\lambda,p,i)\big) 
    \;=\; \sum_{j=1}^m {\bf a}(\tau_j,p_j,i).
\end{equation}
In the same way, we define elements ${\bf b}(\lambda,p,i), {\bf c}(\lambda,p,i)\in \mathbb Z^{V_\lambda}$.  
Finally, by \eqref{eq-ba-relation-k}, one has 
\(
{\bf b}(\lambda,p,i)\in\mathcal B_\lambda.
\)

\begin{lemma}\label{lem-free-ab}
Let $\fS$ be a triangulable pb surface with a triangulation $\lambda$, and let $\mathcal P$ be the set of punctures of $\fS$.
We have
    \begin{enumerate}[label={\rm (\alph*)}]
        \item  ${\bf a}(\lambda,p,i)$, for $1\leq i\leq n-1$ and  $p\in\mathcal P$, are $\mathbb Z$-linearly independent;

        \item  ${\bf b}(\lambda,p,i)$, for $1\leq i\leq n-1$ and  $p\in\mathcal P$, are $\mathbb Z$-linearly independent;
    \end{enumerate}
\end{lemma}
\begin{proof}
Since the proving techniques for (a) and (b) are the same, we only prove (a).
Suppose that $\mathcal P=\{p^1,\dots,p^m\}$ and that 
${\bf pr}_\lambda^{-1}(p^j) = \{p^j_1,\dots,p^j_{t_j}\}$ for some positive integer $t_j$, for each $1\leq j\leq m$. 
Note that 
\begin{align}\label{eq-intersection-emptyset}
    \{p^j_1,\dots,p^j_{t_j}\}\cap \{p^s_1,\dots,p^s_{t_s}\}=\emptyset
\end{align}
for $1\leq j\neq s\leq m$.
For each $1\leq j\leq m$ and $1\leq s\leq t_j$, suppose that $p^j_s$ is contained in the ideal triangle 
$\tau^j_s\in\mathbb F_\lambda$. 
By definition of ${\bf a}(\lambda,p^j,i)$, we have 
\[
\iota\big({\bf a}(\lambda,p^j,i)\big) \;=\; \sum_{s=1}^{t_j} {\bf a}(\tau^j_s,p^j_s,i), \quad 
1\leq j\leq m,\; 1\leq i\leq n-1.
\]

Now suppose that 
\[
\sum_{j=1}^m \sum_{i=1}^{n-1} k_{j,i}\,{\bf a}(\lambda,p^j,i)=0
\]
for some integers $k_{j,i}$. 
Applying $\iota$, this gives
\[
\sum_{j=1}^m \sum_{i=1}^{n-1} k_{j,i} 
\sum_{s=1}^{t_j} {\bf a}(\tau^j_s,p^j_s,i)=0.
\]
By Lemma~\ref{lem-independence-P3}(a) together with 
\eqref{eq-intersection-emptyset}, it follows that 
$k_{j,i}=0$ for all $1\leq j\leq m$ and $1\leq i\leq n-1$. 
This proves the $\mathbb Z$-linear independence.

\end{proof}

\begin{lemma}\label{lem-loop-center}
Let $\fS$ be a triangulable pb surface with a triangulation $\lambda$.
    For any puncture $p$ of $\fS$ and any $1\leq i\leq n-1$, the element $Z^{{\bf a}(\lambda,p,i)}$ is contained in the center of $\mathcal Z_{\hat\omega}(\fS,\lambda)$.
\end{lemma}
\begin{proof}
    From the definition of ${\bf a}(\lambda,p,i)$, it is straightforward to verify that 
    ${\bf a}(\lambda,p,i) Q_{\lambda}= {\bf 0}$.
    This completes the proof. 
\end{proof}

\def\Bla{\mathcal B_{\lambda}}

The following proposition investigates the kernel of the anti-symmetric bilinear form
\begin{align}\label{anti-bilinear}
    (\;,\;)\colon \mathcal B_\lambda \times \mathcal B_\lambda \longrightarrow \mathbb Z,
    \quad ({\bf k}, {\bf t}) \longmapsto {\bf k} Q_\lambda {\bf t}^T.
\end{align}
In view of \eqref{eq-prod-vector-q} and Lemma \ref{lem-center-torus},
this analysis is a crucial step toward formulating the center of $\mathcal Z_{\hat\omega}(\fS,\lambda)$. 
As observed in the cases $n=2,3$, the bilinear form \eqref{anti-bilinear} reflects the product behavior of basis elements through their ``highest terms'' under the quantum trace map 
\cite{BW11,LY22,unicity,Kim20,kim2024unicity}. 
Thus, the study of \eqref{anti-bilinear} plays an important role in understanding both the structure of the center and the representation theory of the $\SL$-skein algebra. 
The following shows that its kernel is freely generated by the elements ${\bf b}(\lambda,p,i)$ associated to the punctures of~$\fS$.

We will further investigate the bilinear form \eqref{anti-bilinear} in \S\ref{sub-sec-rank-Z} from a different perspective.

\begin{proposition}\label{Prop-central-generic}
    Let $\fS$ be a triangulable punctured surface with a triangulation $\lambda$. 
    We use $\mathcal P$ to denote the set of punctures of $\fS$.
    Then the subgroup  $$\{{\bf k}\in \mathcal B_\lambda\mid {\bf k} Q_\lambda {\bf t}^T =0\text{ for all }{\bf t}\in \mathcal B_\lambda\}$$
    of $\mathcal B_\lambda$ is freely generated by
    $\{{\bf b}(\lambda,p,i)\mid 1\leq i\leq n-1,\; p\in\mathcal P\}$.
\end{proposition}
\begin{proof}

Lemma \ref{lem-loop-center} implies that 
${\bf a}(\lambda,p,i) Q_\lambda = {\bf 0}$ for every 
$1 \leq i \leq n-1$ and $p \in \mathcal P$. 
Consequently,
\[
\{{\bf b}(\lambda,p,i)\mid 1\leq i\leq n-1,\; p\in\mathcal P\}
   \subset \{{\bf k}\in \mathcal B_\lambda \mid {\bf k} Q_\lambda {\bf t}^T = 0 
   \ \text{for all }{\bf t}\in \mathcal B_\lambda\}.
\]
Moreover, Lemma~\ref{lem-free-ab} guarantees that the set 
$\{{\bf b}(\lambda,p,i)\mid 1\leq i\leq n-1,\; p\in\mathcal P\}$ 
is $\mathbb Z$-linearly independent.  
Now let ${\bf k}\in \mathcal B_\lambda$ satisfy 
${\bf k} Q_\lambda {\bf t}^T = 0$ for all ${\bf t}\in \mathcal B_\lambda$. 
To prove the proposition, it remains to show that ${\bf k}$ lies in the 
$\mathbb Z$-span of 
$\{{\bf b}(\lambda,p,i)\mid 1\leq i\leq n-1,\; p\in\mathcal P\}$.

    Let $e$ be an ideal arc in $\lambda$. Suppose that $\tau,\tau'$ are the two ideal triangles in $\mathbb F_\lambda$ adjacent to $e$. Our definition of the triangulation guarantees $\tau\neq\tau'$.
    Suppose the two edges of $\tau$ distinct from $e$ are $e_1, e_2$, and the two edges of $\tau'$ distinct from $e$ are $e_3, e_4$.
   Note that some of $e_1,e_2,e_3,e_4$ may coincide.
Cutting $\fS$ along $e_1,e_2,e_3,e_4$, we obtain $\fS'\bigsqcup\mathbb P_{4,e}$, where $\fS'$ is a pb surface and $\mathbb P_{4,e}=\mathbb P_4$ is the component containing the edge $e$.
    The triangulation $\lambda$ of $\fS$ induces a triangulation $\lambda'$ of $\fS'$. Let $\lambda_e$ denote the set of the edge $e$ together with the four boundary edges of $\mathbb P_{4,e}$. 
    \begin{align}\label{def-B-Q-V}
        \text{We use $\mathcal B_e$, $Q_e$, and $V_e$ to denote $\mathcal B_{\lambda_e}$, $Q_{\lambda_e}$, and $V_{\lambda_e}$ respectively.}
    \end{align}
    There is a group embedding
    $\iota_e\colon \Bla\rightarrow\mathcal B_{\lambda'}\oplus\mathcal B_e$ satisfying
    $$\mathcal S_{e_1}\circ\mathcal S_e\circ\mathcal S_e\circ\mathcal S_{e_4}(Z^{\bf k})=Z^{\iota_e({\bf k})} \quad\text{(see \eqref{cutting_homomorphism_for_Z_omega})}$$
    for any ${\bf k}\in\mathcal B_\lambda.$
    Then there are group homomorphisms 
    \begin{align}\label{def-iota-1-2}
        \text{$\iota_1\colon\Bla\rightarrow \mathcal B_{\lambda'}$ and $\iota_2\colon\Bla\rightarrow \mathcal B_{e}$
    such that $\iota_e({\bf k})=(\iota_1({\bf k}),\iota_2({\bf k}))$ for any ${\bf k}\in \Bla.$}
    \end{align}
    The well-definedness of the algebra embedding in \eqref{cutting_homomorphism_for_Z_omega} shows  
    \begin{align}\label{eq-anti-form-equal}
        {\bf k}_1Q_{\lambda}{\bf k}_2^T
    =(\iota_1({\bf k}_1), \iota_2({\bf k}_1))
    \begin{pmatrix}
        Q_{\lambda'} & O\\
        O & Q_e
    \end{pmatrix}
     \begin{pmatrix}
        \iota_1({\bf k}_2)^T  \\
        \iota_2({\bf k}_2)^T  
    \end{pmatrix},
    \end{align}
    for any ${\bf k}_1,{\bf k}_2\in\Bla$, where $O$ represents the zero matrix.

    We use $\mathring{V}_e$ to denote the subset of $V_e$ consisting of vertices contained in the interior of $\mathbb P_{4,e}$.
    Note that $\mathring{V}_e$ is also a subset of $V_\lambda$. Define 
    \begin{align}\label{def-lambda-ee}
        \Lambda_e:=\{{\bf t}\in\Bla\mid{\bf t}(v)=0\text{ for any }v\in V_\lambda\setminus\mathring{V}_e\}.
    \end{align}
    For any ${\bf t}\in\Lambda_e$, 
    we have
    \begin{equation}\label{eq-k-t-zero}
    \begin{split}
        {\bf k} Q_\lambda {\bf t}^T
        &=(\iota_1({\bf k}), \iota_2({\bf k}))
    \begin{pmatrix}
        Q_{\lambda'} & O\\
        O & Q_e
    \end{pmatrix}
     \begin{pmatrix}
        \iota_1({\bf t})^T \\
        \iota_2({\bf t})^T 
    \end{pmatrix}\quad (\because \mbox{\eqref{eq-anti-form-equal}}) \\
        &=(\iota_1({\bf k}), \iota_2({\bf k}))
    \begin{pmatrix}
        Q_{\lambda'} & O\\
        O & Q_e
    \end{pmatrix}
     \begin{pmatrix}
        O \\
        \iota_2({\bf t})^T  
    \end{pmatrix}\quad (\because \mbox{${\bf t}\in \Lambda_e$}) \\
        &=\iota_2({\bf k}) Q_e \iota_2({\bf t})^T\\
        &=0.
        \end{split}
    \end{equation}
Since $n\mathbb Z^{V_\lambda}\subset\mathcal B_\lambda$, we obtain 
\[
   \iota_2({\bf k}) \, Q_e 
   \begin{pmatrix}
      nI \\[2pt]
      O
   \end{pmatrix} = {\bf 0},
\]
where $I$ is the identity matrix indexed by $\mathring{V}_e$, and $O$ is the zero matrix of the form $(V_e\setminus \mathring{V}_e)\times \mathring{V}_e$.

   Put
    \begin{align}\label{eq-K-e-def}
        K_e=K_{\lambda_e},
    \end{align}
    where $K_{\lambda_e}$ is the one in Lemma~\ref{lem-matrix-HK}(a).
    According to Lemma \ref{lem-matrix-HK}(b), there exists a vector ${\bf c}\in\mathbb Z^{V_e}$ such that $\iota_2({\bf k}) = {\bf c}K_e$. Then we have 
    \begin{align}\label{eq-cKQ}
        {\bf c}K_e Q_e 
        \begin{pmatrix}
        nI \\
        O  
    \end{pmatrix}={\bf 0}.
    \end{align}
   It follows from \cite[Equation (7.1)]{KaruoWangToAppear} that 
   \begin{align}\label{eq-KQ}
       K_e Q_e=
       \begin{pmatrix}
           -nI & *\\
           O & *
       \end{pmatrix},
   \end{align}
   where columns and rows are divided into 
   $(\mathring{V}_e,V_e\setminus\mathring{V}_e)$  (note that $K_e$ equals $\overline{\mathsf{K}}_\mu$ in \cite{KaruoWangToAppear} and 
   $Q_e$ equals $\frac{1}{2}\overline{\mathsf{Q}}_\mu$ in \cite{KaruoWangToAppear}). 
   Assume that ${\bf c}=({\bf c}_1,{\bf c}_2)$, where ${\bf c}_1\in\mathbb Z^{\mathring{V}_e}$
   and ${\bf c}_2\in\mathbb Z^{V_e\setminus\mathring{V}_e}$.
   Then equations \eqref{eq-cKQ} and \eqref{eq-KQ} imply that 
   \begin{align}\label{eq-nc1-zero}
       {\bf 0}=({\bf c}_1,{\bf c}_2)
       \begin{pmatrix}
           -nI & *\\
           O & *
       \end{pmatrix}
       \begin{pmatrix}
        nI \\
        O  
    \end{pmatrix}=
    (-n^2{\bf c}_1, {\bf 0}).
   \end{align}
   Equation~\eqref{eq-nc1-zero} shows that ${\bf c}_1={\bf 0}$. 
Hence we may write
\[
   \iota_2({\bf k}) = ({\bf 0}, {\bf c}_2) K_e,
\]
which implies that $\iota_2({\bf k})$ is a $\mathbb Z$-linear combination of the row vectors 
$K_e(v,\bullet)$ with $v\in V_e\setminus\mathring{V}_e$. 
Here $K_e(v,\bullet)$ denotes the row of $K_e$ indexed by $v$.

 \begin{figure}[htbp]
    \centering
        \centering
        \includegraphics[width=7cm]{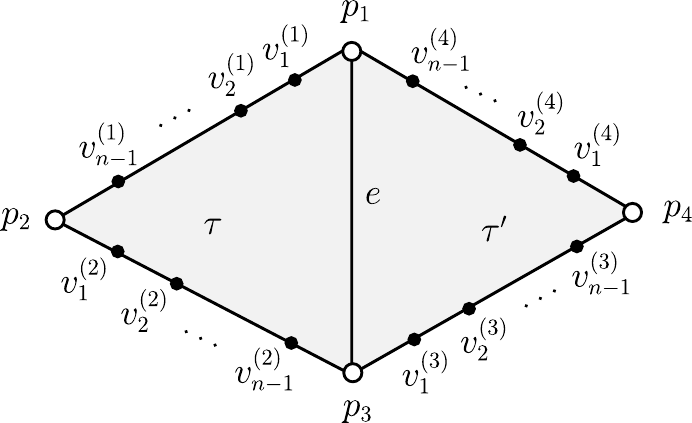}
        \caption{Labelings for $\mathbb P_{4,e}$.}
        \label{P4-boundary-vertices}
\end{figure}

The four punctures of $\mathbb P_{4,e}$, the two ideal triangles of $\mathbb P_{4,e}$, 
and the vertices in $V_e\setminus\mathring{V}_e$ are labeled as in 
Figure~\ref{P4-boundary-vertices}. 
Let $f_e$ denote the group embedding 
\[
    f_e \colon \mathcal B_e \longrightarrow \mathcal B_\tau \oplus \mathcal B_{\tau'}
\]
induced by cutting $\mathbb P_{4,e}$ along the edge $e$. 
It follows immediately from the construction that
\[
    ({\bf b}(\tau,p_1,i), {\bf b}(\tau',p_1,i)),\quad
    ({\bf b}(\tau,p_2,i), {\bf 0}),\quad
    ({\bf b}(\tau,p_3,i), {\bf b}(\tau',p_3,i)),\quad
    ({\bf 0}, {\bf b}(\tau',p_4,i)) \;\in\; \im f_e
\]
for each $1\leq i\leq n-1$. 
Hence, for every $1\leq i\leq n-1$ and $j=1,2,3,4$, 
there exists an element 
${\bf b}(\mathbb P_{4,e},p_j,i)\in\mathcal B_e$ such that 
\begin{align}\label{def-P4-b}
    f_e\bigl({\bf b}(\mathbb P_{4,e},p_j,i)\bigr)=
    \begin{cases}
        ({\bf b}(\tau,p_1,i), {\bf b}(\tau',p_1,i)) & j=1,\\[4pt]
        ({\bf b}(\tau,p_2,i), {\bf 0}) & j=2,\\[4pt]
        ({\bf b}(\tau,p_3,i), {\bf b}(\tau',p_3,i)) & j=3,\\[4pt]
        ({\bf 0}, {\bf b}(\tau',p_4,i)) & j=4.
    \end{cases}
\end{align}

   We will use the following lemma to complete the proof. 
The proof of this lemma is deferred to \S\ref{sub-proof-of-Lemma}.

    \begin{lemma}\label{lem-Ke-b}
        For $1\leq i\leq n-1$ and $j=1,2,3,4$, we have $$K_e(v_i^{(j)},\bullet)
        ={\bf b}(\mathbb P_{4,e},p_j,n-i)\in\mathcal B_e.$$
    \end{lemma}

     For each $\mu\in\mathbb F_\lambda$, define 
     $$g_{\mu}\colon\Bla\rightarrow
     \oplus_{\tau\in\mathbb F_\lambda}\mathcal B_{\tau}\rightarrow \mathcal B_\mu,$$
     where the first map is the group embedding induced by cutting all ideal arcs in $\lambda$, and the second map is the projection.
     For each $\mu\in\mathbb F_\lambda$, we label the three punctures of $\mu$ as $p_{j,\mu}$, $j=1,2,3$.
     Since, for each ideal arc $e\in\lambda$, the vector $\iota_2({\bf k})$ is a $\mathbb Z$-linear combination of the row vectors 
$K_e(v,\bullet)$ with $v\in V_e\setminus\mathring{V}_e$, 
Lemma~\ref{lem-Ke-b} implies that 
$g_{\mu}({\bf k})$ can be expressed as a $\mathbb Z$-linear combination of the elements ${\bf b}(\mu,p_{j,\mu},i)$ 
for $1\leq i\leq n-1$ and $1\leq j\leq 3$.
    Define
     \begin{align*}
         S_{\bf k}:=\{(\mu,i,j)\mid& \mu\in\mathbb F_\lambda,\,1\leq i\leq n-1, \,j=1,2,3,\\
         &\text{and the coefficient of ${\bf b}(\mu,p_{j,\mu},i)$ in $g_\mu({\bf k})$ is nonzero}\}.
     \end{align*}
     
     We will show that ${\bf k}$ is a $\mathbb Z$-linear sum of elements in 
$\{{\bf b}(\lambda,p,i)\mid 1\leq i\leq n-1,\; p\in\mathcal P\}$ using the induction on $|S_{\bf k}|$. When $|S_{\bf k}|=0$, we have ${\bf k}={\bf 0}$ because $g_\mu(k)={\bf 0}$ for any $\mu\in\mathbb F_\lambda$. 
Assume that $|S_{\bf k}|>0$. Then there exist
$\mu_0\in\mathbb F_\lambda$, $1\leq i_0\leq n-1$ and $j_0=1,2,3$ such that the coefficient of ${\bf b}(\mu_0,p_{j_0,\mu_0},i_0)$ in $g_\mu({\bf k})$ is nonzero.
Suppose that this coefficient is $z\in\mathbb Z\setminus\{0\}$.
We use ${\bf pr}_\lambda$ to denote the projection from $\bigsqcup_{\tau\in \mathbb F_\lambda}\tau \rightarrow\fS$.
Suppose ${\bf pr}_\lambda(p_{j_0,\mu_0}) = p_0\in\mathcal P$.
We label the edges and ideal triangles adjacent to $p_0$ as in Figure \ref{puncture-p}. 
Assume that $({\bf pr}_{\lambda})^{-1}(p_0)=\{p_1,\cdots,p_m\}$ and $p_t$ is contained in $\tau_t$ for each $1\leq t\leq m$.
Without loss of generality, we suppose $\mu_0=\tau_1$ and $p_{j_0,\mu_0}=p_1$.
With the repeated applications of Lemma \ref{lem-Ke-b} on $e=e_2,e_3,\cdots,e_m$, we have that the coefficient of ${\bf b}(\tau_t,p_{t},i_0)$ in $g_{\tau_t}({\bf k})$ is also $z$ for $t=2,3,\cdots,m$.
Set ${\bf k}'={\bf k}-z{\bf b}(\lambda,p_0,i_0)$.
Then ${\bf k}' Q_\lambda {\bf t}^T =0\text{ for all }{\bf t}\in \mathcal B_\lambda$ and 
$|S_{{\bf k}'}|<|S_{\bf k}|$. 
Then ${\bf k}'$ is a $\mathbb Z$-linear sum of elements in 
$\{{\bf b}(\lambda,p,i)\mid 1\leq i\leq n-1,\; p\in\mathcal P\}$. Thus so is ${\bf k}$.

\begin{figure}[htbp]
    \centering
        \centering
        \includegraphics[width=4cm]{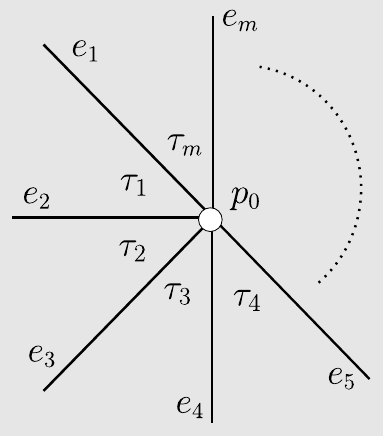}
        \caption{Labelings of the edges and ideal triangles adjacent to $p_0$.}
        \label{puncture-p}
\end{figure}

\end{proof}

For an algebra $\mathcal A$, we will use $\mathsf{Z}(\mathcal A)$ to denote the center of $\mathcal A$.
Lemma~\ref{lem-center-torus}(a), Lemma~\ref{lem-free-ab}(b), and
Proposition~\ref{Prop-central-generic} imply the following.

\begin{corollary}\label{cor-center-generic}
    Let $\fS$ be a triangulable punctured surface with a triangulation $\lambda$. 
    We use $\mathcal P$ to denote the set of punctures of $\fS$.
    Suppose that $\omega$ is not a root of unity. Then:
    \begin{enumerate}[label={\rm (\alph*)}]
        \item $\{{\bf b}(\lambda,p,i)\mid p\in\mathcal P,1\leq i\leq n-1\}$ is a basis of the abelian group
        $\{{\bf k}\in\mathcal B_\lambda\mid
{\bf k}Q_\lambda= {\bf 0}\}$.

\item $\mathsf{Z}(\mathcal Z_{\homega}^{\rm bl}(\fS,\lambda))=\mathbb C[Z^{{\bf b}(\lambda,p,i)}\mid 1\leq i\leq n-1,\; p\in\mathcal P]$.
    \end{enumerate}
\end{corollary}

\subsection{Proof of Lemma \ref{lem-Ke-b}}\label{sub-proof-of-Lemma}

\begin{proof}[Proof of Lemma \ref{lem-Ke-b}]
For $1\leq i\leq n-1$ and $j=1,2,3,4$, 
similarly as ${\bf b}(\mathbb P_{4,e},p_j,i)$, we can define ${\bf a}(\mathbb P_{4,e},p_j,i)$ such that \eqref{def-P4-b} holds with ${\bf b}(\text{-})$ replaced by ${\bf a}(\text{-})$.
For $1\leq i\leq n$ and $j=1,2,3,4$,
define 
    \begin{align}\label{def-c-p-i-P4}
        {\bf c}(\mathbb P_{4,e},p_j,i):=
        \sum_{1\leq t\leq i-1} {\bf a}(\mathbb P_{4,e},p_j,t),
    \end{align}
and $${\bf c} (\mathbb P_{4,e},p_j):=\sum_{1\leq t\leq n} {\bf c}(\mathbb P_{4,e},p_j,t).$$
For each \( j = 1,2,3,4 \), let \( \alpha_j \) denote the corner arc in \( \mathbb{P}_{4,e} \) surrounding \( p_j \), oriented counterclockwise around \( p_j \). For each $t\in\{1,2,\cdots,n\}$, we use $\alpha_j^{(t)}$ to denote the stated arc in $\mathbb P_{4,e}$ obtained by equipping the two endpoints of $\alpha_j$ with the state $t$.
Then we have the following two lemmas.
\begin{lemma}\cite[Lemma 4.5]{KimWang}\label{lem-trace-arc-p4}
    For $1\leq t\leq n$ and $j=1,2,3,4$,  we have 
$${\rm tr}_{\lambda_e}(\alpha_j^{(t)}) = Z^{n{\bf c}(\mathbb P_{4,e},p_j,t)- {\bf c} (\mathbb P_{4,e},p_j)} \in\mathcal Z_{\hat\omega}^{\rm bl}(\mathbb P_{4,e},\lambda_e).$$ 
\end{lemma}

It is shown in \cite[Lemma~7.6(a)]{LY23} that 
$\alpha_j^{(t_1)} \alpha_j^{(t_2)}= \alpha_j^{(t_2)} \alpha_j^{(t_1)}$ 
in $\overline{\cS}_{\hat\omega}^{\rm st}(\mathbb P_{4,e})$, 
for all $1\leq t_1,t_2\leq n$ and $1\leq j\leq 4$.
Hence, in $\mathcal Z_{\hat\omega}^{\rm bl}(\mathbb P_{4,e},\lambda_e)$ we have
\begin{align}\label{commute-P-4}
    {\rm tr}_{\lambda_e}(\alpha_j^{(t_1)}) \,
    {\rm tr}_{\lambda_e}(\alpha_j^{(t_2)})
    = 
    {\rm tr}_{\lambda_e}(\alpha_j^{(t_2)}) \,
    {\rm tr}_{\lambda_e}(\alpha_j^{(t_1)})
\end{align}
for all $1\leq t_1,t_2\leq n$ and $1\leq j\leq 4$.

\begin{lemma}\cite[Lemma 4.13 and Equation (237)]{LY23}\label{lem-Z-tr}
For $1\leq i\leq n-1$ and $j=1,2,3,4$,  we have 
$$Z^{K_e(v_i^{(j)},\bullet)}=\prod_{t=i+1}^n{\rm tr}_{\lambda_e}(\alpha_j^{(t)})\in\mathcal Z_{\hat\omega}^{\rm bl}(\mathbb P_{4,e},\lambda_e).$$

\end{lemma}

For $1\leq i\leq n-1$ and $j=1,2,3,4$, we have 
\begin{equation}\label{eq-ZKvb}
    \begin{split}       Z^{K_e(v_i^{(j)},\bullet)}&=\prod_{t=i+1}^n{\rm tr}_{\lambda_e}(\alpha_j^{(t)})\\
&=Z^{\sum_{t=i+1}^n (n{\bf c}(\mathbb P_{4,e},p_j,t)- {\bf c} (\mathbb P_{4,e},p_j))}
\quad (\because \mbox{Lemma~\ref{lem-trace-arc-p4} and \eqref{commute-P-4}})\\
&= Z^{i{\bf c} (\mathbb P_{4,e},p_j)-
n\sum_{t=1}^i {\bf c} (\mathbb P_{4,e},p_j,t)}\\
&= Z^{{\bf b}(\mathbb P_{4,e},p_j,n-i)} \quad (\because \mbox{\eqref{eq-vectors-b-c}}).
    \end{split}
\end{equation}
This shows that $K_e(v_i^{(j)},\bullet)={\bf b}(\mathbb P_{4,e},p_j,n-i).$ 

\end{proof}

\subsection{The center of the balanced Fock-Goncharov algebra when $\omega$ is a root of unity}
Let $\fS$ be a triangulable punctured surface equipped with a triangulation $\lambda$.
Recall that under assumption \Rlabel{1}, there exists an algebra embedding 
\begin{align}\label{eq-new-def}
    \PT\colon \mathcal{Z}_{\heta}^{\rm bl}(\fS,\lambda)\rightarrow
    \mathcal{Z}^{\rm bl}_{\homega}(\fS,\lambda),\quad
    Z^{\bf k}\mapsto Z^{N{\bf k}}
\end{align}
for ${\bf k}\in\mathcal B_\lambda.$
We will prove that when $\hat\omega$ is a root of unity, the center 
$\mathsf{Z}(\mathcal{Z}_{\heta}^{\rm bl}(\fS,\lambda))$
is generated by the central elements lying in $\im\PT$ together with all $Z^{{\bf b}(\lambda,p,i)}$.  

To describe the central elements arising from $\im\PT$, we introduce the following notation:  
for any positive integer $m$, define
\begin{align}\label{def-B-lambda-m}
    \mathcal B_{\lambda,m}
    :=\{{\bf k}\in \mathcal B_\lambda\mid 
    {\bf k} Q_\lambda {\bf t}^T=0 \;\;\text{in}\;\; \mathbb Z_{mn}\;\; \text{for all }{\bf t}\in\mathcal B_\lambda\}.
\end{align}

The following theorem is our first main theorem. 
It provides an explicit description of the center 
$\mathsf{Z}(\mathcal{Z}_{\heta}^{\rm bl}(\fS,\lambda))$ when $\hat\omega$ is a root of unity. 
In \S\ref{sec-rep-torus-b}, this description will be used to classify the irreducible representations of 
$\mathcal{Z}_{\heta}^{\rm bl}(\fS,\lambda)$.  
As noted in Lemma~\ref{lem-center-torus}(b), determining the center 
$\mathsf{Z}(\mathcal{Z}_{\heta}^{\rm bl}(\fS,\lambda))$ reduces to analyzing the subgroup
\begin{align}\label{eq-define-subgroup}
       B_{N''}:= \{{\bf k}\in\mathcal B_\lambda\mid {\bf k} Q_\lambda {\bf t}^T=0\in\mathbb Z_{N''}\text{ for all }{\bf t}\in\mathcal B_\lambda\}.
\end{align}
Moreover, the study of $B_{N''}$ plays a key role in formulating the center of the ${\rm SL}_n$-skein algebra, 
as illustrated in the cases $n=2,3$ \cite{unicity,kim2024unicity}.

\begin{theorem}\label{thm-center-balanced-root-of-unity}
Assume \Rlabel{1}.  
Let $\fS$ be a triangulable punctured surface with a triangulation $\lambda$, and let $\PT$ be the algebra embedding defined in \eqref{eq-new-def}.  
Then the center of $\mathcal{Z}^{\rm bl}_{\homega}(\fS,\lambda)$, viewed as a subalgebra of $\mathcal{Z}^{\rm bl}_{\homega}(\fS,\lambda)$, is generated by 
$$\{\Phi^{\mathbb T}(Z^{\bf k})\mid{\bf k}\in \mathcal B_{\lambda,d}\}$$
together with the elements $$Z^{{\bf b}(\lambda,p,i)}$$ for $1\leq i\leq n-1$ and $p\in\mathcal P$, where $\mathcal P$ denotes the set of punctures of $\fS$.
\end{theorem}

The following proof makes extensive use of the notations and techniques introduced in the proof of Proposition \ref{Prop-central-generic}. For clarity, readers may wish to review that proof beforehand.

\begin{proof}
By Lemma~\ref{lem-center-torus}, it suffices to show that the subgroup $B_{N''}$ (see \eqref{eq-define-subgroup}) is generated by  
\[
\{N{\bf c}\mid {\bf c}\in\mathcal B_{\lambda,d}\} 
\quad\text{and}\quad 
\{{\bf b}(\lambda,p,i)\mid 1\leq i\leq n-1,\; p\in\mathcal P\}.
\]
It is straightforward to verify that
\[
\{N{\bf c}\mid {\bf c}\in\mathcal B_{\lambda,d}\}\;\cup\; \{{\bf b}(\lambda,p,i)\mid 1\leq i\leq n-1,\; p\in\mathcal P\}
\;\subseteq\; B_{N''}.
\]

Now let ${\bf k}\in B_{N''}$.  
    Let $e$ be an ideal arc in $\lambda$. 
We use the notation from the proof of Proposition~\ref{Prop-central-generic}.  
For any ${\bf t}\in\Lambda_e$ (see \eqref{def-lambda-ee}), as shown in \eqref{eq-k-t-zero}, we have  
\[
{\bf k} Q_\lambda {\bf t}^T
=\iota_2({\bf k}) Q_e \iota_2({\bf t})^T
=0 \in \mathbb Z_{N''} \quad \text{(see \eqref{def-B-Q-V} for $Q_e$ and \eqref{def-iota-1-2} for $\iota_2$)}.
\]

By Lemma~\ref{lem-matrix-HK}(b), there exists ${\bf c}\in\mathbb Z^{V_e}$ (see \eqref{def-B-Q-V} for $V_e$) such that 
\[
\iota_2({\bf k})={\bf c}K_e\quad (\text{see \eqref{eq-K-e-def} for $K_e$}).
\]
Hence,  
\begin{align}\label{new-eq-cKQ}
{\bf c}K_e Q_e 
\begin{pmatrix}
nI \\
O  
\end{pmatrix}
={\bf 0}\in\mathbb Z_{N''},
\end{align}
where ${\bf 0}$ denotes a zero row vector in $\mathbb Z_{N''}$.  
Write ${\bf c}=({\bf c}_1,{\bf c}_2)$ with ${\bf c}_1\in\mathbb Z^{\mathring{V}_e}$ and ${\bf c}_2\in\mathbb Z^{V_e\setminus\mathring{V}_e}$, where $\mathring{V}_e$ is a subset of $V_e$ consisting of vertices contained in the interior of $\mathbb P_{4,e}=\mathbb P_4$.  
From \eqref{new-eq-cKQ} and \eqref{eq-KQ} we obtain  
\begin{align}\label{new-eq-nc1-zero}
({\bf c}_1,{\bf c}_2)
\begin{pmatrix}
-nI & * \\
O   & *
\end{pmatrix}
\begin{pmatrix}
nI \\
O  
\end{pmatrix}
=(-n^2{\bf c}_1, {\bf 0})={\bf 0}\in\mathbb Z_{N''}.
\end{align}
Thus ${\bf c}_1={\bf 0}$ in $\mathbb Z_N$. Note that   
\begin{align}\label{eq-iota-two-zero}
\iota_2({\bf k})=({\bf c}_1,{\bf c}_2)K_e.
\end{align}
This shows that, in $\mathbb Z_N$, the vector $\iota_2({\bf k})$ is a $\mathbb Z_N$-linear combination of $K_e(v,\bullet)$ with $v\in V_e\setminus\mathring{V}_e$.
Here $K_e(v,\bullet)$ denotes the row of $K_e$ indexed by $v$.

For each $\mu\in\lambda$, define  
\[
g_{\mu}\colon\mathcal B_\lambda
\rightarrow \bigoplus_{\tau\in\mathbb F_\lambda}\mathcal B_{\tau}
\rightarrow \mathcal B_\mu,
\]
where the first map is induced by cutting along all ideal arcs in $\lambda$, and the second is the natural projection.  
Label the three punctures of each $\mu\in\mathbb F_\lambda$ by $p_{j,\mu}$, $j=1,2,3$.  
Since, for each ideal arc $e\in\lambda$, the vector $\iota_2({\bf k})$ is a $\mathbb Z_N$-linear combination of the row vectors 
$K_e(v,\bullet)$ with $v\in V_e\setminus\mathring{V}_e$, 
Lemma~\ref{lem-Ke-b} implies that, in $\mathbb Z_N$, the element $g_{\mu}({\bf k})$ is a $\mathbb Z_N$-linear combination of ${\bf b}(\mu,p_{j,\mu},i)$ with $1\leq i\leq n-1$, $j=1,2,3$.  

Define  
\begin{align*}
S_{\bf k}
:=\{(\mu,i,j)\mid &\; \mu\in\mathbb F_\lambda,\; 1\leq i\leq n-1,\; j=1,2,3, \\
&\text{and the coefficient of ${\bf b}(\mu,p_{j,\mu},i)$ in $g_\mu({\bf k})$ is nonzero in $\mathbb Z_N$}\}.
\end{align*}

We now show that ${\bf k}={\bf k}_1+{\bf k}_2$, where ${\bf k}_1$ is a $\mathbb Z$-linear combination of $\{{\bf b}(\lambda,p,i)\}$ and $S_{{\bf k}_2}=\emptyset$.  
We argue by induction on $|S_{\bf k}|$.  

 If $|S_{\bf k}|=0$, set ${\bf k}_1={\bf 0}$ and ${\bf k}_2={\bf k}$.  
Suppose $|S_{\bf k}|>0$. Then there exist $\mu_0\in\mathbb F_\lambda$, $1\leq i_0\leq n-1$, and $j_0\in\{1,2,3\}$ such that the coefficient of ${\bf b}(\mu_0,p_{j_0,\mu_0},i_0)$ in $g_{\mu_0}({\bf k})$ is nonzero in $\mathbb Z_N$.  
Let this coefficient be represented by some $z\in\mathbb Z\setminus\{0\}$.  
Let ${\bf pr}_\lambda\colon\bigsqcup_{\tau\in \mathbb F_\lambda}\tau \to \fS$ be the projection.  
Suppose ${\bf pr}_\lambda(p_{j_0,\mu_0}) = p_0\in\mathcal P$.  
Label the edges and ideal triangles adjacent to $p_0$ as in Figure~\ref{puncture-p}, and write $({\bf pr}_\lambda)^{-1}(p_0)=\{p_1,\dots,p_m\}$ with $p_t\in\tau_t$.  
Without loss of generality, take $\mu_0=\tau_1$ and $p_{j_0,\mu_0}=p_1$.  
By repeated applications of Lemma~\ref{lem-Ke-b} for $e=e_2,\dots,e_m$, we find that the coefficient of ${\bf b}(\tau_t,p_t,i_0)$ in $g_{\tau_t}({\bf k})$ is also $z\in\mathbb Z_N$ for each $t=2,\dots,m$.  
Set ${\bf k}'={\bf k}-z{\bf b}(\lambda,p_0,i_0)$.  
Then ${\bf k}'Q_\lambda{\bf t}^T=0$ for all ${\bf t}\in \mathcal B_\lambda$, and $|S_{{\bf k}'}|<|S_{\bf k}|$.  
This completes the inductive step.

Thus we may write ${\bf k}={\bf k}_1+{\bf k}_2$, where ${\bf k}_1$ is a $\mathbb Z$-linear sum of $\{{\bf b}(\lambda,p,i)\}$ and $S_{{\bf k}_2}=\emptyset$.  
Let $e$ be an ideal arc in $\lambda$.
As in \eqref{eq-iota-two-zero}, there exist ${\bf d}_1\in\mathbb Z^{\mathring{V}_e}$ and ${\bf d}_2\in\mathbb Z^{V_e\setminus \mathring{V}_e}$ with ${\bf d}_1={\bf 0}$ in $\mathbb Z_N$ and  
\[
\iota_2({\bf k}_2)=({\bf d}_1,{\bf d}_2)K_e.
\]
Since $S_{{\bf k}_2}=\emptyset$, it follows that ${\bf d}_2={\bf 0}$ in $\mathbb Z_N$ and hence ${\bf k}_2=N{\bf s}$ for some ${\bf s}\in\mathbb Z^{V_\lambda}$ (this is because $e$ could be any ideal arc in $\lambda$).  
Moreover,  
\[
\iota_2({\bf s})=\tfrac{1}{N}({\bf d}_1,{\bf d}_2)K_e\in\mathbb Z^{V_e} \text{ and } \tfrac{1}{N}({\bf d}_1,{\bf d}_2)\in\mathbb Z^{V_e}.
\]
By Lemma~\ref{lem-matrix-HK}(b), we have $\iota_2({\bf s})\in\mathcal B_e$ (see \eqref{def-B-Q-V} for $\mathcal B_e$), which implies ${\bf s}\in\mathcal B_\lambda$ (since $e$ is arbitrary).  
As ${\bf k}_2Q_\lambda{\bf t}^T=0$ in both $\mathbb Z_{N''}$ and $\mathbb Z_{n}$ (Lemma~\ref{lem-balanced-property}) for all ${\bf t}\in \mathcal B_\lambda$, it follows that
${\bf k}_2Q_\lambda{\bf t}^T=0$ in $\mathbb Z_{(nN'')/d'}$ for all ${\bf t}\in \mathcal B_\lambda$, where $d'=\gcd(N'',n)$.
Consequently,
${\bf s}Q_\lambda{\bf t}^T=0$ in $\mathbb Z_{dn}$ for all ${\bf t}\in \mathcal B_\lambda$, which shows that ${\bf s}\in\mathcal B_{\lambda,d}$. 
Therefore ${\bf k}_2=N{\bf s}$ with ${\bf s}\in\mathcal B_{\lambda,d}$.  

This completes the proof.
\end{proof}

\begin{remark}
    At present, a suitable `good basis' of the ${\rm SL}_n$-skein algebra is not available for computing its center. 
Once such a basis is established, the techniques used in the proof of Theorem~\ref{thm-center-balanced-root-of-unity} 
can be applied to determine the center of the ${\rm SL}_n$-skein algebra.

\end{remark}

For any positive integer $d$, 
define
\begin{align}\label{def-torus-d}
    \mathcal Z_{\hat{\omega}}^{\rm bl}(\fS,\lambda)_d:=
    \text{span}_{\BC}\{Z^{\bf k}\mid
    {\bf k}\in\mathcal B_{\lambda,d}\},
\end{align}
where $\mathcal B_{\lambda,d}$ is defined in \eqref{def-B-lambda-m}.
Obviously,  $\mathcal Z_{\hat{\omega}}^{\rm bl}(\fS,\lambda)_d$ is a subalgebra of $\mathcal Z_{\hat{\omega}}^{\rm bl}(\fS,\lambda)$.
Theorem \ref{thm-center-balanced-root-of-unity} implies the following.

\begin{corollary}\label{cor-center-Z-two}
    Suppose that we have the  assumption \Rlabel{1}.
For a triangulable punctured surface $\fS$ with a triangulation $\lambda$, let $\PT$ be the algebra embedding defined in \eqref{eq-new-def}.
Then we have the following:
\begin{enumerate}[label={\rm (\alph*)}]
    \item the center of 
    $\mathcal Z_{\hat{\eta}}^{\rm bl}(\fS,\lambda)$ is generated by
    $\mathcal Z_{\hat{\eta}}^{\rm bl}(\fS,\lambda)_d$ and $Z^{{\bf b}(\lambda,p,i)}$,
    where $1\leq i\leq n-1$ and $p\in\mathcal P$;

    \item the center of 
    $\mathcal Z_{\hat{\omega}}^{\rm bl}(\fS,\lambda)$ is generated by
    $\PT(\mathcal Z_{\hat{\eta}}^{\rm bl}(\fS,\lambda)_d)$ and $Z^{{\bf b}(\lambda,p,i)}$,
    where $1\leq i\leq n-1$ and $p\in\mathcal P$.
\end{enumerate}
\end{corollary}

\subsection{The rank of $\mathcal Z_{\hat{\omega}}^{\rm bl}(\fS,\lambda)$ over its center}\label{sub-sec-rank-Z}
\def\OfS{\overline{\fS}}
\def\Hnn{H_1(\overline{\fS},\mathbb Z_n)}
\def\BB{\mathcal{B}_\lambda}

Let $\fS$ be a triangulable connected punctured surface with a triangulation $\lambda$.  
In this subsection, we compute the rank of $\mathcal Z_{\hat{\omega}}^{\rm bl}(\fS,\lambda)$ over its center under the assumption \Rlabel{1}.  
It is well known that this rank equals the square of the dimension of any irreducible representation of $\mathcal Z_{\hat{\omega}}^{\rm bl}(\fS,\lambda)$ \cite[Lemma~8.4]{KaruoWangToAppear} (see Lemma~\ref{lem-irre-Azumaya}).  
Hence, this rank plays a crucial role in understanding the representation theory of $\mathcal Z_{\hat{\omega}}^{\rm bl}(\fS,\lambda)$. In particular, for the cases $n=2,3$, it has been observed that this rank coincides with the rank of the $\SL$-skein algebra over its center \cite{frohman2021dimension,kim2024unicity}.

Define
\begin{align}
\label{def-B-lambda-circ}
    \mathcal B_\lambda^{\circ}:=\{{\bf k}\in\mathcal B_\lambda\mid
    {\bf k}Q_\lambda=0\}.
\end{align}
Corollary~\ref{cor-center-generic}(a) shows that $\{{\bf b}(\lambda,p,i)\mid p\in\mathcal P,1\leq i\leq n-1\}$ forms a basis of $\mathcal B_\lambda^{\circ}$.  
Moreover, Lemma~\ref{lem-center-torus}(c) and Theorem~\ref{thm-center-balanced-root-of-unity} imply that
 \begin{align}\label{eq-rank-quotient-equality}
     \text{the rank of $\mathcal Z_{\hat{\omega}}(\fS,\lambda)$ over its center}=\left|\dfrac{\mathcal B_\lambda}{B_{N''}} \right|,
 \end{align}
 where $B_{N''}= N \mathcal B_{\lambda,d} +  \mathcal B_\lambda^{\circ}$ and 
 $\mathcal B_{\lambda,d}$ is defined in 
\eqref{eq-define-subgroup}.  
In the remainder of this subsection, we compute \eqref{eq-rank-quotient-equality}.

We first define a surjective group homomorphism
\begin{align}\label{def-zeta}
    \zeta\colon\mathcal B_\lambda\rightarrow H_1(\overline{\fS},\mathbb Z_n)
    \qquad\text{(see Definition \ref{def:surface} for $\OfS$).}
\end{align}

For each edge $e$ of $\lambda$, we take
two disjoint copies $e'$, $e''$ of $e$ such that no two edges in 
$\bigcup_{e\in\lambda}\{e',e''\}$ intersect. We write
\begin{align*}
    \widehat{\lambda} := \bigcup_{e\in\lambda} \{e',e''\}
\end{align*}
and call it the {\bf split  triangulation} of $\lambda$.  
Cutting $\fS$ along the edges of $\widehat{\lambda}$ yields a collection of $\mathbb P_2$ and $\mathbb P_3$.  
Each edge $e\in\lambda$ corresponds to a $\mathbb P_2$ in $\hat\lambda$, denoted by $B_e$; each triangle $\tau$ in $\lambda$ corresponds to a triangle $\hat\tau$ in $\hat\lambda$, denoted by $\hat\tau$.

Let ${\bf k}\in \BB$.  
For each $\tau\in\mathbb F_\lambda$, recall that we defined the pullback ${\bf k}_\tau\in \mathcal B_\tau$ of ${\bf k}$ by the characteristic map of $\tau$ (see \eqref{eq-pull-back}).  
Label the three vertices of $\tau$ by $v_1,v_2,v_3$.  
Since ${\bf k}_\tau$ is balanced, we may write
\begin{align}\label{eq-k-tau-s1-s2-s3}
    {\bf k}_\tau = s_1{\bf k}_1 + s_2{\bf k}_2 + s_3{\bf k}_3\in \mathbb Z_n^{V_\tau},
\end{align}
where each $s_i\in\mathbb Z$ and each ${\bf k}_i$ is defined in \eqref{def:proj}.  
Inside $\hat\tau$, the element $\zeta({\bf k})\in\Hnn$ is represented by the red curves in the left picture of Figure~\ref{Zn}, where the number next to each red oriented curve indicates the number of copies of that curve.

For each $e\in\lambda$, suppose that the two triangles adjacent to $e$ are $\tau$ and $\tau'$.  
Since our triangulation excludes self-folded triangles, we have $\tau\neq\tau'$.  
Then inside $\hat\tau\cup B_e\cup \hat\tau'$, the element $\zeta({\bf k})\in\Hnn$ is represented by the red components in the right picture of Figure~\ref{Zn}, where again the number next to each red oriented curve indicates the number of copies of that curve.

We now verify that $\zeta({\bf k})$ is a well-defined element of $\Hnn$.  
For each $e\in\lambda$, suppose the vertices of $\hat\tau$ and $\hat\tau'$ are labeled as in Figure~\ref{Zn}.  
Assume that
$${\bf k}_\tau = s_1{\bf k}_1 + s_2{\bf k}_2 + s_3{\bf k}_3\in \mathbb Z_n^{V_\tau}
\quad\text{and}\quad
{\bf k}_{\tau'} = s_1'{\bf k}_1 + s_2'{\bf k}_2 + s_3'{\bf k}_3\in \mathbb Z_n^{V_{\tau'}}.$$
Let the small vertices on $e$ be labeled $u_1,\dots,u_{n-1}$ in the direction from $v_1$ to $v_2$.  
Then the coordinate of $u_i$ in $\tau$ (resp. $\tau'$) is
$(n-i,i,0)$ (resp. $(n-i,0,i)$).  
Thus
$$(n-i)s_1 + is_2=(n-i) s_1' + i s_3'\in\mathbb Z_n,
\quad\text{for $1\leq i\leq n-1$}.$$
It follows that
$$s_2-s_1=s_3'-s_1'\in\mathbb Z_n.$$
Hence $\zeta({\bf k})$ is an element of $\Hnn$.

Next we show that for each $\tau\in\mathbb F_\lambda$, the element $\zeta({\bf k})$ is independent of the choice of $s_1,s_2,s_3$ in \eqref{eq-k-tau-s1-s2-s3}.  
Suppose that
\begin{align*}
    {\bf k}_\tau = s_1{\bf k}_1 + s_2{\bf k}_2 + s_3{\bf k}_3
    =t_1{\bf k}_1 + t_2{\bf k}_2 + t_3{\bf k}_3
    \in \mathbb Z_n^{V_\tau}.
\end{align*}
Then
$$(s_1-t_1){\bf k}_1 +(s_2-t_2){\bf k}_2 + (s_3-t_3){\bf k}_3={\bf 0}\in \mathbb Z_n^{V_\tau}.$$
From this we deduce
\begin{align}\label{eq-s-t-zn}
    s_1-t_1= s_2-t_2= s_3-t_3\in\mathbb Z_n.
\end{align}
It is straightforward to check that
\begin{align}\label{eq-zero-Hn}
    \begin{array}{c}\includegraphics[scale=0.8]{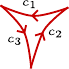}\end{array}
    =0\in\Hnn\quad\text{when $c_1=c_2=c_3\in\mathbb Z_n$}.
\end{align}
Equations \eqref{eq-s-t-zn} and \eqref{eq-zero-Hn} show that $\zeta({\bf k})\in\Hnn$ is independent of the choice of $s_1,s_2,s_3$.

\begin{figure}
	\centering
	\includegraphics[width=13cm]{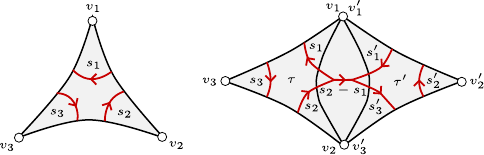}
	\caption{Local pictures of $\zeta({\bf k})$ in $\hat\tau$ and $\hat\tau\cup B_e\cup\hat{\tau}'$, where the number next to each red oriented curve indicates the number of copies of that curve. The vertices of $\hat\tau$ (resp. $\hat\tau'$) are labeled by $v_1,v_2,v_3$ (resp. $v_1',v_2',v_3'$).}
	\label{Zn}
\end{figure}

\def\BBB{\mathcal B_\lambda}

\begin{lemma}\label{lem-surj-zeta-def}
    The group homomorphism $\zeta$ in \eqref{def-zeta} is surjective.
\end{lemma}
\begin{proof}
Let $C$ be an element of $\Hnn$.  
We may choose a good representative of $C$ such that, for each triangle $\tau \in \mathbb F_\lambda$, the local picture of $C$ in $\tau$ is
\[
\begin{array}{c}
    \includegraphics[scale=0.4]{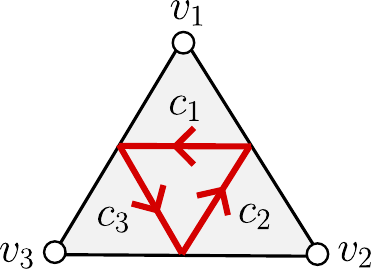}
\end{array}
\]
where $c_i$ denotes the multiplicity of the corresponding curve and each $v_i$ labels a puncture of $\tau$.  
Define
\[
    C_\tau = c_1{\bf k}_1 + c_2{\bf k}_2 + c_3{\bf k}_3 \in \mathcal B_\tau.
\]

We orient all edges of $\lambda$.  
For each $e \in \lambda$, let $\tau'$ and $\tau''$ be the two triangles adjacent to $e$, with $\tau'$ lying on the right when traversing $e$ according to its orientation.  
Label the small vertices of $e$ by $u_1,\dots,u_{n-1}$ in the order induced by the orientation of $e$.  
Assume that the local picture of $C$ in $\tau' \cup \tau''$ is
\[
\begin{array}{c}
    \includegraphics[scale=0.4]{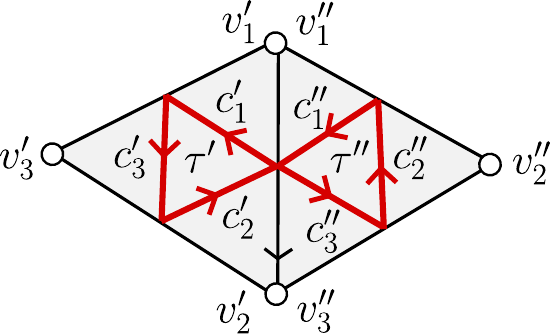}
\end{array},
\]
where the vertices of $\tau'$ and $\tau''$ are labeled.  
Since $C \in \Hnn$, we have
\[
    c_3'' - c_2' - c_1'' + c_1'=0 \in \mathbb Z_n.
\]
Hence,
\begin{align}\label{def-Ce}
    C_{\tau''}(u_i) - C_{\tau'}(u_i)
    &= (n-i)(c_1''-c_1') + i(c_3''-c_2') \notag \\
    &= n(c_1''-c_1') + i(c_3'' - c_2' - c_1'' + c_1') \notag \\
    &= 0 \in \mathbb Z_n.
\end{align}
Define $C_e \in \mathbb Z^{V_{\tau'}}$ by
\[
    C_e(v) =
    \begin{cases}
        C_{\tau''}(u_i) - C_{\tau'}(u_i), & v = u_i \text{ for } 1 \leq i \leq n-1, \\
        0, & \text{otherwise.}
    \end{cases}
\]
Equation~\eqref{def-Ce} shows that $C_e \in \mathcal B_\tau$.

Recall that there is a group embedding
\[
    \iota \colon \mathcal B_\lambda \hookrightarrow \bigoplus_{\tau \in \mathbb F_\lambda} \mathcal B_\tau,\quad
    {\bf k}\mapsto \sum_{\tau\in\mathbb F_\lambda} {\bf k}_\tau
    \qquad \text{(see \eqref{def-iota-emb})}.
\]
It follows immediately that
\[
    \sum_{\tau \in \mathbb F_\lambda} C_\tau +
    \sum_{e \in \lambda} C_e \in \im\iota.
\]
Let ${\bf c} \in \mathcal B_\lambda$ be such that
\[
    \iota({\bf c}) = \sum_{\tau \in \mathbb F_\lambda} C_\tau +
    \sum_{e \in \lambda} C_e.
\]
By the definitions of $C_\tau$ and $C_e$, we conclude that
\[
    \zeta({\bf c}) = C.
\]

\end{proof}

It follows from Lemma~\ref{lem-balanced-property} that there exists an anti-symmetric $\mathbb Z$-bilinear form
\begin{align}\label{eq-bracket-bal}
    \langle\cdot,\cdot\rangle\colon \BBB\times \BBB\rightarrow\mathbb Z,
    \qquad \langle {\bf k}, {\bf t}\rangle\mapsto \frac{1}{n}{\bf k}Q_\lambda{\bf t}^T.
\end{align}
Moreover, it is well known that there is a non-degenerate, anti-symmetric $\mathbb{Z}_n$-bilinear form
\begin{align}\label{eq-bracket-Z3}
    \langle\cdot,\cdot\rangle_n \colon 
    H_1(\overline{\fS},\mathbb{Z}_n)\times H_1(\overline{\fS},\mathbb{Z}_n)
    \longrightarrow \mathbb{Z}_n ,
\end{align}
defined by the algebraic intersection number modulo $n$, where we adopt the convention that, when computing the intersection number, the element on the right of $\langle\cdot,\cdot\rangle_n$ is positioned above the element on the left.

The following proposition represents a key technical and crucial step in computing the rank of $\mathcal Z_{\hat\omega}^{\rm bl}(\fS,\lambda)$ over its center. It asserts that the bilinear forms in \eqref{eq-bracket-bal} and \eqref{eq-bracket-Z3} are compatible under the group homomorphism $\zeta$.

\def\iV{\mathring{V}}
\def\pV{\partial V}

\begin{proposition}\label{prop-bilinear-key}
    For any two elements
    ${\bf k},{\bf h}\in \mathcal B_\lambda$, we have 
    $$\langle {\bf k}, {\bf h}\rangle =
    \langle \zeta({\bf k}),\zeta({\bf h})\rangle_n\in\mathbb Z_n$$
\end{proposition}
\begin{proof}
    Let ${\bf k}, {\bf h}\in \mathcal B_\lambda$.
    Recall that there is an algebra embedding
    \begin{align*}
       \mathcal Z_{\hat{\omega}}^{\rm bl}(\fS,\lambda)\hookrightarrow \otimes_{\tau\in\mathbb F_\lambda}
       \mathcal Z_{\hat{\omega}}^{\rm bl}(\tau),\quad
       Z^{\bf t}\mapsto \otimes_{\tau\in\mathbb F_\lambda} Z^{{\bf t}_\tau}\text{ for ${\bf t}\in\mathcal B_\lambda$}
        \qquad\text{(see \eqref{eq-splitting-torus}).}
    \end{align*}
    This embedding implies
    \begin{align}\label{eq-kh-QQQQQ}
        {\bf k}Q_\lambda{\bf h}^T
        =\sum_{\tau\in\mathbb F_\lambda}{\bf k}_\tau Q_\tau {\bf h}_\tau^T\qquad \text{for any ${\bf k},{\bf h}\in\mathcal B_\lambda$}.
    \end{align}

    For $\tau\in\mathbb F_\lambda$, label the three vertices of $\tau$.
    Suppose that 
    ${\bf k}_\tau=s_1{\bf k}+s_2{\bf k}_2 + s_3{\bf k}_3 + n{\bf b}_\tau\in\mathcal B_\tau.$
    Then
    \begin{align}\label{eq-ah-n-bQh}
        {\bf k}_\tau Q_\tau {\bf h}_\tau^T
        ={\bf a}_\tau{\bf h}_\tau^T + n{\bf b}_\tau Q_\tau{\bf h}_\tau^T,
    \end{align}
    where $ {\bf a}_\tau= (s_1{\bf k}+s_2{\bf k}_2 + s_3{\bf k}_3)Q_\tau.$
    We use $\mathring{V}_\tau$ to denote the set of small vertices in the interior of $\tau$, and $\partial V_\tau$ to denote $V_\tau\setminus\mathring{V}_\tau$.
    It is straightforward to verify that
    ${\bf a}_\tau(v)=0$ for every $v\in\iV_\tau$.
    Hence,
    \begin{align}\label{eq-ah-tau}
        {\bf a}_\tau{\bf h}_\tau^T
        =\sum_{v\in\pV_\tau}
        {\bf a}_\tau(v) {\bf h}_\tau(v).
    \end{align}

    Define $P_\tau= Q_\tau-H_\tau^T$, where $H_\tau$ is given in \eqref{def-Hn-non-boundary} and \eqref{def-Hn-boundary}.
    Then $P_\tau(u,v)=0$ unless $u$ and $v$ lie on the same boundary component of $\tau$. If they do, we have
    \begin{align}\label{eq-matrix-P-tau}
        P_\tau(u,v)=
        \begin{cases}
            -1 & u=v,\\
            \frac{1}{2} & \text{if there is an arrow between $u$ and $v$,}\\
            0 & \text{otherwise.}
        \end{cases}
    \end{align}
    Define ${\bf d}_\tau={\bf b}_\tau P_\tau$.
    Then
    \begin{equation}\label{eq-bQh-P}
        \begin{split}
            {\bf b}_\tau Q_\tau {\bf h}_\tau^T
            &={\bf b}_\tau H_\tau^T {\bf h}_\tau^T +
            {\bf b}_\tau P_\tau {\bf h}_\tau^T\\
            &= {\bf b}_\tau P_\tau {\bf h}_\tau^T \in\mathbb Z_n 
            \quad \text{($\because$ Lemma~\ref{lem-matrix-HK}(c) and ${\bf h}_\tau\in\mathcal B_\tau$)}\\
            &=\sum_{\substack{\text{$u$ and $v$ lie on the same}\\ \text{boundary component of $\tau$}}} {\bf b}_\tau (u) P_\tau(u,v) {\bf h}_\tau (v)\in\mathbb Z_n\\
            &=\sum_{v\in\pV_\tau} {\bf d}_\tau(v){\bf h}_\tau(v)\in\mathbb Z_n.
        \end{split}
    \end{equation}

Equations \eqref{eq-kh-QQQQQ}, \eqref{eq-ah-n-bQh}, \eqref{eq-ah-tau}, and \eqref{eq-bQh-P} together yield
\begin{align}\label{eq-kkkkkkk}
   \frac{1}{n}\, {\bf k}_\tau Q_\tau {\bf h}_\tau^T=\frac{1}{n}
   \sum_{e\in\lambda}\left(
   \sum_{\tau\in\mathbb F_\lambda\text{ adjacent to $e$}}\left(
   \sum_{\text{small vertex }v\in e}
   {\bf a}_\tau(v) {\bf h}_\tau(v) + n {\bf d}_\tau(v) {\bf h}_\tau(v)
   \right)\right)
   \in\mathbb Z_n.
\end{align}

Let $e\in\lambda$, and orient $e$.
Let $\tau'$ and $\tau''$ be the two triangles adjacent to $e$, with $\tau'$ lying on the right when traversing $e$ according to this orientation.  
Label the small vertices of $e$ by $u_1,\dots,u_{n-1}$ in the order induced by the orientation of $e$ (see Figure~\ref{Fig;labeling-uwx}).  
Label the small vertices in $\tau'$ (resp. $\tau''$) connecting $u_1,\dots,u_{n-1}$ by
$x_1,\cdots,x_n$ (resp. $w_1,\cdots,w_n$)
(see Figure~\ref{Fig;labeling-uwx}).
The three punctures of $\tau'$ (resp. $\tau''$) are denoted by $v_1',v_2',v_3'$ (resp. $v_1'',v_2'',v_3''$) (see Figure~\ref{Fig;labeling-uwx}).  

\begin{figure}[h]
    \centering
    \includegraphics[width=100pt]{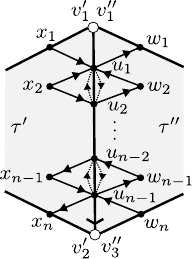}
    \caption{Labeling of small vertices in $\tau'$ and $\tau''$.}\label{Fig;labeling-uwx}
\end{figure}

Suppose that 
\begin{align}
    {\bf k}_{\tau'}=s_1'{\bf k}_1 + s_2'{\bf k}_2 + s_3'{\bf k}_3 + n{\bf b}_{\tau'}
    \text{ and }
    {\bf k}_{\tau''}=s_1''{\bf k}_1 + s_2''{\bf k}_2 + s_3''{\bf k}_3 + n{\bf b}_{\tau''}.
\end{align}
Define 
\begin{align*}
    {\bf c}_{\tau'}= s_1'{\bf k}_1 + s_2'{\bf k}_2 + s_3'{\bf k}_3,\quad
    {\bf a}_{\tau'}={\bf c}_{\tau'} Q_{\tau'};\\
    {\bf c}_{\tau''}= s_1''{\bf k}_1 + s_2''{\bf k}_2 + s_3''{\bf k}_3,\quad
    {\bf a}_{\tau''}={\bf c}_{\tau''} Q_{\tau''}.
\end{align*}
Note that the quiver $\Gamma_{\tau'}$ (resp. $\Gamma_{\tau''}$) is the left (resp. right) part of Figure~\ref{Fig;labeling-uwx}.
For each $1\leq i\leq n-1$, we have
\begin{align*}
    {\bf a}_{\tau'}(u_i)
    &= {\bf c}_{\tau'}(x_i) - {\bf c}_{\tau'}(x_{i+1})
    +\frac{1}{2} {\bf c}_{\tau'}(u_{i+1})
    -\frac{1}{2} {\bf c}_{\tau'}(u_{i-1})
    =s_1'-s_2'+\frac{1}{2} {\bf c}_{\tau'}(u_{i+1})
    -\frac{1}{2} {\bf c}_{\tau'}(u_{i-1}),\\
    {\bf a}_{\tau''}(u_i)
    &=-{\bf c}_{\tau''}(w_i) + {\bf c}_{\tau''}(w_{i+1})
    -\frac{1}{2} {\bf c}_{\tau''}(u_{i+1})
    +\frac{1}{2} {\bf c}_{\tau''}(u_{i-1})
    =-s_1''+s_3''-\frac{1}{2} {\bf c}_{\tau''}(u_{i+1})
    +\frac{1}{2} {\bf c}_{\tau''}(u_{i-1}),
\end{align*}
where ${\bf c}_{\tau'}(u_0) = {\bf c}_{\tau'}(u_n)={\bf c}_{\tau''}(u_0)= {\bf c}_{\tau''}(u_n)=0$.

Define 
\begin{align*}
    {\bf d}_{\tau'}= {\bf b}_{\tau'} P_{\tau'},\qquad
    {\bf d}_{\tau''}= {\bf b}_{\tau''} P_{\tau''}.
\end{align*}
For each $1\leq i\leq n-1$, we have
\begin{align*}
    {\bf d}_{\tau'}(u_i)
    &=-{\bf b}_{\tau'}(u_i) +\frac{1}{2}{\bf b}_{\tau'}(u_{i+1}) + \frac{1}{2}{\bf b}_{\tau'}(u_{i-1}),\\
    {\bf d}_{\tau''}(u_i)
    &=-{\bf b}_{\tau''}(u_i) +\frac{1}{2}{\bf b}_{\tau''}(u_{i+1}) + \frac{1}{2}{\bf b}_{\tau''}(u_{i-1}),
\end{align*}
where ${\bf b}_{\tau'}(u_0) = {\bf b}_{\tau'}(u_n)={\bf b}_{\tau''}(u_0)= {\bf b}_{\tau''}(u_n)=0$.

For each $1\leq i\leq n-1$, we have
$${\bf c}_{\tau'}(u_i) + n{\bf b}_{\tau'} (u_i)= {\bf c}_{\tau''}(u_i) + n{\bf b}_{\tau''} (u_i).$$ 
Hence, 
\begin{align}\label{eq-n-i-s-c}
(n-i)(s_1'-s_1'') + i(s_2'-s_3'')=
    {\bf c}_{\tau'}(u_i)- {\bf c}_{\tau''}(u_i)
    =n({\bf b}_{\tau''}(u_i)- {\bf b}_{\tau'}(u_i)).
\end{align}

It follows that
\begin{equation}\label{eq-1111111}
    \begin{split}
        &{\bf a}_{\tau'}(u_i) + n{\bf d}_{\tau'}(u_i) + {\bf a}_{\tau''}(u_i) + n{\bf d}_{\tau''}(u_i)\\
        =& s_1'-s_2'-s_1''+s_3''+\frac{n}{2}\big({\bf b}_{\tau''}(u_{i+1})-{\bf b}_{\tau'}(u_{i+1})\big)-
        \frac{n}{2}\big({\bf b}_{\tau''}(u_{i-1})-{\bf b}_{\tau'}(u_{i-1})\big)\\
        &-n{\bf b}_{\tau'}(u_i) - n{\bf b}_{\tau''}(u_i)+\frac{n}{2}{\bf b}_{\tau'}(u_{i+1}) + \frac{n}{2}{\bf b}_{\tau'}(u_{i-1})
        +\frac{n}{2}{\bf b}_{\tau''}(u_{i+1}) + \frac{n}{2}{\bf b}_{\tau''}(u_{i-1})\\
        =& s_1'-s_2'-s_1''+s_3'' + 
        n({\bf b}_{\tau'}(u_{i-1}) - {\bf b}_{\tau'}(u_{i})) + 
        n({\bf b}_{\tau''}(u_{i+1}) - {\bf b}_{\tau'}(u_{i})).
    \end{split}
\end{equation}

Equation~\eqref{eq-n-i-s-c} shows that
\begin{align}\label{s1-s2-s3-s4}
   s_1'-s_2'-s_1''+s_3''\in n\mathbb Z.
\end{align}
Suppose that
\begin{align}\label{eq-h-k1-k2-k3}
    {\bf h}_{\tau'} = c_1{\bf k}_1 + c_2{\bf k}_2+
c_3{\bf k}_3\in\mathbb Z_n^{V_{\tau'}}.
\end{align}
Then
\begin{equation}\label{key111}
    \begin{split}
        &\frac{1}{n}\sum_{1\leq i\leq n-1}({\bf a}_{\tau'}(u_i) + n{\bf d}_{\tau'}(u_i)){\bf h}_{\tau'}(u_i)
        + \frac{1}{n}\sum_{1\leq i\leq n-1}({\bf a}_{\tau''}(u_i) + n{\bf d}_{\tau''}(u_i)){\bf h}_{\tau''}(u_i)\\
        =& \frac{1}{n}\sum_{1\leq i\leq n-1}\Big({\bf a}_{\tau'}(u_i) + n{\bf d}_{\tau'}(u_i) + {\bf a}_{\tau''}(u_i) + n{\bf d}_{\tau''}(u_i)
        \Big){\bf h}_{\tau'}(u_i)
        \quad\text{($\because {\bf h}_{\tau'}(u_i)={\bf h}_{\tau''}(u_i)$)}\\
        =&\sum_{1\leq i\leq n-1}\Big({\bf b}_{\tau'}(u_{i-1}) - {\bf b}_{\tau'}(u_{i})+{\bf b}_{\tau''}(u_{i+1}) - {\bf b}_{\tau'}(u_{i})
        \Big){\bf h}_{\tau'}(u_i)\\
        &+ \frac{1}{n}(s_1'-s_2'-s_1''+s_3'')\sum_{1\leq i\leq n-1} {\bf h}_{\tau'}(u_i)
        \quad\text{($\because$\eqref{eq-1111111})}.
    \end{split}
\end{equation}

Next, we compute the two terms separately. First,
\begin{equation}\label{key222}
    \begin{split}
        &\sum_{1\leq i\leq n-1}\Big({\bf b}_{\tau'}(u_{i-1}) - {\bf b}_{\tau'}(u_{i})+{\bf b}_{\tau''}(u_{i+1}) - {\bf b}_{\tau'}(u_{i})
        \Big){\bf h}_{\tau'}(u_i)\\
        =&(c_2-c_1)\sum_{1\leq i\leq n-1}i\Big({\bf b}_{\tau'}(u_{i-1}) - {\bf b}_{\tau'}(u_{i})+{\bf b}_{\tau''}(u_{i+1}) - {\bf b}_{\tau'}(u_{i})
        \Big)\in\mathbb Z_n
        \quad\text{($\because$\eqref{eq-h-k1-k2-k3})}\\
        =& (c_2-c_1)\sum_{1\leq i\leq n-1}
        \big({\bf b}_{\tau'}(u_{i})- {\bf b}_{\tau''}(u_{i}) \big)\in\mathbb Z_n\\
        =& \frac{1}{n}(c_2-c_1)\left(
        (s_1''-s_1')\sum_{1\leq i\leq n-1}(n-i)
        +(s_3''-s_2')\sum_{1\leq i\leq n-1}i
        \right)\in\mathbb Z_n
        \quad\text{($\because$\eqref{eq-n-i-s-c})}\\
        =& \frac{n-1}{2}(c_2-c_1)
        (s_1''-s_1'+s_3''-s_2')\in\mathbb Z_n.
    \end{split}
\end{equation}

On the other hand,
\begin{equation}\label{key333}
    \begin{split}
        &\frac{1}{n}(s_1'-s_2'-s_1''+s_3'')\sum_{1\leq i\leq n-1} {\bf h}_{\tau}(u_i)\\
        =&\frac{1}{n}(s_1'-s_2'-s_1''+s_3'')\sum_{1\leq i\leq n-1}i(c_2-c_1)\in\mathbb Z_n
        \quad\text{($\because$\eqref{s1-s2-s3-s4} and \eqref{eq-h-k1-k2-k3})}\\
        =& \frac{n-1}{2}(c_2-c_1)
        (s_1'-s_2'-s_1''+s_3'')\in\mathbb Z_n.
    \end{split}
\end{equation}

Combining \eqref{key111}, \eqref{key222}, and \eqref{key333}, we obtain
\begin{align*}
    &\frac{1}{n}\sum_{1\leq i\leq n-1}({\bf a}_{\tau'}(u_i) + n{\bf d}_{\tau'}(u_i)){\bf h}_{\tau'}(u_i)
        + \frac{1}{n}\sum_{1\leq i\leq n-1}({\bf a}_{\tau''}(u_i) + n{\bf d}_{\tau''}(u_i)){\bf h}_{\tau''}(u_i)\\
        =&(c_2-c_1)(s_2'-s_3'')
        =(c_2-c_1)(s_1'-s_1'')\in\mathbb Z_n.
\end{align*}

Note that the value $(c_2-c_1)(s_1'-s_1'')$ depends on $e$.  
Therefore, Equation~\eqref{eq-kkkkkkk} shows
\begin{align}\label{prod-k-k-h-Q}
     \frac{1}{n}\, {\bf k}_\tau Q_\tau {\bf h}_\tau^T = \sum_{e\in\lambda}\Big(
     (c_2-c_1)(s_1'-s_1'') \Big)\in\mathbb Z_n.
\end{align}

Finally, observe that the algebraic intersection number between $\zeta({\bf k})$ and $\zeta({\bf h})$ in Figure~\ref{Fig;Zn-number} is precisely $(c_2-c_1)(s_1'-s_1'')$. Thus,
\begin{align}\label{n-prod-k-k-h-Q}
\langle \zeta({\bf k}), \zeta({\bf h}) 
\rangle_n=
    \sum_{e\in\lambda}\Big(
     (c_2-c_1)(s_1'-s_1'') \Big).
\end{align}

Equations~\eqref{prod-k-k-h-Q} and \eqref{n-prod-k-k-h-Q} complete the proof.

\begin{figure}[h]
    \centering
    \includegraphics[width=100pt]{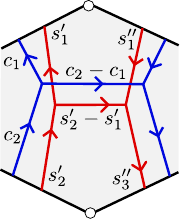}
    \caption{Local configuration of $\zeta({\bf k})$ and $\zeta({\bf h})$ near $e$.}\label{Fig;Zn-number}
\end{figure}
    
\end{proof}

Before discussing Theorem~\ref{thm-rank-Z}, we first prove Corollary~\ref{cor-rank-Bd-la} and Lemmas~\ref{lem-summand-B} and \ref{rank-111}, which will be used in the proof of Theorem~\ref{thm-rank-Z}.

Let $d$ be a positive integer with $d\mid n$.
Define
\begin{align}\label{Hnnd}
    \Hnn_d:=
    \{C\in\Hnn\mid \langle C, D\rangle_n \in d\mathbb Z_n \text{ for any $D\in\Hnn$}\}.
\end{align}
Let $g$ be the genus of $\fS$.
It is well-known that there exists a $\mathbb Z_n$-basis
$\{C_j,D_j\in\Hnn \mid 1\leq j\leq g\}$ of $\Hnn$ such that
$$\langle C_i, D_j\rangle_n= \delta_{ij},
\quad \langle C_i, C_j\rangle_n= 0,\quad\text{and }
\langle D_i, D_j\rangle_n= 0$$
for $1\leq i,j\leq g$.
Hence, we have 
\begin{align}\label{Hnnnnnd}
    \Hnn_d=d\Hnn.
\end{align}

\begin{corollary}\label{cor-rank-Bd-la}
    Let $d$ be a positive integer with
    $d\mid n$. Then we have 
    $$\left| \dfrac{\mathcal B_\lambda}
    {\mathcal B_{\lambda,d}}\right |=d^{2g},$$ where $g$ is the genus of $\fS$.
\end{corollary}
\begin{proof}

    Proposition~\eqref{prop-bilinear-key} shows that there is a group isomorphism
    between $\dfrac{\mathcal B_\lambda}
    {\mathcal B_{\lambda,d}}$ and 
    $\dfrac{\Hnn}
    {\Hnn_d}$.
    Then we have 
    \begin{align*}
        \left| \dfrac{\mathcal B_\lambda}
    {\mathcal B_{\lambda,d}}\right |
    =\left| \dfrac{\Hnn}
    {\Hnn_d}\right |
=\left| \dfrac{\Hnn}
    {d\Hnn}\right |=d^{2g},
    \end{align*}
    where the second equality comes from \eqref{Hnnnnnd}.
\end{proof}

\def\Zee{\mathcal Z_{\hat\eta}^{\rm bl}(\fS,\lambda)}
\def\Zoo{\mathcal Z_{\hat\omega}^{\rm bl}(\fS,\lambda)}

\begin{lemma}\label{lem-summand-B}
    $\BBB^{\circ}$ is a direct summand of $\mathcal B_{\lambda,d}$,
    where $\BBB^\circ$ is defined in  \eqref{def-B-lambda-circ}.
\end{lemma}
\begin{proof}
Obviously, we have $\BBB^\circ\subset\mathcal B_{\lambda,d}$.

     Define the homomorphism
\[
f\colon \mathcal B_{\lambda,d}\rightarrow \mathbb Z^{V_\lambda},\qquad 
{\bf k}\mapsto {\bf k}Q_\lambda.
\]
Proposition~\ref{Prop-central-generic} shows that 
\[
0\rightarrow \BBB^{\circ}\rightarrow
\mathcal B_{\lambda,d}\xrightarrow{f} \im f\longrightarrow 0
\]
is a short exact sequence. Then this short exact sequence is split since $\im f$ is a free abelian group.  
This completes the proof. 
\end{proof}

\begin{lemma}\label{rank-111}
    We have 
    $$\left|\frac{\mathcal B_{\lambda,d}}{N\mathcal B_{\lambda,d}+\mathcal B_{\lambda,d}^\circ}\right|= N^{2(n^2-1)(g-1) + n(n-1)m},$$
    where $g$ is the genus of $\fS$ and $m$ is the number of punctures in $\fS$.
\end{lemma}
\begin{proof}
Note that $(n\mathbb Z)^{V_\lambda}\subset \mathcal B_{\lambda,d}\subset \mathbb Z^{V_\lambda}$. Then $\mathcal B_{\lambda,d}$ is a free abelian group of rank $|V_\lambda|$.

    We have
\begin{align*}
   \left|\frac{\mathcal B_{\lambda,d}}{N\mathcal B_{\lambda,d}+\mathcal B_{\lambda,d}^\circ}\right|
   &=\left|\frac{\mathcal B_{\lambda,d}/\mathcal B_{\lambda,d}^\circ}{(N\mathcal B_{\lambda,d}+\mathcal B_{\lambda,d}^\circ)/\mathcal B_{\lambda,d}^\circ}\right|
   =\left|\frac{\mathcal B_{\lambda,d}/\mathcal B_{\lambda,d}^\circ}{N(\mathcal B_{\lambda,d}/\mathcal B_{\lambda,d}^\circ)}\right|
   =N^{\,|V_\lambda|-(n-1)m},
\end{align*}
where the last equality uses Corollary~\ref{cor-center-generic}(a) and Lemma~\ref{lem-summand-B}.  
Finally, \cite[equation (23)]{LY23} shows that 
\begin{align}\label{eq-V-lambda-rank}
    |V_\lambda|=2(n^2-1)(g-1)+(n^2-1)m,
\end{align}
where $V_\lambda$ in this paper corresponds to $\overline{V}_\lambda$ in \cite{LY23}.  
This completes the proof.
\end{proof}

The following theorem describes the rank of $\mathcal Z_{\hat\omega}^{\rm bl}(\fS,\lambda)$ over its center, representing the second main result of this paper.  
For the cases $n=2,3$, this rank coincides with the rank of the $\mathrm{SL}_n$-skein algebra over its center, which in turn equals the square of the maximal dimension of its irreducible representations \cite{kim2024unicity,unicity,frohman2021dimension}.  
Moreover, the techniques developed in the proof of this theorem may be applied to compute the rank of the $\mathrm{SL}_n$-skein algebra over its center.

\begin{theorem}\label{thm-rank-Z}
    Let $\fS$ be a triangulable connected punctured surface with a triangulation $\lambda$.
    Assume that we have the assumption \Rlabel{1}. Then $\Zoo$ is a free module over its center of rank 
    $$d^{2g} N^{2(n^2-1)(g-1) + n(n-1)m},$$ where $m$ is the number of punctures in $\fS$ and $g$ is the genus of $\fS$.
\end{theorem}
\begin{proof}
    Due to equation \eqref{eq-rank-quotient-equality}, it suffices to show 
    \begin{align}\label{eq-rank-quotient-equality-a}
     \left|\dfrac{\mathcal B_\lambda}{N \mathcal B_{\lambda,d} +  \mathcal B_\lambda^{\circ}} \right|= d^{2g} N^{2(n^2-1)(g-1) + n(n-1)m}.
 \end{align}
We have 
\begin{align*}
    \left|\dfrac{\mathcal B_\lambda}{N \mathcal B_{\lambda,d} +  \mathcal B_\lambda^{\circ}} \right|
    = \left|\dfrac{\mathcal B_\lambda}{\mathcal B_{\lambda,d}} \right|
    \left|\dfrac{\mathcal B_{\lambda,d}}{N \mathcal B_{\lambda,d} +  \mathcal B_\lambda^{\circ}} \right|
\end{align*}
Then Corollary~\ref{cor-rank-Bd-la} and 
Lemma~\ref{rank-111} complete the proof.
 
\end{proof}

We conclude this subsection with a remark that, while not directly related to the present paper, is of independent interest.

\begin{remark}\label{rem-d-trace}
    Let $\fS$ be a triangulable punctured surface with a triangulation $\lambda$.
An $n$-web $\alpha$ in $\fS \times (-1,1)$ naturally determines a homology class $[\alpha] \in H_1(\overline{\fS} \times (-1,1), \mathbb Z_n) = H_1(\overline{\fS}, \mathbb Z_n)$ as follows.  
For each loop component of $\alpha$, choose a point and declare it to be a vertex. In this way, $\alpha$ becomes a directed graph, and hence defines a $1$-chain in $\overline{\fS} \times (-1,1)$ with integer coefficients. This $1$-chain is closed modulo $n$, and thus determines an element of $H_1(\overline{\fS} \times (-1,1), \mathbb Z_n)$, independent of the choice of vertices on loop components.  

Define
\[
    \cS_{1}(\fS)_d := \operatorname{span}_{\mathbb C}\bigl\{
        \text{$n$-web $\alpha$ in $\fS \times (-1,1)$ } \,\bigm|\,
        [\alpha] \in \Hnn_d
    \bigr\}.
\]
It is well known that $\cS_{1}(\fS)_d$ is a subalgebra of $\cS_{\bar\omega}(\fS)$ \cite{wang2024TQFT,KLW}.  

Cut $\fS$ into a collection of triangles using all ideal arcs of $\lambda$.  
By combining Theorem~\ref{thm-trace-cut}, Proposition~\ref{prop-bilinear-key}, and \cite[Theorem~10.5]{LY23}, one obtains that the quantum trace map
\[
    \tr \colon \cS_{1}(\fS) \longrightarrow \mathcal Z_1^{\rm bl}(\fS,\lambda)
\]
restricts to
\[
    \tr \colon \cS_{1}(\fS)_d \longrightarrow \mathcal Z_1^{\rm bl}(\fS,\lambda)_d,
\]
where $\mathcal Z_1^{\rm bl}(\fS,\lambda)_d$ is defined in \eqref{def-torus-d}.
\end{remark}

\subsection{The anti-symmetric matrix decomposition}
Let $\fS$ be a triangulable connected punctured surface with a triangulation $\lambda$.
We denote by $\mathcal P$ the set of punctures of $\fS$.  
Lemma~\ref{lem-summand-B} shows that
\[
\mathcal B_\lambda = \mathcal B_\lambda' \oplus \mathcal B_\lambda^\circ
\]
for some subgroup $\mathcal B_\lambda'$ of $\mathcal B_\lambda$.  
Corollary~\ref{cor-center-generic}(a) shows that
\[
\{{\bf b}(\lambda,p,i)\mid p\in\mathcal P, 1\leq i\leq n-1\}
\]
forms a basis of $\mathcal B^\circ$.  
Let $\{{\bf t}_i \mid 1\leq i\leq |V_\lambda|-(n-1)|\mathcal P|\}$ be a basis of $\mathcal B_\lambda'$.  
Then
\[
\{{\bf t}_i \mid 1\leq i\leq |V_\lambda|-(n-1)|\mathcal P|\} \cup \{{\bf b}(\lambda,p,i) \mid p\in\mathcal P, 1\leq i\leq n-1\}
\]
is a basis of $\mathcal B_\lambda$.

Let $C$ denote the anti-symmetric matrix of the bilinear form \eqref{eq-bracket-bal} restricted to $\mathcal B_\lambda'$ with respect to the basis $\{{\bf t}_i\}$.  
Proposition~\ref{Prop-central-generic} implies that the anti-symmetric matrix of the bilinear form on the full basis of $\mathcal B_\lambda$ is
\begin{align}\label{eq-C-O-OO}
\begin{pmatrix}
C & O \\
O & O
\end{pmatrix},
\end{align}
where $C$ is invertible.  
By \cite[Theorem 4.1]{newman1972integral}, there exists another basis $\{{\bf h}_i\mid 1\leq i\leq |V_\lambda|-(n-1)|\mathcal P|\}$ of $\mathcal B_\lambda'$ such that
\begin{align}\label{eq-decom-cccc}
(\langle {\bf h}_i, {\bf h}_j \rangle)_{ij} = \text{diag}\left\{
\begin{pmatrix}
0 & c_1\\
-c_1 & 0
\end{pmatrix},
\ldots,
\begin{pmatrix}
0 & c_r\\
-c_r & 0
\end{pmatrix}
\right\},
\end{align}
where $r=\frac{|V_\lambda|-(n-1)|\mathcal P|}{2}$, $c_i\in \mathbb N \setminus \{0\}$, and $c_j \mid c_{j+1}$ for $1\leq j\leq r-1$.

Suppose $\fS$ has genus $g$ and $m$ punctures.  
Then \eqref{eq-V-lambda-rank} gives
\[
r = (n^2-1)(g-1) + \frac{1}{2} n(n-1) m.
\]

\begin{corollary}\label{lam-decom}
Let $\fS$ be a triangulable connected punctured surface of genus $g$ with $m$ punctures, and let $\lambda$ be a triangulation of $\fS$.  
There exists a basis $\{{\bf h}_i \mid 1\leq i \leq |V_\lambda|\}$ of $\mathcal B_\lambda$ such that
\[
\{{\bf b}(\lambda,p,i) \mid p\in\mathcal P, 1\leq i\leq n-1\} = \{{\bf h}_i \mid |V_\lambda|-(n-1)m+1 \leq i \leq |V_\lambda|\},
\]
and the matrix $(\langle {\bf h}_i,{\bf h}_j \rangle)_{1\leq i,j\leq |V_\lambda|}$ has the block-diagonal form
\[
\text{diag}\Bigl\{
\begin{pmatrix}
0 & s_1\\
-s_1 & 0
\end{pmatrix},
\ldots,
\begin{pmatrix}
0 & s_r\\
-s_r & 0
\end{pmatrix}, 0, \ldots, 0
\Bigr\},
\]
where $r=(n^2-1)(g-1) + \frac{1}{2} n(n-1) m$, and
\[
s_i = \begin{cases}
1 & 1\leq i \leq g,\\
n & g < i \leq r.
\end{cases}
\]
\end{corollary}

\begin{remark}
    For $n=2$, Corollary~\ref{lam-decom} was proved in \cite{representation2,BL07} by different methods.  
\end{remark}

Before proving Corollary~\ref{lam-decom}, we establish the following lemma, which is motivated by \cite[Lemma 8.9]{KaruoWangToAppear}.

\begin{lemma}\label{lem-key-eq}
Let $(k_1,\dots,k_r)$ and $(l_1,\dots,l_r)$ be sequences of positive integers such that $k_i\mid k_{i+1}$ and $l_i\mid l_{i+1}$ for $1\leq i\leq r-1$.  
If
\[
\prod_{i=1}^r \frac{m}{\gcd(k_i,m)} = \prod_{i=1}^r \frac{m}{\gcd(l_i,m)}
\]
for all positive integers $m$, then $k_i = l_i$ for $1\leq i \leq r$.
\end{lemma}

\begin{proof}
We prove the lemma by induction on $r$.  
The case $r=1$ is trivial.  
Assume the statement holds for $r=k-1$ and consider $r=k$.  

Note that
\[
\prod_{i=1}^r \frac{m}{\gcd(k_i,m)} = \prod_{i=1}^r \frac{m}{\gcd(l_i,m)}
\quad\Longleftrightarrow\quad
\prod_{i=1}^r \gcd(k_i,m) = \prod_{i=1}^r \gcd(l_i,m).
\]
Setting $m=k_r l_r$, we obtain $\prod_{i=1}^r k_i = \prod_{i=1}^r l_i$.  
Setting $m=k_r$, it follows that $l_i = \gcd(l_i, k_r)$ for $1\leq i \leq r$, hence $l_r\mid k_r$.  
Similarly, $k_r\mid l_r$, so $l_r = k_r$.  
Then the induction hypothesis gives $k_i=l_i$ for $1\leq i \leq r-1$.  
\end{proof}

\begin{proof}[Proof of Corollary~\ref{lam-decom}]
By \eqref{eq-C-O-OO}, it suffices to show that
\begin{align}\label{value-c-i}
    c_i = \begin{cases}
1 & 1\leq i\leq g,\\
n & g < i \leq r,
\end{cases}
\end{align}
where $c_i$ is defined in \eqref{eq-decom-cccc}.  

Assuming \Rlabel{1}, \cite[Equation (8.2)]{KaruoWangToAppear} implies that the rank of $\mathcal Z_{\hat\omega}^{\rm bl}(\fS,\lambda)$ over its center is
\[
\left(\prod_{i=1}^r \frac{N'}{\gcd(c_i, N')}\right)^2.
\]
By Theorem~\ref{thm-rank-Z}, we have
\begin{align}\label{rank-c-i}
    \prod_{i=1}^r \frac{N'}{\gcd(c_i, N')} = d^{2g} N^{2(n^2-1)(g-1) + n(n-1)m}.
\end{align}
When the value of $c_i$ is given by \eqref{value-c-i}, equation~\eqref{rank-c-i} holds.
Applying Lemma~\ref{lem-key-eq} completes the proof since $N'$ is arbitrary.
\end{proof}

\def\hom{\text{Hom}_{\text{Alg}}}
\def\Zee{\mathcal Z_{\hat\eta}^{\rm bl}(\fS,\lambda)}
\def\Zoo{\mathcal Z_{\hat\omega}^{\rm bl}(\fS,\lambda)}
\def\Xee{\mathcal X_{\eta}(\fS,\lambda)}
\def\Xoo{\mathcal X_{\omega}^{\rm bl}(\fS,\lambda)}
\def\Zep{\mathcal Z_{\hat\eta}^{\rm bl}(\fS,\lambda')}
\def\Zop{\mathcal Z_{\hat\omega}^{\rm bl}(\fS,\lambda')}
\def\Xep{\mathcal X_{1}^{\rm bl}(\fS,\lambda')}
\def\Xop{\mathcal X_{\omega}^{\rm bl}(\fS,\lambda')}

\def\BC{\mathbb{C}}
\def\End{\text{End}_{\BC}}
\def\id{{\rm id}}

\def\zll{\mathcal Z_{\hat\omega}^\Lambda(\fS,\lambda)}
\def\zell{\mathcal Z_{\hat\eta}^\Lambda(\fS,\lambda)}

\def\zmm{\mathcal Z_{\hat\omega}^{\rm mbl}(\fS,\lambda)}
\def\zem{\mathcal Z_{\hat\eta}^{\rm mbl}(\fS,\lambda)}
\def\zmp{\mathcal Z_{\hat\omega}^{\rm mbl}(\fS,\lambda')}
\def\zemp{\mathcal Z_{\hat\eta}^{\rm mbl}(\fS,\lambda')}

\section{Representations of the related Fock-Goncharov algebras}\label{sec-rep-torus-b}

In this section,
we classify the irreducible representations of the balanced Fock–Goncharov algebra under the assumption \Rlabel{1} (Theorem~\ref{thm-representation-Fock}).   
In \S\ref{sec-rep-sln}, we will combine Theorem~\ref{thm-representation-Fock} with the quantum trace map to construct irreducible representations of the projected $\SL$-skein algebra.

\def\OOOO{\mathcal Z_{\hat{\omega}}^{\rm bl}(\fS,\lambda)}
\def\EEEE{\mathcal Z_{\hat{\eta}}^{\rm bl}(\fS,\lambda)}

\subsection{Representations of the balanced Fock-Goncharov algebras}

Let $\fS$ be a triangulable punctured surface with a triangulation $\lambda$.
We use $\mathcal P$ to denote the set of punctures of $\fS$.

Suppose that we have the assumption \Rlabel{1}. Recall that there is an algebra embedding
\begin{align}\label{AAAAAAAAAAAA}
\PT\colon\EEEE\rightarrow\OOOO,\quad
Z^{\bf k}\mapsto Z^{N{\bf k}}\text{ for }{\bf k}\in\mathcal B_\lambda.
\end{align}
For any two algebras $A_1$ and $A_2$, we use $\text{Hom}_{\text{Alg}}(A_1,A_2)$ to denote the set of algebra homomorphisms from $A_1$ to $A_2$.
Define 
\begin{align}\label{def-Hom-lambda}
    \text{Hom}_Z:=\{(f,(b(p,i))_{1\leq i\leq n-1,p\in\mathcal P})\mid
f\in \hom(\EEEE_d,\mathbb C), \\
b(p,i)^N=f(Z^{{\bf b}(\lambda,p,i)})\in\mathbb C^{*}\text{ for }1\leq i\leq n-1, p\in\mathcal P\},
\end{align}
where ${\bf b}(\lambda,p,i)$ is defined as in \eqref{def-a-l-p-i} and $\EEEE_d$ is defined in \eqref{def-torus-d}.

\begin{lemma}\label{cor-center-bl-equal}
     Suppose that we have the assumption \Rlabel{1}. Then
     the following map is a bijection
         \begin{align}\label{F-bi-center-ch}
              \hom(\mathsf{Z}(\OOOO),\mathbb C)\rightarrow \text{Hom}_Z,\quad
             g\mapsto (g\circ\PT, (g(Z^{{\bf b}(\lambda,p,i)}))_{1\leq i\leq n-1,p\in\mathcal P}).
         \end{align}
\end{lemma}
\begin{proof}

Recall that we defined $B_{N''}$ in \eqref{eq-define-subgroup}, $\mathcal B_{\lambda,d}$ in \eqref{def-B-lambda-m}, and $\mathcal B_{\lambda}^\circ$ in \eqref{def-B-lambda-circ}.

    Since $B_{N''}=N\mathcal B_{\lambda,d}+\mathcal B_\lambda^\circ$ (Theorem~\ref{thm-center-balanced-root-of-unity}) and $\mathcal B_\lambda^\circ$ is a direct summand of $\mathcal B_{\lambda,d}$ (Lemma~\ref{lem-summand-B}), we have
\begin{align}\label{eq-group-relation}
   (N\mathcal B_{\lambda,d}\oplus\mathcal B_\lambda^\circ)/\big(({\bf t},0)=(0,N{\bf t})\ \text{for }{\bf t}\in \mathcal B_\lambda^\circ\big) \;\cong\; B_{N''}.
\end{align}
The group isomorphism in \eqref{eq-group-relation} is given by $({\bf t},{\bf s})\mapsto N{\bf t}+{\bf s}$.  
This group isomorphism induces the following algebra isomorphism:
\begin{align*}
   \dfrac{\EEEE_d\otimes \mathbb C[Z^{{\bf b}(\lambda,p,i)}\mid 1\leq i\leq n-1, p\in\mathcal P]}
   {Z^{{\bf b}(\lambda,p,i)}\otimes 1=1\otimes Z^{N{\bf b}(\lambda,p,i)},\ 1\leq i\leq n-1,p\in\mathcal P}
   \;&\cong\; \mathsf{Z}(\OOOO),\\
   X\otimes Y &\longmapsto \PT(X)Y.
\end{align*}
This algebra isomorphism yields the bijection in \eqref{F-bi-center-ch}.
\end{proof}

We are now ready to classify the irreducible representations of $\OOOO$ under the assumption \Rlabel{1}.  
The following result, which constitutes our third main theorem, establishes a bijection between the set of irreducible representations of $\OOOO$ (up to isomorphism) and the set $\text{Hom}_Z$ (see \eqref{def-Hom-lambda}).  
This theorem will be applied in \S\ref{sec-rep-sln} to construct irreducible representations of the projected $\SL$-skein algebra.

\begin{theorem}\label{thm-representation-Fock}
     Let $\fS$ be a triangulable connected punctured surface with genus $g$ and $m$ punctures, and let $\mathcal P$ be the set of punctures of $\fS$.
    Suppose that we have the assumption \Rlabel{1}.
    Let $f\in \hom(\EEEE_d,\BC)$ (see \eqref{def-torus-d} for $\EEEE_d$), and 
    let $b(p,i)\in\mathbb C^{*}$ satisfying $$b(p,i)^N=f(Z^{{\bf b}(\lambda,p,i)})$$ for $1\leq i\leq n-1$ and $p\in\mathcal P$, where ${\bf b}(\lambda,p,i)$ is defined as in \eqref{def-a-l-p-i}.
    There exists a unique irreducible representation $$\rho\colon \OOOO\rightarrow \End(V)$$ 
   satisfying the following conditions:
    \begin{enumerate}
        \item $\rho(\PT(X)) = f(X)\Id_V$ for 
        $X\in \EEEE_d,$ where $\PT$ is defined in \eqref{AAAAAAAAAAAA};
        \item $\rho(Z^{{\bf b}(\lambda,p,i)})
        =b(p,i)\Id_V$ for $1\leq i\leq n-1$ and $p\in\mathcal P$.
    \end{enumerate}
    In particular, the dimension of $V$ is 
    $$d^g N^{(n^2-1)(g-1) + \frac{1}{2}n(n-1)m}.$$
\end{theorem}
\begin{proof}
It follows from Lemmas~\ref{lem-irre-Azumaya}, \ref{cor-center-bl-equal}, and Theorem~\ref{thm-rank-Z}.
\end{proof}

\begin{remark}
    When $n=2$, Theorem~\ref{thm-representation-Fock} was proved in \cite{representation2} using different techniques.
\end{remark}

\subsection{On other related versions of Fock-Goncharov algebras}

Let $\fS$ be a triangulable connected punctured surface with a triangulation $\lambda$.
Lemma~\ref{lem-balanced-property} shows that 
$$\BB\subset \{{\bf k}\in \mathbb Z^{V_\lambda}\mid {\bf k}Q_\lambda\in (n\mathbb Z)^{V_\lambda}\}.$$
Let $\Lambda$ be a subgroup of $\mathbb Z^{V_\lambda}$ with $$\BB\subset\Lambda\subset \{{\bf k}\in \mathbb Z^{V_\lambda}\mid {\bf k}Q_\lambda\in (n\mathbb Z)^{V_\lambda}\}.$$
Since $n\mathbb Z^{V_\lambda}\subset\Lambda\subset \mathbb Z^{V_\lambda}$, we have that $\Lambda$ is free abelian group of rank $|V_\lambda|$.
Define 
\begin{align}\label{def-Z-lambda}
    \mathcal Z_{\hat\omega}^\Lambda(\fS,\lambda)=\text{span}_{\BC}\{Z^\bk\mid \bk\in\Lambda\}\subset
\mathcal{Z}_{\hat\omega}(\fS,\lambda).
\end{align}
Then $\mathcal Z_{\hat\omega}^\Lambda(\fS,\lambda)$ is a subalgebra of $\mathcal{Z}_{\hat\omega}(\fS,\lambda)$.

In this subsection, we compute the dimension of irreducible representations of 
$\mathcal Z_{\hat\omega}^\Lambda(\fS,\lambda)$ under the assumption~\Rlabel{2} 
(Corollary~\ref{prop-representation-Fock}). 
This result will be applied later in \S\ref{sec-naturality}. 
Before proving Corollary~\ref{prop-representation-Fock}, we first compute the center 
$\mathsf{Z}(\mathcal Z_{\hat\omega}^\Lambda(\fS,\lambda))$ 
under the assumption~\Rlabel{2}.

Define  
\begin{align}\label{def-lambda-circc}
    \Lambda^{\circ}:=\{{\bf k}\in\Lambda\mid
    {\bf k}Q_\lambda= {\bf 0}\}.
\end{align}
The following lemma will be used to compute the rank of 
$\mathcal Z_{\hat\omega}^\Lambda(\fS,\lambda)$ over its center 
$\mathsf{Z}(\mathcal Z_{\hat\omega}^\Lambda(\fS,\lambda))$.

\begin{lemma}\label{lem-direct-lambda}
The following hold:
\begin{enumerate}[label={\rm (\alph*)}]
    \item $\Lambda^{\circ}$ is a free abelian group of rank $(n-1)m$, where $m$ is the number of punctures in $\fS$.
    \item $\Lambda^{\circ}$ is a direct summand of $\Lambda$.
\end{enumerate} 
\end{lemma}
\begin{proof}
(a) Clearly, $\Lambda^{\circ}$ is a subgroup of $\Lambda$. 
Since every subgroup of a free abelian group is itself free, 
it follows that $\Lambda^{\circ}$ is free abelian.  
Moreover, we have
\begin{align*}
    \BB^{\circ}\subset \Lambda^{\circ},
\end{align*}
where $\BB^{\circ}$ is defined in \eqref{def-B-lambda-circ}.
By Corollary~\ref{cor-center-generic}(a), 
$\BB^{\circ}$ is a free abelian group of rank $(n-1)m$.  
Hence the rank of $\Lambda^{\circ}$ is at least $(n-1)m$.

On the other hand, since $n\mathbb Z^{V_\lambda}\subset\BB$,  
Corollary~\ref{cor-center-generic}(a) implies that
\[
\dim_{\mathbb Q}\{{\bf k}\in \mathbb Q^{V_\lambda}\mid 
{\bf k} Q_\lambda={\bf 0}\}=(n-1)m.
\]
Thus the rank of $\Lambda^{\circ}$ is at most $(n-1)m$.  

(b) The proof of Lemma~\ref{lem-summand-B} applies here as well, with 
$\mathcal B_{\lambda,d}$ replaced by $\Lambda$ and $\BBB^\circ$ replaced by $\Lambda^\circ$.

\end{proof}

Suppose that we have the assumption \Rlabel{2}. There is an algebra embedding
\begin{align}\label{eq-fro-Z-lambda}
    \PT\colon\zell\rightarrow\zll,\quad
Z^{\bf k}\mapsto Z^{N{\bf k}}\text{ for }{\bf k}\in\Lambda.
\end{align}

\begin{lemma}\label{cor-center-rank}
     Suppose that we have the assumption \Rlabel{2}. Then:
     \begin{enumerate}[label={\rm (\alph*)}]
         \item As a subalgebra of $\zll$, the center $\mathsf{Z} (\zll)$ is generated by $\im \PT$ and $\{Z^{\bf k}\mid {\bf k}\in\Lambda^\circ\}$.

         \item $\zll$ is a free module over its center $\mathsf{Z}(\zll)$ of rank $N^{2(n^2-1)(g-1) + n(n-1)m}$, where $g$ is the genus of $\fS$ and $m$ is the number of punctures in $\fS$.

     \end{enumerate}
\end{lemma}

\begin{proof}
(a) By Lemma~\ref{lem-center-torus}, it suffices to show that the subgroup 
\begin{align*}
   \Lambda_{N''}:=\{{\bf k}\in\Lambda\mid {\bf k} Q_\lambda {\bf t}^T=0 \in \mathbb Z_{N''}\ \text{for all }{\bf t}\in\Lambda\}
\end{align*}
is generated by $N\Lambda$ and $\Lambda^\circ$.  
Clearly $\Lambda^\circ\subset \Lambda_{N''}$.  
Since $d=1$, we have $N'=N$, and hence $N\Lambda\subset \Lambda_{N''}$ because
$\Lambda \subset \{{\bf k}\in \mathbb Z^{V_\lambda}\mid {\bf k}Q_\lambda\in (n\mathbb Z)^{V_\lambda}\}$.

We first treat the case $\Lambda=\mathcal B_\lambda$.  
From the proof of Theorem~\ref{thm-center-balanced-root-of-unity}, one has
\[
 B_{N''}=N\mathcal B_{\lambda,1} + \mathcal B_{\lambda}^{\circ},
\]
where $B_{N''}$ is defined in \eqref{eq-define-subgroup}, $\mathcal B_{\lambda,1}$ in \eqref{def-B-lambda-m}, and $\mathcal B_{\lambda}^{\circ}$ in \eqref{def-B-lambda-circ}.  
Lemma~\ref{lem-balanced-property} shows that $\mathcal B_\lambda\subset \mathcal B_{\lambda,1}$. Hence 
\begin{align}\label{lambda-one-equal-bl}
    \text{$\mathcal B_\lambda=\mathcal B_{\lambda,1},$}
\end{align}
and consequently
\[
(\mathcal B_\lambda)_{N''}=B_{N''}=N\mathcal B_\lambda+\mathcal B_\lambda^{\circ}.
\]

By Lemma~\ref{lem-direct-lambda}(b), we may write 
$\Lambda=\Lambda^\circ \oplus \Lambda'$.  
Let ${\bf a}\in \Lambda_{N''}$ with ${\bf a}={\bf s}+{\bf t}$, where ${\bf s}\in \Lambda^\circ$ and ${\bf t}\in \Lambda'$.  
Since both ${\bf a}$ and ${\bf s}$ lie in $\Lambda_{N''}$, it follows that ${\bf t}\in \Lambda_{N''}$.  
Hence $n{\bf t}\in B_{N''}$, and so
\[
n{\bf t} = N{\bf t}_1 + {\bf t}_2,
\]
with ${\bf t}_1\in \mathcal B_\lambda \subset \Lambda$ and ${\bf t}_2\in \mathcal B_\lambda^\circ \subset \Lambda^\circ$.  
Decompose ${\bf t}_1={\bf t}_1'+{\bf t}_1''$ with ${\bf t}_1'\in \Lambda'$ and ${\bf t}_1''\in \Lambda^\circ$.  
Then
\[
n{\bf t}=N{\bf t}_1'+(N{\bf t}_1''+{\bf t}_2).
\]
Since $\Lambda=\Lambda^\circ\oplus\Lambda'$, this shows $n{\bf t}=N{\bf t}_1'$.  

Because $\gcd(N',n)=1$ with $N'=N$, there exist $u,v\in\mathbb Z$ such that $uN+vn=1$.  
Thus
\[
{\bf t}=(uN+vn){\bf t}=Nu{\bf t}+vn{\bf t}
= N(u{\bf t}+v{\bf t}_1').
\]
Therefore
\[
{\bf a}={\bf s}+N(u{\bf t}+v{\bf t}_1')\in \Lambda^\circ+N\Lambda.
\]

(c) By \cite[Lemma~3.2]{KaruoWangToAppear} and part (a), $\zll$ is a free module over $\mathsf{Z}(\zll)$ of rank
\[
\Bigl|\frac{\Lambda}{N\Lambda+\Lambda^\circ}\Bigr|.
\]
Then Lemma~\ref{lem-direct-lambda} and the proof of Lemma~\ref{rank-111} complete the proof.

\end{proof}

The following Corollary will be used to prove Theorem~\ref{thm-naturality}.

\begin{corollary}\label{prop-representation-Fock}
 
 Suppose that we have the assumption \Rlabel{2}. Every irreducible representation of $\zll$ has dimension 
 $$N^{(n^2-1)(g-1) + \frac{1}{2}n(n-1)m}.$$
\end{corollary}
\begin{proof}
    It follows from Lemma~\ref{lem-irre-Azumaya} and Lemma~\ref{cor-center-rank}(b).
\end{proof}

\begin{remark}
    Lemma~\ref{cor-center-rank} and Corollary~\ref{prop-representation-Fock} can be easily generalized to any subgroup $\Lambda$ satisfying 
    $\mathcal B_\lambda\subset \Lambda$ provided that $\gcd(N'',n)=1$ in the root of assumption \Rlabel{2}.
\end{remark}

\section{Representations of the projected ${\rm SL}_n$-skein algebra}\label{sec-rep-sln}
In this section, under assumption~\Rlabel{2}, we show that every irreducible representation of the projected $\SL$-skein algebra determines a point of the $\SL$-character variety (see \eqref{eq-Sl3-character}) via the Frobenius homomorphism in Theorem~\ref{thm-Fro} (Proposition~\ref{prop-classical-shadow}). This point is called the \emph{classical shadow} of the irreducible representation.  

When the surface is a triangulable punctured surface, we combine Theorem~\ref{thm-representation-Fock} with the quantum trace map to prove that for a `large' subset $U$ of the $\SL$-character variety, every element of $U$ arises as the classical shadow of some irreducible representation of the projected $\SL$-skein algebra (Theorem~\ref{thm-main-3}), which is the main result of this section.  
In \S\ref{sec-naturality}, we further show that, under mild conditions, the representation constructed in Theorem~\ref{thm-main-3} is independent of the choice of ideal triangulation $\lambda$.

\subsection{The image of peripheral loops under the quantum trace map}\label{subsec-image-loop}
Throughout this subsection, $\fS$ denotes a triangulable pb surface with a triangulation $\lambda$, and $\mathcal P$ denotes the set of interior punctures of $\fS$.

Let $\fS'$ be a pb surface such that $\partial \fS'$ has at least two components, and let $u$ be an oriented corner arc such that $u$ connects two different components of $\partial\fS'$.
Note that $u$ encircles a puncture, and we require its orientation to be counterclockwise around this puncture. See the left picture in Figure~\ref{Fig;badarc}.
For each pair $1\leq i,j\leq n$, we use $u_{ij}$ to denote the stated arc obtained from $u$ by equipping the endpoint (resp. starting point) of $u$ with the state $j$ (resp. $i$). See the left picture in Figure~\ref{Fig;badarc}.
We use $\bar u_{ij}$ to denote the image of the $u_{ij}$ under the projection  $\cS_{\hat\omega}^{\rm st}(\fS')\rightarrow \overline{\cS}_{\homega}^{\rm st}(\fS')$.
Then $\bar u_{ij}=0\in \overline{\cS}_{\homega}^{\rm st}(\fS')$ when $j>i$.

For each $1\leq k\leq n$, define $\mathbb J_k$ to be the set of $k$-element subsets of $\{1,2,\cdots,n\}$.
For each $I=\{1\leq i_1<\cdots < i_k\leq n\}\in\mathbb J_k$, define 
$$M_\omega^I(u):=\sum_{\sigma\in S_k}(-\omega)^{\ell(\sigma)} u_{i_1i_{\sigma(1)}}\cdots u_{i_ki_{\sigma(k)}}\in\cS_{\hat\omega}^{\rm st}(\fS'),$$
where $\ell(\sigma)=\#\{(i,j)\mid 1\leq i<j\leq n,\ \sigma(i)>\sigma(j)\}$ is the length of $\sigma$.
We use $\overline M_\omega^I(u)$ to denote the image of the $M_\omega^I(u)$ under the projection  $\cS_{\hat\omega}^{\rm st}(\fS')\rightarrow \overline{\cS}_{\homega}^{\rm st}(\fS')$. Then 
$$\overline M_\omega^I(u)=\bar u_{i_1i_1}\cdots
\bar u_{i_ki_k}\in \overline{\cS}_{\homega}^{\rm st}(\fS').$$

Let $p\in\mathcal P$. We use $\alpha(p)$ to denote the counterclockwise oriented crossingless diagram encircling $p$.
Recall that for each $1\leq k\leq n-1$, there exists an element $\left[ \alpha(p)_k\right]_{\bar\omega}\in\cS_{\bar\omega}(\fS)$ (see Definition \ref{def.threading_of_element}).

For any subset $I$ of $\{1,2,\cdots,n\}$, define 
\begin{align}
    {\bf c}(\lambda,p,I):=
    \sum_{i\in I} {\bf c}(\lambda,p,i),
\end{align}
where ${\bf c}(\lambda,p,i)$ is defined in
\eqref{def-a-l-p-i} with `${\bf a}$' replaced by `${\bf c}$'.
Set 
\begin{align}
    {\bf c}(\lambda,p):={\bf c}(\lambda,p,\{1,2,\cdots,n\}).
\end{align}

For $1\leq i\leq n$ and $p\in\mathcal P$, define
\begin{align}\label{def-vector-d}
    {\bf d}(\lambda,p,i):=n {\bf c}(\lambda,p,\bar i)
    -{\bf c}(\lambda,p).
\end{align}
Note that 
${\bf c}(\lambda,p,1)={\bf 0}$.
With the convention that ${\bf b}(\lambda,p,0)={\bf b}(\lambda,p,n)={\bf 0}$,
Equation \eqref{eq-vectors-b-c-new} implies
\begin{align}\label{eq-vectors-b-d}
    {\bf d}(\lambda,p,i)={\bf b}(\lambda,p,i)-{\bf b}(\lambda,p,i-1)
\end{align}
for $1\leq i\leq n$, where ${\bf b}(\lambda,p,i)$ is defined in
\eqref{def-a-l-p-i} with `${\bf a}$' replaced by `${\bf b}$'.

The following proposition describes the image of $\left[\alpha(p)_k\right]_{\bar\omega}$ under the quantum trace map.  
As noted for $n=2,3$, understanding 
$\tr\!\left(\left[\alpha(p)_k\right]_{\bar\omega}\right)$ is crucial for studying the center and representation theory of the $\SL$-skein algebra \cite{representation2,unicity,kim2024unicity}.  
This proposition will be applied in \S\ref{subsec-rep-skein} to construct irreducible representations of the projected $\SL$-skein algebra.

\begin{proposition}\label{lem-image-loop-trace}
    For each $p\in\mathcal P$ and $1\leq k\leq n-1$, we have    $$\tr\left(\left[\alpha(p)_k\right]_{\bar\omega}\right)=\sum_{I\in\mathbb J_k} Z^{n{\bf c}(\lambda,p,I)-k{\bf c}(\lambda,p)}=
    \sum_{I\in\mathbb J_k}\left( \left[ \prod_{i\in I} Z^{{\bf d}(\lambda,p,i)}\right] \right).$$
\end{proposition}
\begin{proof}
  Let $e$ be an ideal arc in $\lambda$ connecting to $p$.  
There are two cases:  
(1) the two endpoints of $e$ are distinct, and  
(2) both endpoints of $e$ coincide at $p$.  
We provide a detailed proof only for case (1), since a similar argument applies to case (2).  

Let $\fS'$ denote $\Cut_e(\fS)$ (see \eqref{def-pr-cut-fs}), and let $\lambda'$ denote $\lambda_e$ (see \eqref{induced-ideal-tri}).
By cutting $\alpha(p)$ along $e$, we obtain a corner arc $u$ in $\fS'$.  
This cutting induces a group embedding 
\[
\iota\colon \mathcal B_\lambda \longrightarrow \mathcal B_{\lambda'}.
\]  
We now state two lemmas that will be used in the proof.

\begin{lemma}\cite[Proposition 7.4]{KLW}\label{lem-cutting-image-loop}
    We have $$\mathbb S_e(\left[\alpha(p)_k\right]_{\bar\omega})
    =\sum_{I\in\mathbb J_k}\overline{M}_\omega^I(u)=\sum_{1\leq i_1<\cdots <i_k\leq n}\bar u_{i_1i_1}\cdots \bar u_{i_ki_k}\in\rS^{\rm st}(\fS').$$
\end{lemma}

\begin{lemma}\cite[Lemma 4.5 and Remark 4.6]{KimWang}\label{lem-image-corner-arc}
    For each $1\leq t\leq n$, we have 
    $${\rm tr}_{\lambda'} (\bar u_{tt})
    =Z^{\iota(n{\bf c}(\lambda,p,t)-{\bf c}(\lambda,p))}.$$
\end{lemma}

It is shown in \cite[Lemma~7.6(a)]{LY23} that 
$\bar u_{ii}\bar u_{jj}= \bar u_{jj} \bar u_{ii} \in \overline{\cS}_{\hat\omega}^{\rm st}(\fS')$ 
for all $1\leq i,j\leq n$.
Hence we have
\begin{align}\label{commute-P-fS-Z}
    {\rm tr}_{\lambda'}(\bar u_{ii}) \,
    {\rm tr}_{\lambda'}(\bar u_{jj})
    = 
    {\rm tr}_{\lambda'}(\bar u_{jj}) \,
    {\rm tr}_{\lambda'}(\bar u_{ii})\in \mathcal Z_{\hat\omega}^{\rm bl}(\fS',\lambda')
\end{align}
for all $1\leq i,j\leq n$.
We have
\begin{align*}
    \mathcal S_e(\tr(\left[\alpha(p)_k\right]_{\bar\omega}))
    &={\rm tr}_{\lambda'}(\mathbb S_e(\left[\alpha(p)_k\right]_{\bar\omega}))
    \quad(\because \mbox{Theorem \ref{thm-trace-cut}})\\
    &=\sum_{1\leq i_1<\cdots <i_k\leq n}{\rm tr}_{\lambda'}(\bar u_{i_1i_1})\cdots {\rm tr}_{\lambda'}(\bar u_{i_ki_k})
    \quad(\because \mbox{Lemma \ref{lem-cutting-image-loop}})\\
    &=\sum_{I\in\mathbb J_k} Z^{\iota(n{\bf c}(\lambda,p,I)-k{\bf c}(\lambda,p))}
    \quad(\because \mbox{Lemma \ref{lem-image-corner-arc} and \eqref{commute-P-fS-Z}})\\
    &=\mathcal S_e\left(\sum_{I\in\mathbb J_k} Z^{n{\bf c}(\lambda,p,I)-k{\bf c}(\lambda,p)}\right).
\end{align*}
Then we have $$\tr\left(\left[\alpha(p)_k\right]_{\bar\omega}\right)=\sum_{I\in\mathbb J_k} Z^{n{\bf c}(\lambda,p,I)-k{\bf c}(\lambda,p)}=
    \sum_{I\in\mathbb J_k}\left( \left[ \prod_{i\in I} Z^{{\bf d}(\lambda,p,i)}\right] \right)$$ because $\mathcal S_e$ is injective.

\end{proof}

In the remainder of this section, we assume that $\fS$ is a punctured surface.  
Before turning to the representation theory of the projected $\SL$-skein algebra, we first review the connection between the ${\rm SL}_n$-character variety of $\fS$ and $\cS_{\bar\eta}(\fS)$ in the case $\bar\eta=\pm 1$.
Then we will use this connection and the Frobenius homomorphism in Theorem~\ref{thm-Fro} to define the classical shadow of an irreducible representation of the projected $\SL$-skein algebra.

\subsection{The ${\rm SL}_n$-character variety}
For any topological space $T$ such that $\pi_1(T)$ is a finitely generated group,
we use $\text{Hom}(\pi_1(T),{\rm SL}_n(\mathbb C))$ to denote the set of group homomorphisms from $\pi_1(T)$ to ${\rm SL}_n(\mathbb C)$ (one might want to choose a basepoint of $T$; this choice will not matter later).
We define an equivalence relation $\simeq$ on  $\text{Hom}(\pi_1(T),{\rm SL}_n(\mathbb C))$.
 Let $\chi,\chi'\in \text{Hom}(\pi_1(T),{\rm SL}_n(\mathbb C))$. Define $\chi\simeq\chi'$ if and only if $
 {\rm tr}(\chi(x))=
 {\rm tr}(\chi'(x))$
 for all $x\in\pi_1(T)$.
Define
\begin{align}\label{eq-Sl3-character}
\mathfrak{X}_{{\rm SL}_n(\mathbb{C})}(T) = \text{Hom}(\pi_1(T),{\rm SL}_n(\mathbb C))/\simeq.
\end{align}
Then $\mathfrak{X}_{{\rm SL}_n(\mathbb{C})}(T)$ is an algebraic set over $\mathbb C$  \cite{S2001SLn}. 
We use $\mathcal{O}_{{\rm SL}_n(\mathbb{C})}(T)$
to denote the its regular function ring.

\def\Osl{\mathcal{O}_{{\rm SL}_n(\mathbb{C})}(\fS)}
\def\Osd{\mathcal{O}_{{\rm SL}_n(\mathbb{C})}^{(d)}(\fS)}
\def\Xsl{\mathfrak{X}_{{\rm SL}_n(\mathbb{C})}(\fS)}
\def\Xsd{\mathfrak{X}_{{\rm SL}_n(\mathbb{C})}^{(d)}(\fS)}
\def\tra{{\rm tr}}


When $\bar\eta=\pm 1$, we have 
 \begin{equation}\label{eq-cross-com}
  	\raisebox{-.20in}{
  		
  		\begin{tikzpicture}
  			\tikzset{->-/.style=
  				
  				{decoration={markings,mark=at position #1 with
  						
  						{\arrow{latex}}},postaction={decorate}}}
  			\filldraw[draw=white,fill=gray!20] (-0,-0.2) rectangle (1, 1.2);
  			\draw [line width =1pt,decoration={markings, mark=at position 0.5 with {\arrow{>}}},postaction={decorate}](0.6,0.6)--(1,1);
  			\draw [line width =1pt,decoration={markings, mark=at position 0.5 with {\arrow{>}}},postaction={decorate}](0.6,0.4)--(1,0);
  			\draw[line width =1pt] (0,0)--(0.4,0.4);
  			\draw[line width =1pt] (0,1)--(0.4,0.6);
  			\draw[line width =1pt] (0.4,0.6)--(0.6,0.4);
  		\end{tikzpicture}
  	}
  	= 
  	\raisebox{-.20in}{
  		\begin{tikzpicture}
  			\tikzset{->-/.style=
  				
  				{decoration={markings,mark=at position #1 with
  						
  						{\arrow{latex}}},postaction={decorate}}}
  			\filldraw[draw=white,fill=gray!20] (-0,-0.2) rectangle (1, 1.2);
  			\draw [line width =1pt,decoration={markings, mark=at position 0.5 with {\arrow{>}}},postaction={decorate}](0.6,0.6)--(1,1);
  			\draw [line width =1pt,decoration={markings, mark=at position 0.5 with {\arrow{>}}},postaction={decorate}](0.6,0.4)--(1,0);
  			\draw[line width =1pt] (0,0)--(0.4,0.4);
  			\draw[line width =1pt] (0,1)--(0.4,0.6);
  			\draw[line width =1pt] (0.6,0.6)--(0.4,0.4);
  		\end{tikzpicture}
  	}\in\cS_{\bar\eta}(\fS).
  \end{equation}
Equation \eqref{eq-cross-com} shows that $\cS_{\bar\eta}(\fS)$ is a commutative algebra. 
Then we have the following.

\begin{lemma}[\cite{Sik05,wang2024TQFT}]\label{lem-bijection-character}
Suppose that $\fS$ is a punctured surface with $\partial\fS=\emptyset$,
and $\bar\eta = \pm 1$.
    There is an algebra isomorphism 
    \begin{align*}
        \CT\colon\cS_{\bar\eta}(\fS)\rightarrow \Osl.
    \end{align*}
\end{lemma}
\begin{proof}
    When $\bar\eta=1$, this isomorphism was constructed in \cite{Sik05}.
 \cite[Theorem 4.2]{wang2024TQFT} shows that there is an algebra isomorphism between $\cS_{1}(\fS)$ and $\cS_{-1}(\fS)$.
\end{proof}


We can identify $\Xsl$ with $\HA(\Osl,\mathbb C)$.
Then the algebra isomorphism $$\CT^{-1}\colon\Osl\rightarrow \cS_{\bar\eta}(\fS)$$ induces a bijection
\begin{align}\label{def-CT-star}
    \mathcal{T}_{*}\colon\text{Hom}_{\text{Alg}}(\cS_{\bar\eta}(\fS),\mathbb C)\rightarrow \mathfrak{X}_{{\rm SL}_n(\mathbb{C})}(\fS).
\end{align}
We have the following.

\def\SSS{\cS_{\bar\omega}^{*}(\fS)}
\def\SSE{\cS_{\bar\eta}(\fS)}

\begin{lemma}\label{lem-identity-eq}
    Let $\theta\in \text{Hom}_{\text{Alg}}(\cS_{\bar\eta}(\fS),\mathbb C)$. Suppose that 
    $\CT_{*}(\theta) = \chi$.
    Then 
    $$\CT(\alpha)(\chi) = \theta(\alpha)$$
    for any $\alpha\in\SSE$.
\end{lemma}
\begin{proof}
    We have 
    \begin{align*}
        \CT(\alpha)(\chi) =& \CT_{*}(\theta) (\CT(\alpha)) \quad(\because \mbox{$\CT_{*}(\theta) = \chi$})\\
        =& \theta (\mathcal T^{-1}(\CT(\alpha))) \quad(\because \mbox{the definition of $\mathcal T_{*}$})\\
        =& \theta(\alpha).
    \end{align*}
\end{proof}

\subsection{Construction of irreducible representations of $\SxsS$}\label{subsec-rep-skein}

In this subsection, we work under assumption \Rlabel{2} and further assume that $[n]_\omega! \neq 0$.

\def\End{{\rm End}_{\BC}}

\def\trace{{\rm tr}}

\begin{theorem}\cite[Theorems~4.2 and 5.8]{wang2024TQFT}
Let $\fS$ be a punctured surface such that each connected component of $\fS$ contains at least one puncture.
    Under assumption \Rlabel{2}, we have 
$\cS_{\bar\eta}^{*}(\fS)= \cS_{\bar\eta}(\fS)$.  
In particular, the Frobenius homomorphism $\Phi$ in Theorem~\ref{thm-Fro} is an algebra homomorphism from 
$\cS_{\bar\eta}(\fS)$ to $\cS_{\bar\omega}^{*}(\fS)$.
\end{theorem}

     Let $\fS$ be a punctured surface, and let $\mathcal P$ be the set of punctures of $\fS$. We require that each connected component of $\fS$ contains at least one puncture.
     Let $\rho\colon \SSS\rightarrow \End(V)$ be a finite dimensional irreducible representation of 
     $\cS_{\bar\omega}^{*}(\fS)$. 
From Theorem \ref{thm-Fro}(c), for any $\beta\in \SSE$, we have that
$\Phi(\beta)$ is a central element in $\SSS$.
Schur's lemma implies that
\begin{align}\label{deter-clas}
    \text{$\rho(\Phi(\beta)) = \theta(\beta)\Id_V,\text{ for some  $\theta(\beta)\in\mathbb C$.}$}
\end{align}
Then $\theta\in \hom(\SSE,\BC)$ because $\rho$ is an algebra homomorphism.

For $1\leq i\leq n-1$ and $p\in\mathcal P$, we have that $[\alpha(p)_i]_{\bar\omega}$ is a central element in $\SSS$.
    It follows from the Schur's lemma that
    \begin{align}\label{deter-loop}
        \text{$\rho([\alpha(p)_i]_{\bar\omega}) = s(p,i)\Id_V,\text{ for some  $s(p,i)\in\mathbb C$.}$}
    \end{align}

Set $\chi=\mathcal T_{*}(\theta)\in\Xsl$.
Lemma \ref{lem-identity-eq} shows that $\theta(\beta) = \mathcal T(\beta) (\chi)$ for $\beta\in\SSE$. Then,
for $1\leq i\leq n-1$ and $p\in\mathcal P$,
we have 
\begin{equation}\label{shadow-relation}
\begin{split}
    &\CT([\alpha(p)_i]_{\bar\eta}) (\chi)\Id_V=\theta([\alpha(p)_i]_{\bar\eta})\Id_V = \rho(\Phi([\alpha(p)_i]_{\bar\eta})) = 
    \rho\left(\bar P_{N,i}([\alpha(p)_1]_{\bar\omega},\cdots, [\alpha(p)_{n-1}]_{\bar\omega})\right)\\
    =& \bar P_{N,i}\left(\rho([\alpha(p)_1]_{\bar\omega}),\cdots, \rho([\alpha(p)_{n-1}]_{\bar\omega})\right)\Id_V
    =\bar P_{N,i}(s(p,1),\cdots, s(p,n-1))\Id_V.
\end{split}
\end{equation}

We now give the following definition.

\begin{definition}\label{def-character-pair}
Let $\bar\eta=\pm 1$.
 Let $\fS$ be a punctured surface, and let $\mathcal P$ be the set of punctures of $\fS$. We require that each connected component of $\fS$ contains at least one puncture.
    A {\bf character pair} of $\fS$ is pair $$(\chi,\{s(p,i)\mid p\in\mathcal P, 1\leq i\leq n-1\}),$$ where $\chi\in \Xsl$ and $\{s(p,i)\mid p\in\mathcal P, 1\leq i\leq n-1\}\subset \mathbb C$ such that
  $$\CT([\alpha(p)_i]_{\bar\eta}) (\chi)
  =\bar P_{N,i}(s(p,1),s(p,2),\cdots, s(p,n-1)) \quad (\text{see Lemma \ref{lem-bijection-character} for $\CT$})$$
  for $1\leq i\leq n-1$ and $p\in\mathcal P$.


\end{definition}

Equations~\eqref{deter-clas}, \eqref{deter-loop}, and \eqref{shadow-relation} imply the following.

\begin{proposition}\label{prop-classical-shadow}
     Suppose that we have the assumption  \Rlabel{1} and $[n]_\omega!\neq 0$.
     Let $\fS$ be a punctured surface such that each connected component of $\fS$ contains at least one puncture.
     Let $$\rho\colon \SSS\rightarrow \End(V)$$ be a finite dimensional irreducible representation of 
     $\cS_{\bar\omega}(\fS)$. 
     The $\rho$ determines a character pair
     $$(\chi,\{s(p,i)\mid p\in\mathcal P, 1\leq i\leq n-1\})$$ of $\fS$
     satisfying the following:
     \begin{enumerate}
         \item $\rho([\alpha(p)_i]_{\bar\omega}) = s(p,i)\Id_V$ for 
         $1\leq i\leq n-1$ and $p\in\mathcal P$.

         \item For any $\beta\in\cS_{\bar\eta}(\fS)$, we have 
         $$\rho(\Phi(\beta)) = \CT(\beta)(\chi) \Id_V,$$
         where $\Phi$ is the homomorphism in Theorem \ref{thm-Fro} and $\CT_{*}$ is defined in \eqref{def-CT-star}.
     \end{enumerate}
\end{proposition}

\begin{definition}
  As in Proposition \ref{prop-classical-shadow}, the character pair associated with an irreducible representation $\rho$ of $\SSS$ is called the {\bf classical shadow} of $\rho$.
\end{definition}

\def\End{{\rm End}_{\BC}}

A natural question to ask is: given a character pair, does there exist an irreducible representation whose classical shadow corresponds to this character? In the remainder of this section, we will focus on addressing this question.


Suppose that $\fS$ is a triangulable punctured surface with a triangulation $\lambda$.
Corollary \ref{cor-center-Z-two}(a) and Equation \eqref{lambda-one-equal-bl} implies that $\Zee$
is a commutative algebra under the assumption \Rlabel{2}.

Then the algebra homomorphism 
$\tr\colon \cS_{\bar\eta}(\fS)\rightarrow
    \mathcal{Z}_{\hat\eta}^{\rm bl}(\fS,\lambda)$ (see Theorem \ref{thm.quantum_trace}(a)) induces a map
    $${\rm tr}^*\colon
    \HA(\mathcal{Z}_{\hat\eta}^{\rm bl}(\fS,\lambda),\mathbb C)\rightarrow\HA(\cS_{\bar\eta}(\fS),\mathbb C).$$
We use $\overline{\mathcal{T}}$ to denote the composition of the following maps
$$\HA(\mathcal{Z}_{\hat\eta}^{\rm bl}(\fS,\lambda),\mathbb C)\xrightarrow{{\rm tr}^*}\HA(\cS_{\bar\eta}(\fS),\mathbb C)\xrightarrow{\mathcal{T}_{*}}\mathfrak{X}_{{\rm SL}_n(\mathbb{C})}(\fS).$$

The following theorem establishes that, for any character pair 
\[
(\chi,\{s(p,i)\mid p\in\mathcal P,\; 1\leq i\leq n-1\}) 
\quad \text{with } \chi \in \im\overline{\mathcal T},
\]
there exists an irreducible representation of $\SxsS$ whose classical shadow is precisely this character pair. 
This result constitutes the fourth main theorem of the paper.

\def\alle{[\alpha(p)_i]_{\bar\eta}}
\def\allo{[\alpha(p)_i]_{\bar\omega}}

\begin{theorem}\label{thm-main-3}
Suppose that we have the assumption \Rlabel{2} and $[n]_\omega!\neq 0$.
     Let $\fS$ be a triangulable punctured surface with a triangulation $\lambda$.
    Let $$(\chi,\{s(p,i)\mid p\in\mathcal P, 1\leq i\leq n-1\})$$ be a character pair of $\fS$ (Definition \ref{def-character-pair})
    with $\chi\in\im\overline{\CT}$, where $\mathcal P$ is the set of punctures of $\fS$.
    There exists an irreducible representation of $\Zoo$
    $$\bar\rho\colon\Zoo\rightarrow \End(V)$$
    such that the classical shadow of any irreducible sub-representation of   $$\rho\colon\SSS\xrightarrow{\tr}\Zoo\xrightarrow{\bar\rho}\End(V)$$
    is $(\chi,\{s(p,i): p\in\mathcal P, 1\leq i\leq n-1\})$.
\end{theorem}
\begin{proof}
Since $\chi \in \im \overline{\CT}$, there exists $f \in \HA(\Zee,\BC)$ such that $\overline{\CT}(f)=\chi$.
By Lemma~\ref{lem-identity-eq}, we have  
\begin{align}\label{eq-iden-T-tr-f}
    \CT(\alpha)(\chi) = f(\tr(\alpha))
\end{align}
for any $\alpha \in \SSE$.

For $1 \leq i \leq n-1$ and $p \in \mathcal P$, suppose that
\[
f(Z^{{\bf b}(\lambda,p,i)}) = b(p,i) \in \mathbb C^{*}.
\]
Define  
\begin{align}\label{eq-equality-bd}
    b(p,0)=b(p,n)=1, \qquad d(p,i) = \frac{b(p,i)}{b(p,i-1)}
\end{align}
for $1 \leq i \leq n$.  
Recall that ${\bf d}(\lambda,p,i)$ was defined in \eqref{eq-vectors-b-d} for $1 \leq i \leq n$ and $p \in \mathcal P$.  
For $1 \leq i \leq n-1$, we obtain
\begin{equation}\label{eq-cha-elementary}
    \begin{split}
        \mathcal T([\alpha(p)_i]_{\bar\eta})(\chi) 
        &= f(\tr([\alpha(p)_i]_{\bar\eta})) 
            \quad(\because\text{\eqref{eq-iden-T-tr-f}}) \\
        &= f\!\left( \sum_{J\in\mathbb J_i}\Big[ \prod_{j\in J} Z^{{\bf d}(\lambda,p,j)} \Big] \right) 
            \quad(\because\text{Lemma~\ref{lem-image-loop-trace}}) \\
        &= e_i(d(p,1), \dots, d(p,n)) 
            \quad(\because\text{\eqref{eq-vectors-b-d} and \eqref{eq-equality-bd}}),
    \end{split}
\end{equation}
where $e_i$ denotes the $i$-th elementary symmetric polynomial \eqref{elementary_symmetric_polynomial}.  

Since $(\chi,\{s(p,i) \mid p \in \mathcal P,\, 1 \leq i \leq n-1\})$ is a character pair, we also have  
\begin{align}\label{eq-cha-pair-PN}
    \mathcal T([\alpha(p)_i]_{\bar\eta})(\chi)=
        \bar P_{N,i}(s(p,1), \dots, s(p,n-1)).
\end{align}

Combining Lemma~\ref{lem-solu} with \eqref{eq-cha-elementary}–\eqref{eq-cha-pair-PN}, we deduce
\begin{align}\label{eq-s-e-t}
    (s(p,1), \dots, s(p,n-1)) 
    = \big(e_1(t(p,1), \dots, t(p,n)), \dots, e_{n-1}(t(p,1), \dots, t(p,n))\big),
\end{align}
where $t(p,i) \in \mathbb C^*$ satisfies
\begin{align}\label{eq-equality-td}
    t(p,i)^N = d(p,i), \qquad 1 \leq i \leq n.
\end{align}

For $1 \leq i \leq n$, define
\[
l(p,i):=\prod_{j=1}^i t(p,j).
\]
Then from \eqref{eq-equality-bd} and \eqref{eq-equality-td},  
\[
l(p,i)^N = b(p,i) = f(Z^{{\bf b}(\lambda,p,i)})
\]
for $1 \leq i \leq n-1$ and $p \in \mathcal P$.  

By Theorem~\ref{thm-representation-Fock}, there exists an irreducible representation  
\[
\bar\rho\colon \Zoo \to \End(V)
\]
such that
\begin{align}\label{fro-rep}
    \bar\rho(\PT(Z)) = f(Z)\,\Id_V \qquad (Z \in \Zee),
\end{align}
and
\[
\bar\rho(Z^{{\bf b}(\lambda,p,i)}) = l(p,i)\,\Id_V \qquad (1 \leq i \leq n-1,\; p \in \mathcal P).
\]

Hence, for $1 \leq i \leq n$ and $p \in \mathcal P$,  
\begin{align}\label{loop-rep}
    \bar\rho(Z^{{\bf d}(\lambda,p,i)}) 
    = \bar\rho(Z^{{\bf b}(\lambda,p,i)})\,
      \bar\rho(Z^{{\bf b}(\lambda,p,i-1)})^{-1}
    = l(p,i) l(p,i-1)^{-1}
    = t(p,i),
\end{align}
where ${\bf b}(\lambda,p,0)={\bf b}(\lambda,p,n)={\bf 0}$ and the first equality follows from \eqref{eq-vectors-b-d}.  

Now set $\rho=\bar\rho\circ\tr$.  
For any $\alpha \in \SSE$, we have
\begin{equation}\label{shadow-fro}
    \begin{split}
        \rho (\Phi(\alpha)) 
        &= \bar\rho (\PT(\tr(\alpha)))           \quad(\because\text{$\rho=\bar\rho\circ\tr$ and Theorem~\ref{thm-Fro}(b)}) \\
        &= f(\tr(\alpha))\,\Id_V 
            \quad(\because\text{\eqref{fro-rep}}) \\
        &= \mathcal T(\alpha)(\chi)\,\Id_V 
            \quad(\because\text{\eqref{eq-iden-T-tr-f}}).
    \end{split}
\end{equation}

For $1 \leq i \leq n-1$ and $p \in \mathcal P$,  
\begin{equation}\label{shadow-loop}
    \begin{split}
        \rho (\alle) 
        &= \bar\rho \!\left( \sum_{J\in\mathbb J_i} \Big[ \prod_{j\in J} Z^{{\bf d}(\lambda,p,j)} \Big] \right)           \quad(\because\text{$\rho=\bar\rho\circ\tr$ and Lemma~\ref{lem-image-loop-trace}})\\
        &= \left( \sum_{J\in\mathbb J_i} \prod_{j\in J} t(p,j) \right)\Id_V 
            \quad(\because\text{\eqref{loop-rep}}) \\
        &= s(p,i)\,\Id_V 
            \quad(\because\text{\eqref{eq-s-e-t}}).
    \end{split}
\end{equation}

Equations \eqref{shadow-fro} and \eqref{shadow-loop} show that
the classical shadow of any irreducible subrepresentation of $\rho$ is precisely 
\[
(\chi,\{s(p,i) \mid p \in \mathcal P,\, 1 \leq i \leq n-1\}).
\]
\end{proof}

\begin{remark}
 In the case when $n=2$, Theorem \ref{thm-main-3}
 was proved in \cite{representation2}.
\end{remark}

\section{The naturality of the constructed representations of the ${\rm SL}_n$-skein algebra}\label{sec-naturality}

Let $\fS$ be a triangulable punctured surface with two triangulations $\lambda$ and $\lambda'$.  
It was shown in \cite{KimWang} that $\tr$ and ${\rm tr}_{\lambda'}$ are related to each other by a sequence of balanced $n$-th root version of quantum coordinate change isomorphisms.  
In this section, we use these isomorphisms to prove that, under mild conditions, the representation constructed in Theorem~\ref{thm-main-3} with respect to $\lambda$ is isomorphic to the one constructed with respect to $\lambda'$ (Theorem~\ref{thm-naturality}).

In \S\ref{subsec-mutation} and \S\ref{subsec-coordinate-change}, we recall the quantum coordinate change isomorphism from \cite{KimWang} and present several related results.  
In \S\ref{sub-sec-naturality}, we state and prove Theorem~\ref{thm-naturality}.

\subsection{Classical and quantum cluster $\mathcal X$-mutations}\label{subsec-mutation}

\def\Fr{\text{Frac}}

Fix a non-empty set $\VV$, and a non-empty subset $\VV_{\rm mut}$ of $\VV$. The elements of $\VV$, $\VV_{\rm mut}$ and $\VV\setminus \VV_{\rm mut}$ are called {\bf vertices}, {\bf mutable vertices} and {\bf frozen vertices}, respectively. By a quiver $\Gamma$ we mean a directed graph whose set of vertices is $\VV$ 
and equipped with weight on the edges so that the weight on each edge is $1$ unless the edge connects two frozen vertices, in which case the weight is $\frac{1}{2}$. An edge of weight $1$ will be called an {\bf arrow}, and an edge of weight $\frac{1}{2}$ a {\bf half-arrow}. As before, we denote by $Q = (Q(u,v))_{u,v\in \VV}$ to denote the signed adjacency matrix of $\Gamma$. A quiver is considered up to equivalence, where two quivers are equivalent if they yield the same signed adjacency matrix. In particular, a quiver can be represented by a representative quiver without an oriented cycle of length 1 or 2. Let $\mathcal{F}$ be the field of rational functions over $\mathbb{Q}$ with the set of algebraically independent generators enumerated by $\VV$, say $\mathcal{F} = \mathbb{Q}(\{y_v\}_{v\in \VV})$. A {\bf cluster $\mathcal{X}$-seed} is a pair $\mathcal{D} = (\Gamma,(X_v)_{v\in \VV})$, where $\Gamma$ is a quiver and $\{X_v\}_{v\in \VV}$ forms an algebraically independent generating set of $\mathcal{F}$ over $\mathbb{Q}$. The signed adjacency matrix $Q$ of $\Gamma$ is called the {\bf exchange matrix} of the seed $\mathcal{D}$.

\def\sgn{\text{sgn}}

Suppose that $k\in\VV_{
\rm mut}$. The {\bf mutation}
$\mu_k$ at the 
mutable vertex $k\in\VV_{\rm mut}$ is defined to be the process that transforms a seed $\mathcal{D} = (\Gamma,(X_v)_{v\in \VV})$ into a new seed
$\mu_k(\mathcal D) = \mathcal D' = 
(\Gamma', (
X_v')_{v\in 
\VV})
$.
Here 
\begin{align*}
    X_v'= \begin{cases}
    X_v^{-1} & v=k,\\
    X_v (1+ 
    X_k^{-\text{sgn}(Q(v,k)) })^{-Q(v,k)} & v\neq k,
\end{cases}
\end{align*}
where 
$$
\sgn(a)=
\begin{cases}
    1 & a>0,\\
    0 & a=0,\\
    -1 & a<0,
\end{cases}
$$
for $a\in\mathbb R$, and $\Gamma'$ is obtained from 
$\Gamma$ by the following procedures: (1)
 reverse all the arrows incident to the vertex $k$,
(2) for each pair of arrows $k\rightarrow i$ and $j\rightarrow k$ draw an arrow $i\rightarrow j$,
(3) delete pairs of arrows $i\rightarrow j$ and $j\rightarrow i$ going in the opposite directions (more precisely, tidy up the quiver while maintaining the signed adjacency matrix).
We use $Q' = (Q'(u,v))_{u,v\in\VV}$ to denote the signed adjacency matrix of $\Gamma'$. Then we have 
\begin{align*}
Q'(u,v) = \begin{cases}
    - Q(u,v) & k\in\{u,v\},\\
    Q(u,v) +\frac{1}{2}(Q(u,k)|Q(k,v)|+
    |Q(u,k)|Q(k,v)) & k\notin\{u,v\}.
\end{cases}
\end{align*}

It is well-known that $\mu_k(\mu_k(\mathcal D)) = \mathcal D$.

Suppose that 
$\mathcal{D} = (\Gamma,(X_v)_{v\in \VV})$ is an $\mathcal X$-seed 
whose 
exchange matrix 
is $Q$.
Define the {\bf Fock-Goncharov algebra} associated to $\mathcal D$ to be the quantum torus algebra
\begin{align*}
\mathcal{X}_\omega(\mathcal{D})  = \bT_\omega(Q)  = \mathbb C\langle X_v^{\pm 1}, v\in \VV \rangle / (X_v X_{v'} = \omega^{2Q(v,v')} X_{v'} X_v \mbox{ for } v,v' \in \VV).
\end{align*}

A useful notion is the Weyl-ordered product, which is a certain normalization of a Laurent monomial defined as follows: for any $v_1,\ldots,v_r \in \VV$ and $a_1,\ldots,a_r \in \mathbb{Z}$,
\begin{align}
\label{Weyl_ordering}
\left[ X_{v_1}^{a_1} X_{v_2}^{a_2} \cdots X_{v_r}^{a_r} \right] := \omega^{-\sum_{i<j} Q(v_i,v_j)} X_{v_1}^{a_1} X_{v_2}^{a_2} \cdots X_{v_r}^{a_r}
\end{align}
In particular, for ${\bf t} = (t_v)_{v\in \VV} \in \mathbb{Z}^\VV$, the following notation for the corresponding Weyl-ordered Laurent monomial will become handy:
$$
X^{\bf t} := \left[ \prod_{v\in \VV} X_v^{t_v}\right].
$$

The following special function is a crucial ingredient.
\begin{definition}[compact quantum dilogarithm \cite{F95,FK94}]
    The quantum dilogarithm for a quantum parameter
$\omega$ is the function
$$\Psi^\omega(x) = \prod_{r=1}^{+\infty} (1+ \omega^{2r
-1} x)^{-1}.$$
\end{definition}
For our purposes, the infinite product can be understood formally, as what matters for us is only the following:

\def\sgn{{\rm sgn}}

\begin{lemma}\label{lem-F-P}
    Suppose that $xy = \omega^{2m} yx$ for some $m\in \mathbb{Z}$. Then we have 
    $${\rm Ad}_{\Psi^\omega(x)}(y) = y F^\omega(x,m),$$
    where
$$F^\omega(x,m)=\prod_{r=1}^{|m|} (1+ \omega^{(2r-1){\rm sgn}(m)}x)^{\sgn(m)}.$$
\end{lemma}

\begin{definition}[quantum $\mathcal X$-mutation for Fock-Goncharov algebras \cite{BZ,FG09a,FG09b}]\label{def.quantum_X-mutation}
     Suppose 
     that $\mathcal{D} = (\Gamma,(X_v)_{v\in \VV})$ is an $\mathcal X$-seed, 
     $k$$\in \VV_{\rm mut}$ is a 
     mutable vertex of $\Gamma$, and $\mathcal D'=\mu_k(\mathcal D)$.
      The quantum mutation map 
      is the 
      isomorphism between the 
      skew-fields
      $$\mu_{
      \mathcal{D},\mathcal{D}'}^\omega=\mu_k^\omega \, \colon \, \Fr(
      \mathcal{X}_\omega(\mathcal D'))\rightarrow
      \Fr(
      \mathcal{X}_\omega(\mathcal D))$$
      given by the composition
      $$\mu_k^\omega= \mu_k^{\sharp \omega}\circ \mu_k',$$
      where $\mu'_k$ is the isomorphism of skew-fields
      $$\mu_k'\colon \Fr(
      \mathcal{X}_\omega(\mathcal D'))\rightarrow
      \Fr(
      \mathcal{X}_\omega(\mathcal D))$$ 
      given by 
      \begin{align}\label{eq-quantum-mutation}
      \mu_k'(
      X_v') = 
      \begin{cases}
          X_k^{-1}& \mbox{if } v=k,\\
          \left[ X_v X_k^{[Q(v,k)]_+}\right]
          & \mbox{if } v\neq k.
      \end{cases}
      \end{align}
      where $[a]_+$ stands for the positive part of a real number $a$
      $$
      [a]_+ := \max\{a,0\} = {\textstyle \frac{1}{2}(a+|a|)},
$$
and $[\sim]$ is the Weyl-ordered product as in \eqref{Weyl_ordering},
      while
       $\mu^{\sharp \omega}_k$ is the following skew-field automorphism
       $$\mu_k^{\sharp \omega} = {\rm Ad}_{\Psi^\omega(
       X_k)}\colon \Fr(
       \mathcal{X}_\omega(\mathcal D))\rightarrow
      \Fr(
      \mathcal{X}_\omega(\mathcal D)).$$
\end{definition}

\subsection{The coordinate change isomorphism for the balanced part}\label{subsec-coordinate-change}
Suppose 
that $\mathcal{D} = (\Gamma,(X_v)_{v\in \VV})$ is an $\mathcal X$-seed 
whose exchange matrix is $Q$.
Define the {\bf $n$-th root Fock-Goncharov algebra} associated to $\mathcal D$ to be
\begin{align*}
\mathcal{Z}_{\hat{\omega}}(\mathcal{D})  = \mathbb{T}_{\hat{\omega}}(Q) = \BC\langle Z_v^{\pm 1}, v\in \VV \rangle / (Z_v Z_{v'} = {\hat{\omega}}^{2Q(v,v')} Z_{v'} Z_v \mbox{ for } v,v' \in \VV),
\end{align*}
into which the Fock-Goncharov algebra $\mathcal{X}_q(\mathcal{D})$ embeds as
$$
X_v \mapsto Z_v^n, \quad \forall v\in \VV.
$$

We use $\Fr(
\mathcal{Z}_{\hat{\omega}}(\mathcal{D}))$ to denote the 
skew-field of fractions of $
\mathcal{Z}_{\hat{\omega}}(\mathcal D)$. 
Naturally, ${\rm Frac}(\mathcal{X}_q(\mathcal{D}))$ embeds into ${\rm Frac}(\mathcal{Z}_{\hat{\omega}}(\mathcal{D}))$. The Weyl-ordered product can be defined similarly as before: for any $v_1,\ldots,v_r \in \VV$ and $a_1,\ldots,a_r \in \mathbb{Z}$,
\begin{align*}
\left[ Z_{v_1}^{a_1} \cdots Z_{v_r}^{a_r} \right] := {\hat{\omega}}^{-\sum_{i<j} Q(v_i,v_j)} Z_{v_1}^{a_1} \cdots Z_{v_r}^{a_r}.
\end{align*}
For ${\bf t} =(t_v)_{v\in \VV} \in \mathbb{Z}^\VV$ one defines the Weyl-ordered Laurent monomial:
\begin{align*}
Z^{\bf t} := \left[ \prod_{v\in \VV}Z_v^{t_v} \right],
\end{align*}

Define the 
     skew-field isomorphism $$\nu_k'\colon \Fr(
     \mathcal{Z}_{\hat{\omega}}(\mathcal D'))\rightarrow
      \Fr(
      \mathcal{Z}_{\hat{\omega}}(\mathcal D))$$
      such that the action of $\nu_k'$ on the generators is given by 
      the formula \eqref{eq-quantum-mutation} for $\mu'_k$ with $
      X_{v}$ replaced by $
      Z_{v}$:
     \begin{align}\label{eq-quantum-mutation_Z}
      \nu_k'(Z_v') = 
      \begin{cases}
          Z_k^{-1}& \mbox{if } v=k,\\
          \left[ Z_v Z_k^{[Q(v,k)]_+}\right]
          & \mbox{if } v\neq k.
      \end{cases}
      \end{align}
      It is easy to check that $\nu'_k$ extends $\mu'_k$. 

     For an $n$-root extension of the automorphism part $\mu^{\sharp q}_k$, we consider the following subalgebra of $\mathcal{Z}_{\hat{\omega}}(\mathcal{D})$.
   \begin{definition}\label{def.mbl}
   Define  
       $$\mathcal B_{\mathcal D}=\{{\bf t}=(
       t_v)_{v\in\VV}\in\mathbb Z^{
       \VV}\mid \sum_{v\in 
       \VV} Q(u,v)
       t_v=0\in\mathbb Z_n=
       \mathbb{Z}/n \mathbb{Z}\text{ for all }u\in\VV_{
       {\rm mut}}\},$$
    and define 
    $\mathcal{Z}^{\rm mbl}_{\hat{\omega}}(\mathcal{D})$
    to be the $\BC$-submodule of $
    \mathcal{Z}_{\hat{\omega}}(\mathcal D)$ spanned by $Z^{{\bf t}}$ for ${\bf t}\in \mathcal B_{\mathcal D}$. We call $\mathcal{Z}^{\rm mbl}_{\hat{\omega}}(\mathcal{D})$ the {\bf mutable-balanced subalgebra} of $\mathcal{Z}_{\hat{\omega}}(\mathcal{D})$.
   \end{definition}

Note that $\mathcal B_{\mathcal D}$ is a subgroup of $\mathbb Z^{
\VV}$. So 
$\mathcal{Z}^{\rm mbl}_{\hat{\omega}}(\mathcal{D})$
    is 
    indeed a subalgebra of $
    \mathcal{Z}_{\hat{\omega}}(\mathcal D)$.
    Since $n \mathbb Z^{
    \VV}\subset\mathcal B_{\mathcal D}\subset \mathbb Z^{
    \VV}$, we have 
    $$\mathcal{X}_{\omega}(\mathcal{D}) \subset \mathcal{Z}^{\rm mbl}_{\hat{\omega}}(\mathcal{D}) \subset \mathcal{Z}_{\hat{\omega}}(\mathcal{D}).$$

By \cite[Lemma 3.7]{KimWang},
there exists a unique skew-field isomorphism
    $$ \nu_k^{\sharp {\hat{\omega}}}:={\rm Ad}_{\Psi^q(
    X_k)}\colon \Fr(
    \mathcal{Z}^{\rm mbl}_{\hat{\omega}}(\mathcal D))\rightarrow
      \Fr(
      \mathcal{Z}^{\rm mbl}_{\hat{\omega}}(\mathcal D))$$
      such that for each ${\bf t} = (t_v)_{v\in \VV} \in \mathcal{B}_\mathcal{D}$, we have
      $$
      \nu^{\sharp {\hat{\omega}}}_k(Z^{\bf t}) = Z^{\bf t} \, F^\omega(X_k, m),
$$
where $m = \frac{1}{n} \sum_{v\in \VV} Q(k,v) t_v$.

By \cite[Lemma 3.8]{KimWang},
    the map $\Fr(
    \mathcal{Z}_{\hat{\omega}}(\mathcal D'))\xrightarrow{\nu_k'}
      \Fr(
      \mathcal{Z}_{\hat{\omega}}(\mathcal D))$ defined in \eqref{eq-quantum-mutation_Z} restricts to a  
      skew-field isomorphism
      $\Fr(
      \mathcal{Z}^{\rm mbl}_{\hat{\omega}}(\mathcal D'))\xrightarrow{\nu_k'}
      \Fr(
      \mathcal{Z}^{\rm mbl}_{\hat{\omega}}(\mathcal D)).$ \qed

We define the quantum mutation $\nu_k^{{\hat{\omega}}}$ for the $n$-th root Fock-Goncharov algebras, or more precisely for the skew-fields of fractions of their mutable-balanced subalgebras,  
to be the composition 
\begin{align}\label{eq-def-nu}
    \nu^{\hat{\omega}}_k := \nu^{\sharp{\hat{\omega}}}_k \circ \nu'_k ~:~ 
\Fr(
\mathcal{Z}^{\rm mbl}_{\hat{\omega}}(\mathcal D'))\xrightarrow{\nu_k'}
      \Fr(
      \mathcal{Z}^{\rm mbl}_{\hat{\omega}}(\mathcal D))\xrightarrow{\nu_k^{\sharp {\hat{\omega}}}}
      \Fr(
      \mathcal{Z}^{\rm mbl}_{\hat{\omega}}(\mathcal D)).
\end{align}

\begin{lemma}\cite[Lemma 3.9]{KimWang}\label{lem:nu_extends_mu}
    The map $\nu^{\hat{\omega}}_k$ extends $\mu^\omega_k$. \qed
\end{lemma}

 In order to apply to our situation, consider a triangulable pb surface $\frak{S}$ with a triangulation $\lambda$.
    Letting $\VV$ be $V_\lambda$ and $\VV_{\rm mut}$ be the subset of $V_\lambda$ consisting of the vertices contained in the interior of $\frak{S}$, one obtains a cluster $\mathcal{X}$-seed $\mathcal{D}_\lambda := (\Gamma_\lambda,(X_v)_{v\in V_\lambda})$.
    We denote the $n$-root Fock-Goncharov algebra and its mutable-balanced subalgebra as
    $$
    \mathcal{Z}_{\hat{\omega}}(\fS,\lambda) = \mathcal{Z}_{\hat{\omega}}(\mathcal{D}_\lambda), \qquad
    \mathcal{Z}_{\hat{\omega}}^{\rm mbl}(\fS,\lambda) = \mathcal{Z}^{\rm mbl}_{\hat{\omega}}(\mathcal{D}_\lambda).
$$
Recall the balanced subalgebra $\mathcal{Z}^{\rm bl}_{\hat{\omega}}(\fS,\lambda)$ of $\mathcal{Z}_{\hat{\omega}}(\frak{S}, \lambda)$ defined in \eqref{def-balanced}.
It follows from \cite[Lemma 3.11]{KimWang} that
$\mathcal{Z}^{\rm bl}_{\hat{\omega}}(\fS,\lambda) \subset \mathcal{Z}^{\rm mbl}_{\hat{\omega}}(\fS,\lambda)$.

Let $\fS$ be a triangulable pb surface with a triangulation $\lambda$. In the last subsection we considered a particular cluster $\mathcal{X}$-seed $\mathcal{D}_\lambda = (\Gamma_\lambda, (X_v)_{v\in V_\lambda})$, associated to a special quiver $\Gamma_\lambda$, with certain vertex sets $\VV = V_\lambda$ and $\VV_{\rm mut}$.

Suppose $\fS$ is a triangulable punctured surface with two triangulations $\lambda$ and $\lambda'$, where $\lambda'$ is obtaine from $\lambda$ by a flip (see Figure \ref{flip}).
There is a special sequence of $\mathcal X$-mutations $\mu_{v_r},\cdots,\mu_{v_2},\mu_{v_1}$ (\cite{FG06,GS19}, see \cite[Figure 3]{KimWang})  such that
\begin{align}
\label{two_seeds_connected_by_sequence_of_mutations}
\mathcal{D}_{\lambda'} = \mu_{v_r} \cdots \mu_{v_2} \mu_{v_1} (\mathcal{D}_\lambda)    
\end{align}
and in particular,
$$
\Gamma_{\lambda'} = \mu_{v_r} \cdots \mu_{v_2} \mu_{v_1} (\Gamma_\lambda).
$$
Define 
\begin{equation}\label{eq-Theta1}
\Phi_{
\lambda \lambda'}^{\omega}:=\mu_{v_1}^{\omega}\circ\cdots\circ\mu_{v_r}^{\omega}\colon\Fr(\mathcal X_{\omega}(\fS,\lambda'))\rightarrow 
\Fr(\mathcal X_{\omega}(\fS,\lambda)).
\end{equation}
\begin{equation}\label{eq-Theta2}
\Theta_{
\lambda \lambda'}^{{\hat{\omega}}}:=\nu_{v_1}^{{\hat{\omega}}}\circ\cdots\circ\nu_{v_r}^{{\hat{\omega}}}\colon\Fr(\mathcal Z_{\hat{\omega}}^{\rm mbl}(\fS,\lambda'))\rightarrow 
\Fr(\mathcal Z_{\hat{\omega}}^{\rm mbl}(\fS,\lambda)).
\end{equation}

\begin{figure}
	\centering
	\includegraphics[width=8cm]{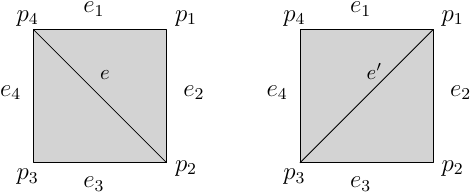}
	\caption{This is a local figure for doing a flip on the edge $e$.}\label{flip}
\end{figure}

\def\wl{\widetilde\lambda}

Suppose $\lambda$ and $\lambda'$ are two triangulations of $\fS$. A {\bf triangulation sweep} connecting 
$\lambda$ and $\lambda'$ is a sequence of triangulations
$\Lambda=(\lambda_1,\cdots,\lambda_m)$ such that 
$\lambda_1=\lambda$, $\lambda_m=\lambda'$, and $\lambda_{i+1}$ is obtained from $\lambda_i$ by a flip for each $1\leq i\leq m-1$.
Define 
\begin{equation}\label{eq-Theta-change1}
\Phi_\Lambda^{\omega} :=\Phi_{
\lambda_1\lambda_2}^{\omega}\circ\cdots\circ\Phi_{
\lambda_{m-1}\lambda_m}^{\omega}\colon\Fr(
\mathcal{X}_{\omega}(\fS,\lambda'))\rightarrow 
\Fr(
\mathcal{X}_{\omega}(\fS,\lambda)),
\end{equation}
\begin{equation}\label{eq-Theta-change2}
\Theta_{\Lambda}^{{\hat{\omega}}}  :=\Theta_{
\lambda_1\lambda_2}^{{\hat{\omega}}}\circ\cdots\circ\Theta_{
\lambda_{m-1}\lambda_m}^{{\hat{\omega}}}\colon\Fr(
\mathcal{Z}^{\rm mbl}_{\hat{\omega}}(\fS,\lambda'))\rightarrow 
\Fr(
\mathcal{Z}^{\rm mbl}_{\hat{\omega}}(\fS,\lambda)).
\end{equation}

\begin{proposition}\cite[Propositions 3.13 and 3.14]{KimWang}
\label{division_homomorphisms_for_sweep_independent}
    The division homomorphisms in Equations \eqref{eq-Theta-change1}-\eqref{eq-Theta-change2} are independent of the choice of $\Lambda$.
\end{proposition}

By Proposition~\ref{division_homomorphisms_for_sweep_independent}, we write 
$$\Phi_{\lambda\lambda'}:=\Phi_{\Lambda}^{{\hat{\omega}}} \text{ and }
\Theta_{\lambda\lambda'}:=\Theta_{\Lambda}^{{\hat{\omega}}}.$$

\begin{theorem}[\cite{KimWang}]\label{thm-naturality-trace}
    Suppose that $\fS$ is a triangulable pb surface with two triangulations $\lambda$ and $\lambda'$. We have 
    \begin{enumerate}[label={\rm (\alph*)}]
        \item the division isomorphism $\Theta_{\lambda\lambda'}$ restricts to a division isomorphism
        \begin{align}\label{eq-coordinate-iso}\Theta_{\lambda\lambda'}\colon\Fr(\mathcal Z_{\hat \omega}^{\rm bl}(\fS,\lambda'))\rightarrow 
\Fr(\mathcal Z_{\hat \omega}^{\rm bl}(\fS,\lambda))
        \end{align}
        \item let $\lambda''$ be another triangulation of $\fS$, we have
        \begin{align}\label{eq-com-lambda}         \Theta_{\lambda\lambda'}\circ\Theta_{\lambda'\lambda''}=\Theta_{\lambda\lambda''},
        \end{align}
        \item the division isomorphism in \eqref{eq-coordinate-iso} is the extension of $\Phi_{\lambda\lambda'}$,
        \item the following diagram commutes
        \begin{equation*}
\begin{tikzcd}
\overline{\cS}_{\hat \omega}^{\rm st}(\fS) \arrow[r,"{\rm tr}_{\lambda'}"]
\arrow[d, "\tr"]  
&  \Fr(\mathcal Z_{\hat \omega}^{\rm bl}(\fS,\lambda'))  \arrow[dl, "\Theta_{\lambda\lambda'}"] \\
 \Fr(\mathcal Z_{\hat \omega}^{\rm bl}(\fS,\lambda))
&  \\
\end{tikzcd}
\end{equation*}
    \end{enumerate}
\end{theorem}

\begin{remark}
    $\Theta_{\lambda\lambda'}$ was also established in \cite{LY23} using a different technique.
\end{remark}

The following proposition will later be used to prove the main result of this section, Theorem~\ref{thm-naturality}.

\begin{proposition}\label{prop-loop-nature}
    Let $\fS$ be a triangulable pb surface with two triangulations $\lambda$ and $\lambda'$.
    For each puncture $p$ of $\fS$ and $1\leq i\leq n-1$, we have 
    $$\Theta_{\lambda\lambda'}(Z^{{\bf b}(\lambda',p,i)}) = Z^{{\bf b}(\lambda,p,i)},$$
    where ${\bf b}(\lambda,p,i)$ and ${\bf b}(\lambda',p,i)$ are defined as in \eqref{def-a-l-p-i}.
\end{proposition}
\begin{proof}
We may assume that $\lambda'$ is obtained from $\lambda$ by a flip on edge $e$, by equation \eqref{eq-com-lambda}.
Suppose the local picture for this flip is as illustrated in Figure \ref{flip}, where the left picture corresponds to $\lambda$ and the right picture to $\lambda'$.

Let $\fS^{\sharp}$ denote the surface obtained from $\fS$ by cutting along the edges $e_i$ for $1\leq i\leq 4$. Then we have $\fS^{\sharp} = \fS^{*}\bigsqcup\mathbb P_4$. Let {\bf pr} denote the projection from
$\fS^{\sharp}$ to $\fS$. If $({\bf pr})^{-1}(p)\cap \{p_1,p_2,p_3,p_4\}=\emptyset$, the statement of the proposition is trivial. Otherwise, suppose that $({\bf pr})^{-1}(p)\cap \{p_1,p_2,p_3,p_4\}\neq\emptyset$. We only treat the case $({\bf pr})^{-1}(p)\cap \{p_1,p_2,p_3,p_4\}=\{p_4\}$, since the same techniques apply to the other cases.

Let $\lambda_4$ (resp. $\lambda_4'$) denote the triangulation of $\mathbb P_4$ shown in the left (resp. right) picture of Figure \ref{flip}.
The triangulations $\lambda$ and $\lambda'$ of $\fS$ induce the same triangulation of $\fS^{*}$, denoted by $\lambda^{*}$.
Define algebra embeddings
\begin{align*}
        \mathcal S_E&:=\mathcal S_{e_1}\circ \mathcal S_{e_2}\circ \mathcal S_{e_3}\circ \mathcal S_{e_4}\colon 
        \mathcal Z_{\hat \omega}^{\rm bl}(\fS,\lambda)\rightarrow
        \mathcal Z_{\hat \omega}^{\rm bl}(\fS^{*},\lambda^{*})\otimes_{\mathbb C}
        \mathcal Z_{\hat \omega}^{\rm bl}(\mathbb P_4,\lambda_4)\\
         \mathcal S_E'&:=\mathcal S_{e_1}'\circ \mathcal S_{e_2}'\circ \mathcal S_{e_3}'\circ \mathcal S_{e_4}'\colon 
        \mathcal Z_{\hat \omega}^{\rm bl}(\fS,\lambda')\rightarrow
        \mathcal Z_{\hat \omega}^{\rm bl}(\fS^{*},\lambda^{*})\otimes_{\mathbb C}
        \mathcal Z_{\hat q}^{\rm bl}(\mathbb P_4,\lambda_4'),
\end{align*}
where $\mathcal S_{e_i}$ and $\mathcal S_{e_i}'$ are defined in \eqref{cutting_homomorphism_for_Z_omega}.
Then $\mathcal S_{E}$ and $\mathcal S_{E}'$ induce the following division embeddings:
\begin{align*}
        \mathcal S_E&\colon 
        \Fr(\mathcal Z_{\hat \omega}^{\rm bl}(\fS,\lambda))\rightarrow
        \Fr(\mathcal Z_{\hat \omega}^{\rm bl}(\fS^{*},\lambda^{*})\otimes_{\mathbb C}
        \mathcal Z_{\hat \omega}^{\rm bl}(\mathbb P_4,\lambda_4))\\
         \mathcal S_E'&\colon 
        \Fr(\mathcal Z_{\hat \omega}^{\rm bl}(\fS,\lambda'))\rightarrow
        \Fr(\mathcal Z_{\hat \omega}^{\rm bl}(\fS^{*},\lambda^{*})\otimes_{\mathbb C}
        \mathcal Z_{\hat \omega}^{\rm bl}(\mathbb P_4,\lambda_4')).
\end{align*}

By \cite[Lemma 3.18]{KimWang}, the following diagram commutes:
\begin{equation}\label{eq-commmm-key}
\begin{tikzcd}
\Fr(\mathcal Z_{\hat \omega}^{\rm bl}(\fS,\lambda')) \arrow[r, "\mathcal S_{E}'"]
\arrow[d, "\Theta_{
\lambda \lambda'}"]  
& \Fr(\mathcal Z_{\hat \omega}^{\rm bl}(\fS^{*},\lambda^{*})\otimes_R
        \mathcal Z_{\hat \omega}^{\rm bl}(\mathbb P_4,\lambda_4')) \arrow[d, "\Theta_{
(\lambda^{*}\sqcup\lambda_4) (\lambda^{*}\sqcup\lambda_4')}"] \\
 \Fr(\mathcal Z_{\hat \omega}^{\rm bl}(\fS,\lambda))
 \arrow[r, "\mathcal S_E"] 
&  \Fr(\mathcal Z_{\hat \omega}^{\rm bl}(\fS^{*},\lambda^{*})\otimes_R
        \mathcal Z_{\hat \omega}^{\rm bl}(\mathbb P_4,\lambda_4))
\end{tikzcd}
\end{equation}

We label the two ideal triangles in the left (resp. right) picture of Figure \ref{flip} as $\tau_1,\tau_2$ (resp. $\tau_1',\tau_2'$).
Assume that $\tau_2'$ contains $p_4$.
Let $f_e$ (resp. $f_{e'}$) denote the group embedding induced by cutting $\mathbb P_4$ along $e$ (resp. $e'$)
$$\mathcal B_{\lambda_4}\rightarrow
  \mathcal B_{\tau_1}\oplus \mathcal B_{\tau_2}\; (\text{resp. } 
  \mathcal B_{\lambda_4'}\rightarrow
  \mathcal B_{\tau_1'}\oplus \mathcal B_{\tau_2'}).$$
It is straightforward to verify that 
$$({\bf b}(\tau_1,p_4,i), {\bf b}(\tau_2,p_4,i))\in\im f_{e}
    \quad\text{and}\quad
    ({\bf 0}, {\bf b}(\tau_2',p_4,i))\in\im f_{e'},$$
for $1\leq i\leq n-1$.
Hence, for $1\leq i\leq n-1$, there exist 
${\bf b}(p_4,i)\in\mathcal B_{\lambda_4}$ and 
${\bf b}'(p_4,i)\in\mathcal B_{\lambda_4'}$ such that 
\begin{align*}
        f_e({\bf b}(p_4,i))&=
        ({\bf b}(\tau_1,p_4,i), {\bf b}(\tau_2,p_4,i)),\\
        f_{e'}({\bf b}'(p_4,i))&=
        ({\bf 0}, {\bf b}(\tau_2',p_4,i)).
\end{align*}

Let $C$ be the corner arc in $\mathbb P_4$ encircling $p_4$, oriented counterclockwise around $p_4$. For each $t \in \{1,2,\dots,n\}$, denote by $C_t$ the stated arc in $\mathbb P_4$ obtained by assigning state $t$ to both endpoints of $C$.
It follows from Lemma \ref{lem-Z-tr}  and \eqref{eq-ZKvb}
that
\begin{align}\label{eq-Zb-tC}
 Z^{{\bf b}(p_4,i)}=\prod_{t=n-i+1}^{n} {\rm tr}_{\lambda_4}(C_t) \text{ and }
  Z^{{\bf b}'(p_4,i)}=\prod_{t=n-i+1}^{n} {\rm tr}_{\lambda_4'}(C_t).
\end{align}

We then compute:
\begin{align*}
    &\Theta_{
(\lambda^{*}\sqcup \lambda_4) (\lambda^{*}\sqcup\lambda_4')}(\mathcal S_E'(Z^{{\bf b}(\lambda',p,i)}))\\
    =&
    \Theta_{
(\lambda^{*}\sqcup \lambda_4) (\lambda^{*}\sqcup \lambda_4')}(
    Z^{\bf k}\otimes_{\mathbb C}
    Z^{{\bf b}'(p_4, i)})
    \quad(\mbox{here $Z^{\bf k}\in \mathcal Z_{\hat\omega}^{\rm bl}(\fS^{*},\lambda^{*})$})\\
    =& 
     Z^{\bf k}\otimes_{\mathbb C}
     \Theta_{\lambda_4\lambda_4'}\left(
     \prod_{t=n-i+1}^{n} {\rm tr}_{\lambda_4'}(C_t)\right)
     \quad(\because \mbox{\eqref{eq-Zb-tC}})\\
     =&Z^{\bf k}\otimes_{\mathbb C}
     \prod_{t=n-i+1}^{n} {\rm tr}_{\lambda_4}(C_t)
     \quad(\because \mbox{Theorem \ref{thm-naturality}})\\
     =& 
    Z^{\bf k}\otimes_{\mathbb C}
    Z^{{\bf b}(p_4, i)} \quad(\because \mbox{\eqref{eq-Zb-tC}})\\
    =& \mathcal S_E(Z^{{\bf b}(\lambda,p,i)}).
\end{align*}
Diagram \eqref{eq-commmm-key} shows that 
$$
\mathcal S_E(\Theta_{\lambda\lambda'}(Z^{{\bf b}(\lambda',p,i)}))
=\Theta_{
(\lambda^{*}\sqcup \lambda_4) (\lambda^{*}\sqcup \lambda_4')}(\mathcal S_E'(Z^{{\bf b}(\lambda',p,i)}))
=\mathcal S_E(Z^{{\bf b}(\lambda,p,i)}).
$$
This completes the proof since $\mathcal S_E$ is injective. 
\end{proof}

Suppose 
that $\mathcal{D} = (\Gamma,(X_v)_{v\in \VV})$ is an $\mathcal X$-seed whose exchange matrix is $Q$.
Under the assumption \Rlabel{2}, 
 there is an algebra embedding 
$$
\PT\colon \mathcal Z_{\hat \eta}(\mathcal D)\rightarrow \mathcal Z_{\hat \omega}(\mathcal D),
$$
defined by 
$$
\PT (Z_v) = Z_v^{N} \quad\mbox{for all}\quad v\in V.
$$
It is easy to show that $\PT$ induces an algebra embedding
\begin{align}\label{eq-fro-mbl}
\PT\colon \mathcal Z_{\hat \eta}^{\rm mbl}(\mathcal D)\rightarrow \mathcal Z_{\hat \omega}^{\rm mbl}(\mathcal D),
\end{align}
where $\mathcal Z_{\hat \eta}^{\rm mbl}(\mathcal D)$ and $\mathcal Z_{\hat \omega}^{\rm mbl}(\mathcal D)$ are defined in Definition \ref{def.mbl}.
Since $\PT$ is injective,
 $\PT$ induces a skew-field embedding
\begin{align*}
    \PT\colon \Fr(\mathcal Z_{\hat \eta}^{\rm mbl}(\mathcal D))\rightarrow
    \Fr(\mathcal Z_{\hat \omega}^{\rm mbl}(\mathcal D)).
\end{align*}

Let $k$ be a mutable vertex, and $\mathcal D'=\mu_k(\mathcal D)$.
As defined in \eqref{eq-def-nu}, there are isomorphisms
$$
\nu_k^{\hat\eta}\colon\Fr(
\mathcal{Z}^{\rm mbl}_{\hat{\eta}}(\mathcal D'))\rightarrow
\Fr(
\mathcal{Z}^{\rm mbl}_{\hat{\eta}}(\mathcal D)) \text{ and }
\nu_k^{\hat\omega}\colon\Fr(
\mathcal{Z}^{\rm mbl}_{\hat{\omega}}(\mathcal D'))\rightarrow
\Fr(
\mathcal{Z}^{\rm mbl}_{\hat{\omega}}(\mathcal D)).$$
Then the following holds.

\def\Xbe{\mathcal{Z}_{\heta}^{\rm bl}(\fS,\lambda)}
\def\Xbz{\mathcal{Z}_{\hat{\omega}}^{\rm bl}(\fS,\lambda)}

\begin{lemma}
    The following diagram commutes
    \begin{equation}\label{eq-commmm-Fro}
\begin{tikzcd}
\Fr(\mathcal Z_{\hat \eta}^{\rm mbl}(\mathcal D')) \arrow[r, "\PT"]
\arrow[d, "\nu_k^{\hat\eta}"]  
& \Fr(\mathcal Z_{\hat \omega}^{\rm mbl}(\mathcal D')) \arrow[d, "\nu_k^{\hat\omega}"] \\
\Fr(\mathcal Z_{\hat \eta}^{\rm mbl}(\mathcal D))
 \arrow[r, "\PT"] 
&  \Fr(\mathcal Z_{\hat \omega}^{\rm mbl}(\mathcal D))
\end{tikzcd}
\end{equation}
\end{lemma}
\begin{proof}
From \eqref{eq-def-nu}, it suffices to show that the following two diagrams commute:
\begin{equation}\label{eq-diag-vF}
\begin{tikzcd}
\Fr(\mathcal Z_{\hat \eta}^{\rm mbl}(\mathcal D)) \arrow[r, "\PT"]
\arrow[d, "\nu_k^{\sharp\hat \eta}"]  
& \Fr(\mathcal Z_{\hat \omega}^{\rm mbl}(\mathcal D)) \arrow[d, "\nu_k^{\sharp\hat \omega}"] \\
\Fr(\mathcal Z_{\hat \eta}^{\rm mbl}(\mathcal D)) \arrow[r, "\PT"] 
& \Fr(\mathcal Z_{\hat \omega}^{\rm mbl}(\mathcal D))
\end{tikzcd}
\end{equation}
and
\begin{equation}\label{eq-diag-vF1}
\begin{tikzcd}
\Fr(\mathcal Z_{\hat \eta}(\mathcal D')) \arrow[r, "\PT"]
\arrow[d, "\nu_k'"]  
& \Fr(\mathcal Z_{\hat \omega}(\mathcal D')) \arrow[d, "\nu_k'"] \\
\Fr(\mathcal Z_{\hat \eta}(\mathcal D)) \arrow[r, "\PT"] 
& \Fr(\mathcal Z_{\hat \omega}(\mathcal D))
\end{tikzcd}.
\end{equation}

For any ${\bf t}=(t_v)_{v\in V}\in\mathcal B_{\mathcal D}$ (see Definition~\ref{def.mbl}), suppose that
\[
\sum_{v\in \VV} Q(u,v){\bf t}_v=mn.
\]
Then
\[
v_k^{\sharp\hat \eta}(Z^{\bf t}) = Z^{\bf t} F^\eta(X_k, m)
= Z^{\bf t} \prod_{r=1}^{|m|} \bigl(1+\eta^{(2r-1){\rm sgn}(m)} X_k \bigr)^{{\rm sgn}(m)}
\in \mathcal Z_{\hat \eta}^{\rm mbl}(\mathcal D),
\]
where $X_k=Z_k^n \in \mathcal X_{\eta}(\mathcal D)$.  
By the definition of $\PT$, we obtain
\begin{align}\label{eq-Fveta}
    \PT(v_k^{\sharp\hat \eta}(Z^{\bf t}))
    = Z^{N{\bf t}} \prod_{r=1}^{|m|} \bigl(1+\eta^{(2r-1){\rm sgn}(m)} X_k^N \bigr)^{{\rm sgn}(m)}
    \in \mathcal Z_{\hat \omega}^{\rm mbl}(\mathcal D).
\end{align}

Moreover,
\[
\sum_{v\in \VV} Q(u,v)N{\bf t}_v= mNn.
\]
Thus,
\begin{align}\label{eq-vomeF}
    v_k^{\sharp\hat \omega}(\PT(Z^{\bf t}))
    = v_k^{\sharp\hat \omega}(Z^{N{\bf t}})
    = Z^{N{\bf t}} \prod_{r=1}^{N|m|} \bigl(1+\omega^{(2r-1){\rm sgn}(m)} X_k \bigr)^{{\rm sgn}(m)}
    \in \mathcal Z_{\hat \omega}^{\rm mbl}(\mathcal D).
\end{align}
Recall that the order of $\omega^2$ is $N$ and $\eta=\omega^{N^2}$.  
It is well known that
\begin{align}\label{eq-quantum-equation}
    \prod_{r=1}^{|m|} \bigl(1+\eta^{(2r-1){\rm sgn}(m)} X_k^N \bigr)^{{\rm sgn}(m)}
    = \prod_{r=1}^{N|m|} \bigl(1+\omega^{(2r-1){\rm sgn}(m)} X_k \bigr)^{{\rm sgn}(m)}.
\end{align}
Equations~\eqref{eq-Fveta}, \eqref{eq-vomeF}, and~\eqref{eq-quantum-equation} imply that the diagram in~\eqref{eq-diag-vF} commutes.  

To prove the commutativity of the diagram in~\eqref{eq-diag-vF1}, it suffices to show that
\[
    \PT(\nu_k'(Z_v')) = \nu_k'(\PT(Z_v'))
\]
for every $v\in \VV$.  

If $v\neq k$, then
\begin{align*}
    \PT(\nu_k'(Z_v')) &= \PT([Z_v Z_k^{[Q(v,k)]_+}]_{\hat\eta})
    = [Z_v^N Z_k^{N[Q(v,k)]_+}]_{\hat\omega}, \\
    \nu_k'(\PT(Z_v')) &= \nu_k'((Z_v')^N)
    = (\nu_k'(Z_v'))^N
    = ([Z_v Z_k^{[Q(v,k)]_+}]_{\hat\omega})^N
    = [Z_v^N Z_k^{N[Q(v,k)]_+}]_{\hat\omega}.
\end{align*}
If $v=k$, then
\[
    \PT(\nu_k'(Z_v')) = \PT(Z_k^{-1}) = Z_k^{-N},
    \qquad
    \nu_k'(\PT(Z_v')) = \nu_k'((Z_v')^N)
    = (\nu_k'(Z_v'))^N = Z_k^{-N}.
\]
Hence the diagram in~\eqref{eq-diag-vF1} commutes as well.  
This completes the proof.
\end{proof}

\def\hom{\text{Hom}_{\text{Alg}}}

\subsection{The naturality of the constructed representations}\label{sub-sec-naturality}
Let $\mathcal{D} = (\Gamma,(X_v)_{v\in \VV})$ be an $\mathcal X$-seed,  let $k$ be a mutable vertex. Suppose that $\mu_k(\mathcal D)=
\mathcal{D}' = (\Gamma',(X_v')_{v\in \VV})$.
Let $\eta=\pm 1$.
From Definition \ref{def.quantum_X-mutation}, we have $$ \mu_k^\eta(X_v')=\begin{cases}
    X_k^{-1}  & v=k\\
    X_v X_k^{[Q(v,k)]_{+}} (1+\eta X_k)^{-Q(v,k)} & v\neq k.
\end{cases}
$$
Let $f\in\hom(\mathcal X_\eta(\mathcal D),\mathbb C)$.
Define $f'=\mu_k^\eta(f)\in \hom(\mathcal X_\eta(\mathcal D'),\mathbb C)$ with
\begin{align}\label{eq-def-mu-number}
    f'(X_v')= \begin{cases}
    f(X_v)^{-1} & v=k,\\
    f(X_v) f(X_k)^{[Q(v,k)]_{+}} (1+\eta f(X_k))^{-Q(v,k)} & v\neq k.
\end{cases}
\end{align}
To make equation \eqref{eq-def-mu-number} well-defined, we require $f(X_k)\neq -\eta$.
We call that $f$ is {\bf compatible} with the mutation $\mu_k^\eta$ if $f(X_k)\neq -\eta$.


\begin{definition}\label{def-compatible}
    Let $\mathcal{D} = (\Gamma,(X_v)_{v\in \VV})$ be an $\mathcal X$-seed,  let $k_1,\cdots,k_m$ be mutable vertices, and let
   $f\in\hom(\mathcal X_\eta(\mathcal D),\mathbb C)$.
    We call that $f$ is {\bf compatible} with the mutation sequence $(\mu_{k_1}^\eta,\cdots,\mu_{k_m}^\eta)$ if 
    $\mu_{k_i}^\eta\cdots \mu_{k_1}^\eta(f))$ is compatible with the mutation $\mu_{k_{i+1}}^\eta$ for each $0\leq i\leq n-1$.
    Here $\mu_{k_0}^\eta(f) = f$.
\end{definition}


\def\sss{\cS_{\bar\omega}(\fS)}
\def\see{\cS_{\bar\eta}(\fS)}
\def\zzm{\mathcal Z_{\hat\omega}^{\rm mbl}}
\def\ZEE{\mathcal Z_{\hat\eta}^{\rm mbl}}
\def\trl{{\rm tr}_{\lambda}}
\def\trp{{\rm tr}_{\lambda'}}
\def\TR{{\rm tr}}

Under the assumption \Rlabel{2}, we have $\hat\eta^{2n}=1$ and $\hat\omega^{2nN}=1$.  
Consequently, the algebra $\ZEE(\mathcal D)$ is commutative, and
\[
\PT\!\left(\mathcal Z_{\hat \eta}^{\rm mbl}(\mathcal D)\right)
   \;\subset\; \mathsf{Z}\!\left(\mathcal Z_{\hat \omega}^{\rm mbl}(\mathcal D)\right),
\]
where $\PT$ is the homomorphism defined in \eqref{eq-fro-mbl}.

The following theorem is our last  main result. 
It asserts that, under mild conditions, the irreducible representation of the projected $\SL$-skein algebra established in Theorem~\ref{thm-main-3} is independent of the choice of ideal triangulation~$\lambda$.

\begin{theorem}\label{thm-naturality}
Let $\fS$ be a triangulable punctured surface, and let $\lambda$ and $\lambda'$ be two triangulations of $\fS$ such that $\lambda'$ is obtained from $\lambda$ by a flip. 
Recall that there exists a sequence of $\mathcal X$-mutations $\mu_{v_r}, \dots, \mu_{v_2}, \mu_{v_1}$ such that
(see \eqref{two_seeds_connected_by_sequence_of_mutations})
\[
\mathcal{D}_{\lambda'} \;=\; \mu_{v_r} \cdots \mu_{v_2} \mu_{v_1}(\mathcal{D}_\lambda).
\]
Suppose that we have the assumption \Rlabel{2}. 
Let 
$$
\text{$h\in \hom(\mathsf{Z}(\zmm),\BC)$ (see Definition~\ref{def.mbl} for $\mathcal Z^{\rm mbl}$),}
$$
and set 
$$g = h\circ \PT,$$ where $\mathsf{Z}(\zmm)$ is the center of $\zmm$ and $\PT\colon \zem\to\mathsf{Z}(\zmm)$ is defined in \eqref{eq-fro-mbl}.  
Suppose that $g|_{\Xee}$ is compatible with the mutation sequence $(\mu_{k_1}^\eta, \dots, \mu_{k_m}^\eta)$. Then:

\begin{enumerate}[label={\rm (\alph*)}]
    \item For any $X\in \zemp$, one can write 
    \[
    \Theta_{\lambda\lambda'}(X) = PQ^{-1}\in \Fr(\zem),
    \]
    with $P,Q\in \zem$ and $g(Q)\neq 0$.  
    Define $g'(X):=g(P)g(Q)^{-1}$. Then $g'$ is a well-defined algebra homomorphism $\zemp \to \BC$. Moreover, the following diagram commutes:
    \[
    \begin{tikzcd}
    \cS_{\bar \eta}(\fS) \arrow[r, "{\rm tr}_\lambda"]
    \arrow[d, "{\rm tr}_{\lambda'}"]  
    & \zem \arrow[d, "g"] \\
    \zemp
     \arrow[r, "g'"] 
    &  \BC
    \end{tikzcd}.
    \]

    \item Let $\mathcal P$ be the set of punctures of $\fS$.
    For each $p\in \mathcal P$ and $1\leq i\leq n-1$, we have
    \[
    g\!\left(Z^{{\bf b}(\lambda,p,i)}\right) \;=\; g'\!\left(Z^{{\bf b}(\lambda',p,i)}\right).
    \]

    \item Define $b(p,i) := h\!\left(Z^{{\bf b}(\lambda,p,i)}\right)$ for $1\leq i\leq n-1$ and $p\in\mathcal P$. Then
    \[
    b(p,i)^N \;=\; g\!\left(Z^{{\bf b}(\lambda,p,i)}\right) \;=\; g'\!\left(Z^{{\bf b}(\lambda',p,i)}\right).
    \]
    Let $f = g|_{\Zee}$ and $f' = g'|_{\Zep}$.  
    By Theorem~\ref{thm-representation-Fock}, the data $f$ (resp. $f'$) together with $(b(p,i))_{1\leq i\leq n-1, \, p\in\mathcal P}$ uniquely determine an irreducible representation
    \[
    \rho\colon \Zoo \to \End(V) \quad \text{(resp. $\rho'\colon \Zop \to \End(V')$)}.
    \]
    As representations of $\cS_{\hat\omega}(\fS)$, the following two are isomorphic:
    \[
    \cS_{\bar\omega}(\fS)\xrightarrow{\tr}\Zoo\xrightarrow{\rho}\End(V), 
    \qquad
    \cS_{\bar\omega}(\fS)\xrightarrow{{\rm tr}_{\lambda'}}\Zop\xrightarrow{\rho'}\End(V').
    \]
\end{enumerate}
\end{theorem}

Before we prove the above theorem, we introduce some notation and results. 
For each $1\leq i\leq n$, define 
$$\mathcal{D}_i = \mu_{v_i} \cdots \mu_{v_2} \mu_{v_1} (\mathcal{D}_\lambda).$$
Note that $\mathcal D_r=\mathcal D_{\lambda'}$. We set $\mathcal D_0=\mathcal D_{\lambda}$.
For each $0\leq i\leq n$, define 
\begin{align}\label{def-tr-i}
    {\rm tr}_i:=\nu_{v_{i+1}}^{\hat\omega}\circ\cdots \circ\nu_{v_r}^{\hat\omega}\circ {\rm tr}_{\lambda'}\colon \sss\rightarrow \Fr(\zzm(\mathcal D_i)).
\end{align}
Then $\TR_0=\trl,\TR_r=\trp$. We have the following:
\begin{lemma}\cite[Equation~(139)]{WH2025}
 For each $0\leq j\leq r$, we have $\im \TR_j\subset\ZEE(\mathcal D_j)$.
\end{lemma}

\begin{lemma}\label{lem-img-v-b}
    For $0\leq j\leq r$, $p\in\mathcal P$, and $1\leq i\leq n-1$, we have $\nu_{v_{j+1}}^{\hat\omega}\cdots \nu_{v_r}^{\hat\omega}(Z^{{\bf b}(\lambda',p,i)}) = Z^{{\bf b}(j,p,i)}\in \zzm(\mathcal D_j)$ for some ${\bf b}(j,p,i)\in \mathcal B_{\mathcal D_j}$. In particular ${\bf b}(0,p,i)= {\bf b}(\lambda,p,i)$ and ${\bf b}(r,p,i)= {\bf b}(\lambda',p,i)$.
\end{lemma}
\begin{proof}
Suppose that $\mu_{v_r}(Q_{\lambda'}) = Q_{r-1}$. 
From the definition of $\nu_{v_r}'$, we know that 
$\nu_{v_r}'(Z^{{\bf b}(\lambda',p,i)}) = Z^{{\bf b}(r-1,p,i)}\in \zzm(\mathcal D_{r-1})$ for some ${\bf b}(r-1,p,i)\in\mathcal B_{\mathcal D_{r-1}}$.
From Proposition \ref{Prop-central-generic}, we know that ${\bf b}(\lambda',p,i) Q_{\lambda'}={\bf 0}$. We have ${\bf b}(r-1,p,i) Q_{r-1}={\bf 0}$ because $\nu_{v_r}'$ is a well-defined skew-field isomorphism. Thus $\nu_{v_r}^{\sharp\hat\omega}(Z^{{\bf b}(r-1,p,i)})= Z^{{\bf b}(r-1,p,i)}$ and 
$\nu_{v_r}^{\hat\omega}(Z^{{\bf b}(\lambda',p,i)})= Z^{{\bf b}(r-1,p,i)}$.

Using a similar argument as above, we can show that  
$\nu_{v_{r-1}}^{\hat\omega}\big( \nu_{v_r}^{\hat\omega}(Z^{{\bf b}(\lambda',p,i)}) \big) 
= Z^{{\bf b}(r-2,p,i)} \in \zzm(\mathcal D_{r-2})$  
for some ${\bf b}(r-2,p,i) \in \mathcal B_{\mathcal D_{r-2}}$.  
Proceeding inductively, the lemma follows.
\end{proof}

For each $0\leq i\leq n$, define 
\begin{equation}\label{eq-nu-i-def}
    \nu_i^{\hat\omega}:=
\nu_{v_{0}}^{\hat\omega}\circ\nu_{v_{1}}^{\hat\omega}\circ\cdots \circ\nu_{v_i}^{\hat\omega}\colon \Fr(\zzm(\mathcal D_i))\rightarrow\Fr(\zmm),
\end{equation}
where $\nu_{v_{0}}^{\hat\omega}$ is the identity map.
\begin{lemma}\label{lem-com-v-tr}
    The following diagram commutes:
    \begin{equation*}
\begin{tikzcd}
\cS_{\hat \omega}(\fS) \arrow[r,"{\rm tr}_i"]
\arrow[d, "\tr"]  
&  \Fr(\zzm(\mathcal D_i))  \arrow[dl, "\nu_i^{\hat\omega}"] \\
 \Fr(\mathcal Z_{\hat \omega}^{\rm bl}(\fS,\lambda))
&  \\
\end{tikzcd}
\end{equation*}
\end{lemma}
\begin{proof}
    It follows from Theorem \ref{thm-naturality-trace}(d), equations \eqref{eq-Theta2}, \eqref{def-tr-i}, and \eqref{eq-nu-i-def}.
\end{proof}

\begin{proof}[Proof of Theorem \ref{thm-naturality}]
    (a) 
We will use induction on $0\leq i\leq n$ to prove the following statement.

{\bf Statement (*)}: For any element $X\in \ZEE(\mathcal D_i)$, we can write
    $\nu_i^{\hat\eta}(X) = PQ^{-1}$,
    where $P,Q\in\zem$, such that $g(Q)\neq 0$.
    Here $\nu_i^{\hat\eta}$ is defined as in \eqref{eq-nu-i-def} with $\hat\omega$ replaced by $\hat\eta$.
    Then $g_i(X):=g(P)g(Q)^{-1}$ is a well-defined 
    algebra homomorphism from $\ZEE(\mathcal D_i)$
    to $\BC$ and 
    \begin{align*}
        g_{i}|_{\mathcal X_{\eta}(\mathcal D_{i})}=
\mu_{v_{i}}^\eta\cdots \mu_{v_1}^{\eta} \mu_{v_0}^{\eta}(g|_{\Xee}),
    \end{align*}
where $\mu_{v_0}^{\eta}(g|_{\Xee})= g|_{\Xee}$.
Furthermore, the following diagram commutes:
      \begin{equation}\label{diag-tr-i-g}
\begin{tikzcd}
\cS_{\bar \eta}(\fS) \arrow[r, "{\rm tr}_\lambda"]
\arrow[d, "{\rm tr}_{i}"]  
& \zem \arrow[d, "g"] \\
\ZEE(\mathcal D_i)
 \arrow[r, "g_i"] 
&  \BC
\end{tikzcd}.
\end{equation}

Statement (*) is obviously true when $i=0$.
Assume that Statement (*) holds for $i-1$ ($i\geq 1$). 
Let $X\in \ZEE(\mathcal D_i)$.
From the definition of $\nu_{v_i}^\eta$, we can write $\nu_{v_i}^\eta(X) = UV^{-1}$, where $U,V\in\ZEE(\mathcal D_i)$ and $V$ is the product in $1 + \eta^{2r - 1} X_{v_{i}}= 1 + \eta X_{v_{i}}$. Since $g_{i-1}|_{\mathcal X_{\eta}(\mathcal D_{i-1})}$ is compatible with the mutation $\mu_{v_i}^\eta$, we have $g_{i-1}(X_{v_i})\neq -\eta$. This shows that 
$g_{i-1}(V)\neq 0$. 

From the assumption step, we can write 
$\nu_{i-1}^\eta(U) = U_1U_2^{-1}$ and 
$\nu_{i-1}^\eta(V) = V_1V_2^{-1}$, where 
$U_1,U_2,V_1,V_2\in\zem$ such that 
$$\text{$g(U_2)g(V_2)\neq 0$,  
$g_{i-1}(U) = g(U_1)g(U_2)^{-1}$, and $g_{i-1}(V) = g(V_1)g(V_2)^{-1}$.}$$
Then $g(V_1)\neq 0$ because $g_{i-1}(V)\neq 0$.
Then $$\nu_i^\eta(X) = \nu_{i-1}^\eta(\nu_{v_i}^\eta(X))=
\nu_{i-1}^\eta (U) \nu_{i-1}^\eta(V)^{-1}
=U_1V_2(U_2V_1)^{-1}.$$
Here $g(U_2V_1)\neq 0$. 

Set 
$$g_i(X)= g(U_1V_2) g(U_2 V_1)^{-1}$$
It is a trivial check that $g_i$ is well-defined algebra homomorphism from $\ZEE(\mathcal D_i)$ to $\mathbb C$ and 
\begin{align}\label{eq-i-i1-qqq}
    g_i(X) = g_{i-1}(U) g_{i-1}(V)^{-1}.
\end{align}
Since $\nu_{v_i}^{\eta}(X)=UV^{-1}$, 
Lemma \ref{lem:nu_extends_mu} and \eqref{eq-i-i1-qqq} show that 
$$g_i|_{\mathcal X_{\eta}(\mathcal D_{i})}
=\mu_{v_i}^\eta (g_{i-1}|_{\mathcal X_{\eta}(\mathcal D_{i-1})}) = 
\mu_{v_{i}}^\eta\cdots \mu_{v_1}^{\eta} \mu_{v_0}^{\eta}(g|_{\Xee}).$$

Next we want to show that the diagram in \eqref{diag-tr-i-g} commutes. Let $\alpha\in \see$.
We can write $\nu_i^{\hat\eta}(\TR_i(\alpha)) = ZW^{-1}$, where 
$Z,W\in\zem$ such that $g(W)\neq 0$. 
Then $g_i(\TR_i(\alpha)) = g(Z)g(W)^{-1}$.
Lemma \ref{lem-com-v-tr} shows that $\tr(\alpha) = ZW^{-1}$.
Thus $g(\tr(\alpha)) g(W) = g(Z)$.
This implies that $g(\tr(\alpha)) =g(Z)g(W)^{-1} = g_i(\TR_i(\alpha))$.

Then Statement (*) holds for $0\leq i\leq r$.
We complete the proof of (a) by putting $i=r$ in Statement (*).

(b) follows from (a) and Proposition \ref{prop-loop-nature}.

(c) For $1\leq i\leq n-1$ and $p\in\mathcal P$, we have $$g(Z^{{\bf b}(\lambda,p,i)}) = h(\PT(Z^{{\bf b}(\lambda,p,i)})) =h (Z^{N{\bf b}(\lambda,p,i)})
=h(Z^{{\bf b}(\lambda,p,i)})^N = b(p.i)^N.$$

Lemma~\ref{lem-irre-Azumaya} and
Corollary~\ref{prop-representation-Fock}  imply that there exists an irreducible representation of $\zmm$
$$\bar\rho\colon\zmm\rightarrow\End(V)$$
such that $\bar\rho(X) = h(X)\Id_V$, for any $X\in\mathsf{Z}(\zzm)$, and $\dim_\BC V= N^{(n^2-1)(g-1)+\frac{1}{2}n(n-1)|\mathcal P|}$.
Define 
$$\rho\colon \Zoo\rightarrow\zmm\xrightarrow{\bar\rho}\End(V).$$
Then $\rho$ is an irreducible representation of $\Zoo$ because 
$(\dim_\BC V)^2$ equals the rank of $\Zoo$ over $\mathsf{Z}(\Zoo)$ (Theorem~\ref{thm-rank-Z}).
It is a trivial check that 
$$\rho(\PT(W)) = f(W)\Id_V \text{ and }
\rho(Z^{{\bf b}(\lambda,p,i)}) = b(p,i)\Id_V$$
for $W\in \Zee$, $p\in\mathcal P$, and $1\leq i\leq n-1$.

To complete the proof, we need the following lemma. We will prove the following lemma later.
\begin{lemma}\label{lem-key-final}
    Let $\bar \rho_0=\bar\rho$, and let $1\leq j\leq r$. Suppose that we have a representation
    $$\rho_{j-1}\colon \zzm(\mathcal D_{j-1})\rightarrow \End(V)$$
    satisfying the following conditions:
    \begin{enumerate}
        \item $\bar \rho_{j-1}(\PT(Y)) = g_{i-1}(Y)\Id_V$, for $Y\in\ZEE(\mathcal D_{i-1})$, where $g_{j-1}$ is defined in the proof of (a).

        \item $\bar \rho_{j-1}(Z^{{\bf b}(j-1,p,i)}) = \bar\rho_0(Z^{{\bf b}(\lambda,p,i)})$ for $p\in\mathcal P$ and $1\leq i\leq n-1$, where $Z^{{\bf b}(j-1,p,i)}$ is defined in Lemma \ref{lem-img-v-b}.

        \item The following diagram commutes:
        \begin{equation*}
\begin{tikzcd}
\cS_{\bar \omega}(\fS) \arrow[r, "{\rm tr}_\lambda"]
\arrow[d, "{\rm tr}_{j-1}"]  
& \zmm \arrow[d, "\bar\rho_0"] \\
\zzm(\mathcal D_{j-1})
 \arrow[r, "\bar\rho_{j-1}"] 
&  \End(V)
\end{tikzcd}.
\end{equation*}
    \end{enumerate}
  Then, for any element $x\in \zzm(\mathcal D_j)$, we can write 
  $\nu_{v_j}^{\hat\omega}(x) = uv^{-1}=w^{-1}z\in \Fr(\zzm(\mathcal D_{j-1}))$, where $u,v,w,z\in \zzm(\mathcal D_{j-1})$ such that 
  $\bar\rho_{j-1}(v)$ and $\bar\rho_{j-1}(w)$ are invertible. 
  Put $\bar\rho_j(x)=\bar\rho_{j-1}(u) \bar\rho_{j-1}(v)^{-1}
  =\bar\rho_{j-1}(w)^{-1}\bar\rho_{j-1}(z).$
 We have that $\rho_{j}\colon \zzm(\mathcal D_{j})\rightarrow \End(V)$ is a well-defined representation satisfying Conditions~(1), (2), and (3) above, with $j-1$ replaced by $j$.

\end{lemma}

Note that $\bar \rho_{j-1}$ satisfies Conditions (1), (2), and (3) in Lemma~\ref{lem-key-final} when $j=1$. 
After repeatedly applying Lemma~\ref{lem-key-final}  $r$ times, we can get a representation $\bar \rho_r\colon \zmp\rightarrow\End(V)$
satisfying Conditions (1), (2), and (3) in Lemma \ref{lem-key-final} with $j-1$ replaced by $r$.
Define $V'=V$ and 
$$\rho'\colon \Zop\rightarrow \zmp\xrightarrow{\bar\rho_r}
\End(V').$$
For any $K\in \Zep$, we have 
$$\rho'(\PT(K)) = \bar\rho_r(\PT(K)) = g_r(K)\Id_{V'} = g'(K)\Id_{V'}= f'(K)\Id_{V'},$$
where $g'=g_r$ and $f'=g'|_{\Zep}$.
For any $p\in\mathcal P$ and $1\leq i\leq n-1$, we have 
$$\rho'(Z^{{\bf b}(\lambda',p,i)})
=\bar\rho_r(Z^{{\bf b}(r,p,i)})=\bar \rho (Z^{{\bf b}(\lambda',p,i)}) = b(p,i).$$
Since $\dim_\BC(V')$ equals the rank of $\Zep$ over $\mathsf{Z}(\Zep)$, the representation $\rho'$ is irreducible. 

As representations of $\sss$, we have that $\rho'\circ\TR_{\lambda'}$ is isomorphic to 
$\rho\circ\TR_{\lambda}$ because $\rho'\circ\TR_{\lambda'}=\rho\circ\TR_{\lambda}$.

\end{proof}

Before proving Lemma \ref{lem-key-final}, we first note the fact that 
$(-\omega)^N = -\eta$.
Recall that $\eta=\omega^{N^2}$ and the order of $\omega^2$ is $N$.

{\bf Case (1)} $N$ is even: Then $\eta=(\omega^{N})^N=(\pm 1)^N=1$. We have $\omega^N=-1$. Otherwise $\omega^N=1$ and the order of $\omega^2$ is $\frac{1}{N}$, a contracition.
Thus $(-\omega)^N=\omega^N=-1=-\eta$.

{\bf Case (2)} $N$ is odd: Then $\eta=(\omega^{N})^N=\omega^N$. 
We have $(-\omega)^N=-\omega^N=-\eta$.

\begin{proof}[Proof of Lemma \ref{lem-key-final}]
    Let ${\bf t}\in\mathcal B_{\mathcal D_j}$.
    From the definition of $\nu_{v_j}^{\hat\omega}$,
    we have 
$$\nu_{v_j}^{\hat\omega}(Z^{\bf t}) = Z^{{\bf t}'}\prod_{r=1}^{|m|}(1+\omega^{(2r-1){{\rm sgn}(m)}} X_{v_j})^{{\rm sgn}(m)} =\left(\prod_{r=1}^{|m|}(1+\omega^{(2r-1-2m){{\rm sgn}(m)}} X_{v_j})^{{\rm sgn}(m)}\right)Z^{{\bf t}'},$$
where ${\bf t}'\in \mathcal B_{\mathcal D_{j-1}}$, and $m\in\mathbb Z$.
Thus, for $x\in \zzm(\mathcal D_j)$,
we have 
$\nu_{v_j}^{\hat\omega}(x) = uv^{-1} = w^{-1}z \in \Fr(\zzm(\mathcal D_{j-1}))$, 
where $u, v, w, z \in \zzm(\mathcal D_{j-1})$, and $v$ and $w$ are products in $(1 + \omega^{2r-1} X_{v_j})$.

Note that $g_{j-1}(X_{v_j})\neq -\eta$ because $g_{j-1}|_{\mathcal X_\eta(\mathcal D_{i-1})}$ is compatible with the mutation $\mu_{v_i}^\eta$.
We will show that $\Id_V + \omega^{2r-1} \bar\rho_{j-1}(X_{v_j})$ is invertible. 
Assume, to the contrary, that it is not. 
Then $-\omega^{-2r+1}$ is an eigenvalue of $\bar\rho_{j-1}(X_{v_j})$.
Thus $(-\omega^{-2r+1})^N=(-\omega)^N=-\eta$ is an eigenvalue of 
$\bar\rho_{j-1} (X_{v_j})^N = g_{j-1}(X_{v_j})\Id_V$.
This contradicts with $g_{j-1}(X_{v_j})\neq -\eta$.

From the above discussion, we know that $\bar\rho_{j-1}(v)$ and $\bar\rho_{j-1}(w)$ are invertible. 
Put  $\bar\rho_j (x)=\bar\rho_{j-1}(u) \bar\rho_{j-1}(v)^{-1}
  =\bar\rho_{j-1}(w)^{-1}\bar\rho_{j-1}(z)$.
It is a trivial check that 
$\rho_{j}\colon \zzm(\mathcal D_{j})\rightarrow \End(V)$ is a well-defined representation.

{\bf Condition (1)}: Let $Y\in \ZEE(\mathcal D_j)$.
From the proof of (a),
we know that $\nu_{v_j}^{\hat\eta}(Y) = Y_1 Y_2^{-1}
=Y_2^{-1} Y_1\in\Fr(\ZEE(\mathcal D_{j-1}))$, where $Y_1,Y_2\in \ZEE(\mathcal D_{j-1})$ such that $g_{j-1}(Y_2)\neq 0$.
Lemma \ref{eq-commmm-Fro} shows that 
$$\nu_{v_j}^{\hat\omega}(\PT(Y)) = \PT(\nu_{v_j}^{\hat\eta}(Y))
=\PT(Y_1)\PT(Y_2)^{-1}=\PT(Y_2)^{-1}\PT(Y_1)\in\Fr(\zzm(\mathcal D_{j-1})).$$
Here $\bar\rho_{j-1}(\PT(Y_2)) = g_{j-1}(Y_2)\Id_V$ is invertible because $g_{j-1}(Y_2)\neq 0$. Then we have 
$$\bar\rho_j(\PT(Y)) = \bar\rho_{j-1}(\PT(Y_1))\bar\rho_{j-1}(\PT(Y_2))^{-1}=g_{j-1}(Y_1) g_{j-1}(Y_2)^{-1}\Id_V=g_j(Y)\Id_V.$$

{\bf Condition (2)}: Lemma \ref{lem-img-v-b} implies that 
$\nu_{v_j}^{\hat\omega} (Z^{{\bf b}(j,p,i)}) = Z^{{\bf b}(j-1,p,i)}$. Then 
$$\bar\rho_j(Z^{{\bf b}(j,p,i)}) = \bar\rho_{j-1}(Z^{{\bf b}(j-1,p,i)})=\bar\rho_0(Z^{{\bf b}(\lambda,p,i)}).$$

{\bf Condition (3)}: Let $\beta\in\sss$. 
Suppose that $\nu_{v_j}^{\hat\omega}(\TR_j(\beta)) = u_1u_2^{-1}$, where $u_1,u_2\in \zzm(\mathcal D_{j-1})$ such that $\bar\rho_{j-1}(u_2)$ is invertible. 
Then 
$$\bar\rho_j(\TR_j(\beta)) = \bar\rho_{j-1}(u_1) \bar\rho_{j-1}(u_2)^{-1}.$$
It follows from Lemma \ref{lem-com-v-tr} that
$\TR_{j-1} = \nu_{v_j}^{\hat\omega}\circ \TR_j$.
Thus $\TR_{j-1}(\beta) = u_1 u_2^{-1}$. This shows that 
$$\text{$\bar\rho_{j-1}(\TR_{j-1}(\beta)) \bar\rho_{j-1}(u_2)=
\bar\rho_{j-1}(u_1)$ and $\bar\rho_{j-1}(\TR_{j-1}(\beta))=\bar\rho_{j-1}(u_1) \bar\rho_{j-1}(u_2)^{-1}.$}$$
We have
$$\bar\rho_j(\TR_j(\beta))=\bar\rho_{j-1}(\TR_{j-1}(\beta))=\bar\rho_0(\TR_{\lambda}(\beta)).$$

\end{proof}

\begin{remark}
When $n=2$, Theorem \ref{thm-naturality} was proved in \cite{representation3}.  
In this case, the quiver $\Gamma_{\lambda'}$ is obtained from $\Gamma_{\lambda}$ by a single mutation, and  
$\nu_{v_1}^{\omega}$ is defined from  
$\mathcal Z_{\hat\omega}^{\rm bl}(\fS,\lambda')$  
to $\mathcal Z_{\hat\omega}^{\rm bl}(\fS,\lambda)$.  
Consequently, the authors of \cite{representation3} did not introduce $\mathcal Z^{\rm mbl}$.  

When $n>2$, however, we do not know how to define the “balanced part” $\mathcal Z_{\hat\omega}^{\rm bl}(\mathcal D_j)$  
for $1 \leq j \leq r-1$.  
Therefore, in Theorem \ref{thm-naturality} for general $n$, we must instead involve $\mathcal Z^{\rm mbl}$.

\end{remark}

\begin{remark}
   Although Theorem~\ref{thm-naturality} concerns two triangulations related by a flip, it extends to arbitrary triangulations via Theorem~\ref{thm-naturality-trace}(b).  
We state Theorem~\ref{thm-naturality} only in the flip case, as it is more convenient to formulate in this setting.

\end{remark}

\bibliography{ref.bib}

\begin{thebibliography}{FKBL21}

\bibitem[BH24]{BH23}
Francis Bonahon and Vijay Higgins.
\newblock {Central elements in the ${\rm SL}_d$-skein algebra of a surface}.
\newblock {\em Mathematische Zeitschrift}, 308(1):1, 2024.

\bibitem[BL07]{BL07}
Francis Bonahon and Xiaobo Liu.
\newblock {Representations of the quantum Teichm{\"u}ller space and invariants
  of surface diffeomorphisms}.
\newblock {\em Geometry \& Topology}, 11(2):889--937, 2007.

\bibitem[BtD95]{BtD}
Theodor Br\"ocker and Tammo tom Dieck.
\newblock {\em Representations of compact {L}ie groups}, volume~98 of {\em
  Graduate Texts in Mathematics}.
\newblock Springer-Verlag, New York, 1995.
\newblock Translated from the German manuscript, Corrected reprint of the 1985
  translation.

\bibitem[BW11]{BW11}
Francis Bonahon and Helen Wong.
\newblock {Quantum traces for representations of surface groups in ${\rm
  SL}_2(\mathbb{C})$}.
\newblock {\em Geometry \& Topology}, 15(3):1569--1615, 2011.

\bibitem[BW16]{BW16}
Francis Bonahon and Helen Wong.
\newblock {Representations of the Kauffman bracket skein algebra I: invariants
  and miraculous cancellations}.
\newblock {\em Inventiones mathematicae}, 204:195--243, 2016.

\bibitem[BW17]{representation2}
Francis Bonahon and Helen Wong.
\newblock {Representations of the Kauffman bracket skein algebra, II: Punctured
  surfaces}.
\newblock {\em Algebraic \& geometric topology}, 17(6):3399--3434, 2017.

\bibitem[BW19]{representation3}
Francis Bonahon and Helen Wong.
\newblock {Representations of the Kauffman bracket skein algebra III: closed
  surfaces and naturality}.
\newblock {\em Quantum Topology}, 10(2):325--398, 2019.

\bibitem[BZ05]{BZ}
Arkady Berenstein and Andrei Zelevinsky.
\newblock {Quantum cluster algebras}.
\newblock {\em Advances in Mathematics}, 195(2):405--455, 2005.

\bibitem[CD24]{cremaschi2024monomial}
Tommaso Cremaschi and Daniel~C Douglas.
\newblock {Monomial web basis for the $SL(N)$ skein algebra of the twice
  punctured sphere}.
\newblock {\em arXiv preprint arXiv:2407.04178}, 2024.

\bibitem[Fad95]{F95}
LD~Faddeev.
\newblock {Discrete Heisenberg-Weyl group and modular group}.
\newblock {\em Letters in Mathematical Physics}, 34(3):249--254, 1995.

\bibitem[FG06]{FG06}
Vladimir Fock and Alexander Goncharov.
\newblock {Moduli spaces of local systems and higher Teichm{\"u}ller theory}.
\newblock {\em Publications Math{\'e}matiques de l'IH{\'E}S}, 103:1--211, 2006.

\bibitem[FG07]{FG09b}
Vladimir~V Fock and Alexander~B Goncharov.
\newblock {The quantum dilogarithm and representations quantum cluster
  varieties}.
\newblock {\em arXiv preprint math/0702397}, 2007.

\bibitem[FG09]{FG09a}
Vladimir~V Fock and Alexander~B Goncharov.
\newblock {Cluster ensembles, quantization and the dilogarithm}.
\newblock In {\em Annales scientifiques de l'{\'E}cole normale sup{\'e}rieure},
  volume~42, pages 865--930, 2009.

\bibitem[FK94]{FK94}
Ludwig~D Faddeev and Rinat~M Kashaev.
\newblock Quantum dilogarithm.
\newblock {\em Modern Physics Letters A}, 9(05):427--434, 1994.

\bibitem[FKBL19]{unicity}
Charles Frohman, Joanna Kania-Bartoszynska, and Thang L{\^e}.
\newblock {Unicity for representations of the Kauffman bracket skein algebra}.
\newblock {\em Inventiones mathematicae}, 215:609--650, 2019.

\bibitem[FKBL21]{frohman2021dimension}
Charles Frohman, Joanna Kania-Bartoszynska, and Thang L{\^e}.
\newblock {Dimension and trace of the Kauffman bracket skein algebra}.
\newblock {\em Transactions of the American Mathematical Society, Series B},
  8(18):510--547, 2021.

\bibitem[GS19]{GS19}
Alexander Goncharov and Linhui Shen.
\newblock {Quantum geometry of moduli spaces of local systems and
  representation theory}.
\newblock {\em arXiv preprint arXiv:1904.10491}, 2019.

\bibitem[Hig23]{higgins2020triangular}
Vijay Higgins.
\newblock {Triangular decomposition of ${\rm SL}_3$ skein algebras}.
\newblock {\em Quantum Topology}, 14(1):1--63, 2023.

\bibitem[Hig25]{higgins2025miraculous}
Vijay Higgins.
\newblock {Miraculous Cancellations and the Quantum Frobenius for $SL_3$ Skein
  Modules}.
\newblock {\em International Mathematics Research Notices}, 2025(15):rnaf226,
  2025.

\bibitem[HW25]{WH2025}
Min Huang and Zhihao Wang.
\newblock {Quantum cluster realization for projected stated ${\rm SL}_n$-skein
  algebras}.
\newblock {\em arXiv preprint arXiv:2509.25938}, 2025.

\bibitem[Kim20]{Kim20}
Hyun~Kyu Kim.
\newblock {${\rm SL}_3$-laminations as bases for ${\rm PGL}_3$ cluster
  varieties for surfaces}.
\newblock {\em arXiv preprint arXiv:2011.14765}, 2020.

\bibitem[KLW25]{KLW}
Hyun~Kyu Kim, Thang~TQ L{\^e}, and Zhihao Wang.
\newblock {Frobenius homomorphisms for stated ${\rm SL}_n$-skein modules}.
\newblock {\em arXiv preprint arXiv:2504.08657}, 2025.

\bibitem[KW24a]{KimWang}
Hyun~Kyu Kim and Zhihao Wang.
\newblock {Naturality of ${\rm SL}_n$ quantum trace maps for surfaces}.
\newblock {\em arXiv preprint arXiv:2412.16959}, 2024.

\bibitem[KW24b]{kim2024unicity}
Hyun~Kyu Kim and Zhihao Wang.
\newblock {The Unicity Theorem and the center of the ${\rm SL}_3$-skein
  algebra}.
\newblock {\em arXiv preprint arXiv:2407.16812}, 2024.

\bibitem[KW25]{KaruoWangToAppear}
Hiroaki Karuo and Zhihao Wang.
\newblock {Center of stated ${\rm SL}(n)$-skein algebras}.
\newblock {\em Transactions of the American Mathematical Society}, 2025.
\newblock to appear.

\bibitem[L{\^e}18]{le2018triangular}
Thang~TQ L{\^e}.
\newblock Triangular decomposition of skein algebras.
\newblock {\em Quantum Topology}, 9(3):591--632, 2018.

\bibitem[LS24]{LS21}
Thang~TQ L{\^e} and Adam~S Sikora.
\newblock {Stated SL(n)-skein modules and algebras}.
\newblock {\em Journal of Topology}, 17(3):e12350, 2024.

\bibitem[LY22]{LY22}
Thang~TQ L{\^e} and Tao Yu.
\newblock Quantum traces and embeddings of stated skein algebras into quantum
  tori.
\newblock {\em Selecta Mathematica}, 28(4):66, 2022.

\bibitem[LY23]{LY23}
Thang~TQ L{\^e} and Tao Yu.
\newblock {Quantum traces for $SL_n$-skein algebras}.
\newblock {\em arXiv preprint arXiv:2303.08082}, 2023.

\bibitem[Mac15]{Macdonald}
I.~G. Macdonald.
\newblock {\em Symmetric functions and {H}all polynomials}.
\newblock Oxford Classic Texts in the Physical Sciences. The Clarendon Press,
  Oxford University Press, New York, second edition, 2015.
\newblock With contribution by A. V. Zelevinsky and a foreword by Richard
  Stanley.

\bibitem[New72]{newman1972integral}
Morris Newman.
\newblock {\em Integral matrices}.
\newblock Academic Press, 1972.

\bibitem[QR18]{queffelec2018sutured}
Hoel Queffelec and David Rose.
\newblock {Sutured annular Khovanov-Rozansky homology}.
\newblock {\em Transactions of the American Mathematical Society},
  370(2):1285--1319, 2018.

\bibitem[Sik01]{S2001SLn}
Adam Sikora.
\newblock {$SL_n$-character varieties as spaces of graphs}.
\newblock {\em Transactions of the American Mathematical Society},
  353(7):2773--2804, 2001.

\bibitem[Sik05]{Sik05}
Adam~S Sikora.
\newblock {Skein theory for $SU(n)$-quantum invariants}.
\newblock {\em Algebraic \& Geometric Topology}, 5(3):865--897, 2005.

\bibitem[Wan24]{wang2024TQFT}
Zhihao Wang.
\newblock {Stated $SL_n$-skein modules, roots of unity, and TQFT}.
\newblock {\em arXiv preprint arXiv:2401.09995}, 2024.

\end{thebibliography}

\end{document}